\documentclass{article}
\usepackage[english]{babel}
\usepackage{amsmath,amssymb,amsthm,mathrsfs, accents}
\usepackage{mathtools}
\usepackage{url}
\usepackage{float}
\usepackage{graphicx, xcolor}

 \usepackage[colorlinks=true]{hyperref}
 \hypersetup{urlcolor=blue, citecolor=red}

\newtheorem{proposition}{Proposition}
 \newtheorem{cor}[proposition]{Corollary}
\newtheorem{assumption}{Assumption}
\usepackage{color, comment}
\newtheorem{remark}{Remark}
\newtheorem{theorem}{Theorem}
\newtheorem{lemma}{Lemma}

\numberwithin{equation}{section} 

\newcommand{\Rb}{\mathbb{R}}

\newcommand{\Cc}{\mathcal{C}}

\newcommand{\Ec}{\mathcal{E}}
\newcommand{\Fc}{\mathcal{F}}

\newcommand{\Oc}{\mathcal{O}}
\newcommand{\Mc}{\mathcal{M}}

\newcommand{\Qc}{\mathcal{Q}}
\newcommand{\Pc}{\mathcal{P}}

\newcommand{\Vc}{\mathcal{V}}

\newcommand{\pa}{\partial}

\newcommand{\Ls}{\mathscr{L}}
\newcommand{\Ns}{\mathscr{N}}
\newcommand{\Fs}{\mathscr{F}}

\newcommand{\Ds}{\mathscr{D}}

 \newcommand{\beq}{\begin{equation}}
 \newcommand{\eeq}{\end{equation}}

 \newcommand{\bseq}{\begin{subequations}}
 \newcommand{\eseq}{\end{subequations}}
 \newcommand{\bal}{\begin{aligned} }
 \newcommand{\eal}{\end{aligned}}

 \newcommand{\bali}{\begin{align}}
 \newcommand{\eali}{ \end{align}}
  \newcommand{\tr}{\mathrm{tr}}
   \newcommand{\TT}{\mathbf{T}}

 \newcommand{\RR}{\mathbf{R}}
 \def\cR{{\mathcal R}}
   \let \mfr = \mathfrak
\let\lam=\lambda
 \let\f=\frac
  \let\D=\Delta
  \let \al = \alpha
 \let \les = \lesssim
 \let \gtr = \gtrsim
 \let\Om=\Omega
 \def\na{\nabla}
 \def\la{\langle}
 \def\ra{\rangle}
 \def \mr{\mathring}
  \let\d=\delta
  \let\B = \Big
  \let \kp = \kappa
  \let\e=\epsilon
  \let\td = \tilde
  \let\Ga= \Gamma
  \let\im= \imath
  
  \let \be = \beta
  \let \th = \theta
  \let \Th = \Theta
  \let\Lam=\Lambda
  
  \def\cE{{\mathcal E}}
   \let\s=\sigma
   \def\cH{{\mathcal H}}

  \def\cP{{\mathcal P}}
 \def\one{\mathbf{1}}
 \def \msD{  \mathscr{D}}

\def \msN{ \mathscr{N} }
\def \msL{ \mathscr{L} }

\title{On the stability of blowup solutions to the complex Ginzburg-Landau equation in $\Rb^d$}
\date{\today}
\begin{document}
 \author{Jiajie Chen\footnote{Courant Institute of Mathematical Sciences, New York University, New York, NY. Email: \href{jiajie.chen@cims.nyu.edu}{jiajie.chen@cims.nyu.edu}}, 
 Thomas Y. Hou\footnote{Applied and Computational Mathematics, Caltech, Pasadena, CA. Email: \href{hou@cms.caltech.edu}{hou@cms.caltech.edu}}, 
 Van Tien Nguyen\footnote{Department of Mathematics, National Taiwan University, Taiwan. Email: \href{vtnguyen@ntu.edu.tw}{vtnguyen@ntu.edu.tw}}, 
 Yixuan Wang\footnote{Applied and Computational Mathematics, Caltech, Pasadena, CA. Email: \href{roywang@caltech.edu}{roywang@caltech.edu} }}

\maketitle

\begin{abstract}

{
  \noindent Building upon the idea in \cite{HNWarXiv24}, we establish stability of 
the type-I blowup with log correction for the complex Ginzburg-Landau equation. In the amplitude-phase representation, a generalized dynamic rescaling formulation is introduced, with modulation parameters capturing the spatial translation and rotation symmetries of the equation and novel additional modulation parameters perturbing the scaling symmetry. 
This new formulation provides enough degrees of freedom to impose normalization conditions on the rescaled solution, completely eliminating the unstable and neutrally stable modes of the linearized operator around the blowup profile. It enables us to establish the full stability of the blowup by enforcing vanishing conditions via the choice of normalization and using weighted energy estimates, without relying on a topological argument or a spectrum analysis. The log correction for the blowup rate is captured by the energy estimates and refined estimates of the modulation parameters.

}

    \end{abstract}
\section{Introduction}

We consider the complex Ginzburg-Landau equation 
\begin{equation}
    \label{cgl} 
    \psi_t=(1+\imath\beta)\Delta \psi + (1+\imath \delta)|\psi|^{p-1}\psi- \gamma \psi,  \tag{CGL}
\end{equation}
where $\psi(t): \Rb^d \to \mathbb{C}$, $\beta, \delta, \gamma$ are real constants and $p > 1$. The model equation \eqref{cgl} was first derived by Stewardson and Stuart in \cite{SSjfm71} (see also \cite{DHSjfm74}, \cite{PSbook85}) to examine afresh the problem of plane Poiseuille flow in a wave system. The equation is also used to describe various phenomena in many fields, among which are nonlinear optics with dissipation \cite{MHKnonl98}, turbulent behavior \cite{BLNphysd86}, Rayleigh-Bénard convection or Taylor-Couette flow in hydrodynamics \cite{diprima1971non}, \cite{NWjfm69}, \cite{newell1971review}, reaction-diffusion systems \cite{hagan1982spiral}, \cite{iooss1992time}, \cite{scheel1998bifurcation}, \cite{schneider1998hopf}, the theory of superconductivity \cite{bethuel1994ginzburg}, \cite{chapman1992macroscopic}, \cite{du1992analysis}, \cite{ginzburg2009theory}, etc. For further details on the physical background and derivation of the complex Ginzburg-Landau equation, we refer to the surveys \cite{AKaps02}, \cite{Mmono02}, and the references therein.

The local Cauchy problem has been well established through a semigroup approach in the works  \cite{GVphysd96,  GVmr97, GVcmp97}. A solution to \eqref{cgl} blows up in finite time if $\lim_{t \to T} \|\psi(t)\|_{L^\infty(\Rb^d)} = +\infty$ for some $ T < + \infty$.
Singularity formation has been intensively studied for the two limiting models of \eqref{cgl}: the classical nonlinear heat equation in the limit $\beta, \delta, \gamma \to 0$,
\begin{equation}\label{nlh}
\pa_t \psi = \Delta \psi + |\psi|^{p-1}\psi, \quad \psi(t): x\in \Rb^d \to \Rb, \tag{NLH}
\end{equation}
and the nonlinear Schr\"odinger equation in the limit $\beta, |\delta| \to \infty$,
\begin{equation}\label{nls}
\imath \pa_t \psi + \Delta \psi +\mu |\psi|^{p-1}\psi = 0, \quad \mu = \pm 1. \tag{NLS}
\end{equation}
{We refer to \cite{QSbook07} and \cite{Fbook15} for intensive
lists of references from the early 1960s concerning blowup results of these two equations.} However, singularities in \eqref{cgl} (collapse, chaotic, or blowup) are much less understood in comparison with what have been established for \eqref{nlh} and \eqref{nls}. The study of singularity in \eqref{cgl}  is a challenging problem due to the lack of variational structure, no maximum principle, non-self-adjoint linearized operator, etc.  Nevertheless, singularity in \eqref{cgl} was experimentally reported in \cite{KBSaps88}, \cite{KSALphysD95} where the authors described an extensive series of experiments on traveling-wave convection in an ethanol/water mixture, and collapse solutions were observed. We have a sharp sufficient criteria for collapse in \eqref{cgl} for the case of subcritical bifurcation described in \cite{Taps93}. In \cite{HDphysA98}, the authors used the modulation theory and numerical observations to show that the collapse dynamic is governed in the \eqref{cgl} limit of the $L^2$-critical cubic \eqref{nls}.  For the existence of blowup, there are the results of  \cite{CDFjee14} and \cite{CDFWsiam13} 
in which the authors studied \eqref{cgl} for the case $\beta = \delta$. In \cite{BRWejam13} and \cite{RScpam01}, the authors gave some evidence for the existence of a radial solution that blows up in a self-similar way, their arguments were based on the combination of rigorous analysis and numerical computations. In \cite{ZAAihn98} and \cite{MZjfa08}, the authors rigorously constructed particular examples of initial data for which the solutions of \eqref{cgl} blow up in finite time  for $(\beta, \delta)$ in the \textit{subcritical} range
\begin{equation}\label{def:subCri}
{\flat_* := p - \delta^2  - \beta \delta (p+1), \quad \flat_* > 0}  \quad \textup{(subcritical range)}.
\end{equation}
The constructed blowup solution in the subcritical case admits the asymptotic behavior 
\begin{equation}\label{eq:asyGLsub}
\psi(x,t) \sim  |\log (T-t)|^{\imath \mu} \Big[(T-t) \big( p -1   +  c_p |Z|^2\big) \Big]^{-\frac{1 + \imath \delta}{p-1}}, \quad Z = \frac{x}{\sqrt{(T-t)|\log (T-t)|}},
\end{equation}
where the constants $c_p$ and $\mu$ are given by
\begin{equation}
\label{def:constcp}c_p = \frac{(p-1)^2}{4\flat_*} > 0\,,\quad \mu = -\frac{\beta(1 + \delta^2)}{2\flat_*}\,.
\end{equation} The spectral analysis for a non-self-adjoint operator developed in \cite{MZjfa08} can be implemented for other problems where an energy-type method is not applicable, see for example \cite{GNZihp18}. The blowup for the critical range, $\flat_* = 0$, has been recently solved in \cite{NZarma18}, \cite{DNZmems23} following the approach of \cite{MZjfa08}. 
The blowup for \eqref{cgl} in the supercritical range, $\flat_* < 0$, has recently been solved in \cite{DNZarxiv23} for the special choice $\beta = 0$. 
We remark that in the mentioned works (\cite{ZAAihn98}, \cite{MZjfa08}, \cite{NZarma18}, \cite{DNZmems23}, \cite{DNZarxiv23}), the authors focused on the case of dimension $d=1$, and briefly described the stability properties of constructed blowup solutions through a spectral approach in a restricted (well-prepared) class of initial data.

In this paper, we aim to develop a new approach based on the dynamical rescaling formulation and simple vanishing conditions to study blowup solutions to \eqref{cgl}. This new approach allows us to establish asymptotically self-similar blowup and a clear notion of stability capturing the logarithm correction \eqref{eq:asyGLsub} 
in the subcritical case for a large class of initial data in all dimensions $d \geq 1$. Throughout this paper, we use the amplitude-phase representation,  
\begin{equation}\label{eq:intro_rep}
\psi(x,t) = u(x,t) e^{\imath \theta(x,t)},
\end{equation}
where $u$ and $\theta$ are real-valued functions of time and space solving the coupled system 
\bseq\label{eq:CGL_intro}
 \begin{align}
 &\pa_t u = \big[\Delta - |\nabla \theta|^2 \big]u - \beta\big(2 \nabla u \cdot \nabla \theta  + u \Delta \theta \big)  + u^p - \gamma u, \label{eq:A_intro}\\
& u \pa_t \theta = \beta \big[\Delta - |\nabla \theta|^2   \big] u + 2 \nabla u \cdot \nabla \theta  + u \Delta \theta + \delta u^p .\label{eq:theta_intro}
\end{align}
\eseq
The case $\beta = 0$ is related to a class of reaction-diffusion equations appearing in the study of pattern formation, see for example \cite{hagan1982spiral} and references therein.

\subsection{Main result}
For any $k \geq 1$, we introduce the functional spaces $\mathfrak E_k$ and  $\mfr F_k$ 
 \begin{equation}\label{def:EkFk}
   \mfr E_k = \Big\{ w: \; \|w\|_{\mfr E_k} = \sum_{j = 0}^k\|\nabla ^j w\|_{\rho_j} < +\infty \Big\}, \quad \mfr F_k =\Big\{ \phi, \; \| \phi \|_{\mfr F_k} = \sum_{j = 1}^k \|\nabla^j \phi \|_{\mathring{\rho}_j} < +\infty \Big\},
 \end{equation}
 where $\| \cdot\|_{\rho_k}$ and $\|\cdot\|_{\mathring{\rho}_k}$  stand for the standard weighted $L^2$-norm with $\rho_k$ and $\mathring{\rho}_k$ being defined as in \eqref{wg:rho}. 
Let $\bar U$ be the universal profile 
\begin{equation}\label{def:Ubar}
\bar U(z) = \Big(p-1 + \frac{(p-1)^2}{4\flat_*} |z|^2 \Big)^{-\frac{1}{p-1}}, \quad \forall \ z \in \Rb^d,
\end{equation}
and $V_0$ be a non-degenerate global maximizer of $u_0$ defined by 
 \beq\label{eq:ass_init}
 V_0 = \arg\max u_0(z) , \quad  u_0(V_0)>0, \quad -\na^2 u_0( V_0 ) \succ 0,
 \eeq
 where $A \succ 0$ means that $A$ is a positive definite matrix. The main result of this paper is the following theorem.

\begin{theorem}[Existence and stability of blowup solutions to \eqref{cgl}]\label{thm:blp} 
Consider $\beta, \delta$ in the sub-critical range \eqref{def:subCri}, i.e. $\flat_* > 0$, $p > 1$ and $d \geq 1$. Let $K = K(d,p) \in \mathbb{N}$ be defined as in \eqref{K}. There exists an open set $\Oc \subset \mfr E_K \times \mfr F_{K}$ of initial data $\psi_0 = u_0 e^{\theta_0}$ with the property \eqref{eq:ass_init} such that the corresponding solution $\psi = u e^\theta$ to \eqref{cgl} blows up in finite time $T$ and the following asymptotic behaviors hold. \\

\noindent (i) \textup{(The amplitude-phase decomposition)}
\begin{equation} \label{est:asymptoticEKFK}
    \Big\| H(t) u \big( \RR(t) z + V(t), t \big) -  \bar U(z) \Big\|_{\mfr E_K} + \Big\| \theta \big( \RR(t) z + V(t), t \big) -  \mu(t) - \delta \log \bar U(z)\Big\|_{\mfr F_K} \leq \frac{C}{1 + |\log (T-t)|},
\end{equation}
where $H(t)$ and $\mu(t)$ are scalar functions, $\RR(t)$ is an upper triangular matrix and $V(t)$ is a vector in $\Rb^d$, 
\begin{equation} \label{est:limitHRV}
    \lim_{t \to T} \frac{H(t)^{p-1}}{T-t} =1, \quad \lim_{t \to T} \frac{\RR(t)}{\sqrt{(T-t)|\log (T-t)|}} = \textup{I}_d, \quad \lim_{t \to T} V(t) = V_T,
\end{equation}
for some $V_T \in \Rb^d$, and $\mu(t)$\footnote{While the
$\mfr F_K$ norm in \eqref{est:asymptoticEKFK} only involves $\na^i \phi, i \geq 1 $ and $\mu(t)$ does not play a role in \eqref{est:asymptoticEKFK}, we keep $\mu(t)$ in \eqref{est:asymptoticEKFK} to indicate that it captures the phase of $\psi$. See \eqref{est:limitmu}, \eqref{eq:thm_estb}.}   {admits the expansion}
\begin{equation}\label{est:limitmu}
\mu(t) = -\frac{\delta}{p-1} \log (T-t)  - 
\frac{d\beta(1  +\delta^2)}{2\flat_*} \log |\log (T-t)| + \hat \mu(t), \quad \lim_{t\to T} \hat \mu(t) = \hat \mu_T,
\end{equation}
 for some scalar function $\hat \mu(t)$ and $\hat \mu_T \in \Rb$. \\

\noindent (ii) \textup{($L^\infty$ asymptotic behavior)} 
\begin{equation}\label{eq:thm_estb}
 \Big\| |\log(T-t)|^{\imath  \frac{d\beta (1 + \delta^2)}{2\flat_*}} (T-t)^{\frac{1 + \imath \delta}{p-1}}  e^{ - \imath \hat \mu(t) } \; \psi( \RR(t) z + V(t) , t )  - \bar U^{1 + i \d} \Big\|_{L^{\infty}}  \leq \frac{C}{1 + |\log( T- t)|^{\sigma'}},
\end{equation}
where 
$\sigma' = \min \big\{1, \frac{4}{p-1}\big\}$ and $C = C(u_0, \theta_0) > 0$.
\end{theorem}

\medskip

 \begin{remark}
 [Description of the set $\Oc$ of initial data]\label{rem:blowup} For initial data $u_0$ satisfies the property \eqref{eq:ass_init}, we can define an upper triangular matrix $\Mc_0$ with $\Mc_{0,ii} > 0$\footnote{Simple linear algebra shows that $\Mc$ is uniquely determined.} and the rescaled variables $(U_0, \Th_0)$ 
\beq\label{eq:init_res}
\bal
& H_0 = \f{ \kp_0}{ u_0(V_0 )}, \quad 
\Mc_0^T \Mc_0  =  - \f{  \kp_0  \na^2 u_0(V_0)}{ \kp_2 u_0(V_0)  }
= H_0 \f{   \na^2 u_0(V_0)}{ \kp_2   },\quad \kp_0 = \bar U(0), \quad \kp_2 = \pa_1^2 \bar U(0), \\
& U_0(z) = H_0 u_0( \Mc_0^{-1} z + V_0),  \quad  \Th_0(z) =  \th_0( \Mc_0^{-1} z + V_0),
\eal
\eeq
where $\bar U$ is defined in \eqref{def:Ubar}. Since $\na u(V_0, 0) =0$, by definition, \eqref{eq:init_res} implies the following normalization 
 \beq\label{eq:init_res2}
    U_0(0)= \kp_0 = \bar{U}(0)\,,\quad \nabla U_0(0)=0\,,\quad \nabla^2 U_0(0)=\nabla^2 \bar{U}(0)=\kappa_2 I_d\,. 
  \eeq
Let $\nu > 0$ be small,  $\epsilon_2$ and $C_b$ be defined in \eqref{eps} and \eqref{eq:boot_U},  the set of initial data in Theorem \ref{thm:blp} consists of initial data $(u_0, \th_0)$ satisfying \eqref{eq:ass_init} and its rescaled variable $(U_0, \Th_0)$ satisfies 
\beq\label{eq:ass_init2}
U_0 \bar U^{-1- \e_2} > 2C_b , \quad H_0^{p-1} < \nu, \quad  u_0(V_0)^{-p} \tr( \na^2 u_0(V_0) ) < \nu ,
\eeq
and
\begin{equation}\label{def:EKFKbar}
\|W_0\|_{\mfr E_K}= \|U_0 - \bar U\|_{\mfr E_K} < \nu, \quad \|\Phi_0\|_{\bar {\mfr F}_K} := \| \Th_0 - \delta \log \bar U \|_{\mfr F_{K-1}} + \|\langle z \rangle^{K - \frac{d}{2}} \nabla^K (\Th_0 - \delta \log \bar U) \|_{L^2} < \nu.
 \end{equation}
The last quantity in \eqref{eq:ass_init2} is invariant under the parabolic rescaling: $u_{0, l}(z) 
= l^{1/(p-1)} u_0( l^{1/2} z)$. We will use its smallness to show that the viscous terms are small compared to the nonlinear terms. The lower bound $U_0 \bar U^{-1-\e} > 2 C_b$ in \eqref{eq:ass_init2} ensures that $U_0(z) \neq 0$ for any $z$, without which can lead to low regularity of rescaled velocity $|U|^p$ of $u^p$ \eqref{eq:A_intro}.
\end{remark}

\begin{remark}[Positive definiteness of the Hessian of the initial data]\label{rem:shape}
While the limiting blowup profile $\bar U$ \eqref{def:Ubar} is isotropic near $z=0$, we do not need to assume that the initial data $u_0$ is isotropic near $V_0$, i.e., $\na^2 u_0(V_0)$ is close to $c I_d$ for some $c \neq 0$. By introducing the upper triangular matrix $\RR(t)$ in the rescaling (see \eqref{eq:dyn_scal0}), 
we can handle a much larger class of initial data with non-degenerate global maximizer \eqref{eq:ass_init}. 
\end{remark}

  \begin{remark}\label{rem:1D}
  The asymptotics of the blowup solution \eqref{eq:thm_estb}  recovers the constructed result of Masmoudi-Zaag \cite{MZjfa08} for the case $d = 1$.  
  The set of initial data leading to the blowup solution described in Theorem \ref{thm:blp} is larger than the one in \cite{MZjfa08} which is only a subset of $L^\infty(\Rb)$. We note that there is a free phase-shift $\hat \mu$ in \eqref{eq:thm_estb} corresponding to the phase invariant of \eqref{cgl} that was fixed to be $\hat \mu(t) = 0$ in \cite{MZjfa08} by a specific choice of initial data through a topological argument. We remark that the relaxing asymptotics \eqref{est:limitHRV} and \eqref{est:limitmu} are natural for rigorous stability analysis in all dimension $d \geq 1$ treated in this present paper.
\end{remark}

{Note that the asymptotic behavior \eqref{est:asymptoticEKFK}, \eqref{eq:thm_estb} involves the 4 parameter functions $H, \RR, V$ and $\mu$ (or $\hat \mu$) which are responsible for all the symmetries of \eqref{cgl}.\footnote{
Although $H(t)$ is absent in \eqref{eq:thm_estb}, we can replace the factor $(T-t)^{1/(p-1)}$ by $H(t)$ using \eqref{est:limitHRV}.} Theorem \ref{thm:blp} is stated in terms of the rescaled profiles, with a singular weight at the origin. 
We note that the condition \eqref{eq:ass_init} and parameters $\Mc_0, H_0, V_0$ in \eqref{eq:init_res} are $C^2$-stable if the global maximizer is unique.
We can therefore simplify the assumptions in Theorem \ref{thm:blp} to obtain the following stability results with a more explicit description of the open set of initial data. 
 }

\begin{theorem}
    [Stability of blowup solutions to \eqref{cgl}]\label{cor:stab}

Let $K = K(d,p) \in \mathbb{N}$ be defined as in \eqref{K}, and $\cH^K, \mfr F_{K-1}$ be the norms defined in \eqref{norm:Hk}, \eqref{def:EkFk}. Suppose that $(u_0, \th_0)$ satisfies the assumptions \eqref{eq:ass_init}, \eqref{eq:ass_init2},
\eqref{def:EKFKbar} and $V_0$ is the unique global maximizer: $u_0(V_0) > u_0(z)$ for all $z \neq V_0$. 
There exists {$\epsilon_0 = \e_0(u_0) > 0$} such that if 
\beq\label{eq:ass_cor}
\| \td u_0 - u_0 \|_{\cH^K} + \|(\td u_0 - u_0 ) \bar U^{-1-\e_2} \|_{L^{\infty}} +  \| \td \th_0 - \th_0\|_{\mfr F_{K-1}} + \|\langle z \rangle^{K - \frac d2} \nabla^K (\td \th_0 - \th_0) \|_{L^2}  < \epsilon_0,
\eeq
the solution $\tilde \psi = \td ue^{\td \th}$ to \eqref{cgl} with the initial data $ \tilde \psi_0 = \td u_0e^{\td \th_0}$ blows up in finite time $\td T$. Moreover, there exists $H(t), \RR(t), V(t), \mu(t)$ satisfying \eqref{est:limitHRV} and \eqref{est:limitmu} such that \eqref{est:asymptoticEKFK} and \eqref{eq:thm_estb} holds for $(\td u(t), \td \th(t))$ with $T$ being replaced by $\tilde T$. 
\end{theorem}

{From the proof of Theorem \ref{cor:stab}, it can be shown that $\e_0$ depends on $u_0$ through 
its certain norms. We do not state the dependence explicitly for simplicity.
}
\begin{remark} 
The assumptions in Theorems \ref{thm:blp},\ref{cor:stab} are satisfied, e.g. for $ u_0 = C \bar U$, $\th_0 = \bar \Th_0$ with $C$ sufficiently large. 
The weighted norms $\|\cdot \|_{\cH^K}$ \eqref{norm:Hk} and $\|\cdot \|_{\Fc_K}$ \eqref{def:EkFk} are well-defined for sufficiently smooth functions with fast decay. We do not require that $u_0 - \td u_0$ agrees up to $\Oc(|z - V_0|^k), k >0$ near the maximizer $V_0$ of $u_0$. 
\end{remark}

\subsection{
Dynamic rescaling formulation with extra modulation parameters} 
\label{sec:intro_dyn}

{The  dynamic rescaling formulation or the modulation technique was developed to study singularity formulation in the nonlinear Schr\"odinger equation  \cite{mclaughlin1986focusing}, \cite{LPSSphysA88} numerically and various nonlinear PDEs; see the comprehensive references on \eqref{nlh},  \eqref{nls}, and related models in \cite{QSbook07}, \cite{Fbook15}.}
Recently, researchers also generalized this technique for fluid mechanics 
\cite{chen2019finite2,chen2022stable,chen2023stable}, \cite{elgindi2021finite}.  We can establish singularity in two steps. Firstly, one constructs an approximate steady state of the dynamic rescaling equation (analytically or numerically). Secondly, 
  one performs linear and nonlinear stability analysis for perturbation around the approximate steady state with appropriate normalization conditions. The law of blowup will then be prescribed by the normalizing constants.
  One can establish stability using a $L^2$-based \cite{chen2021finite,chen2021HL} or $L^\infty$-based \cite{chen2022stable} argument. In these arguments, {one of the crucial steps} is to design appropriate singular weights depending on the profile, and then use the weights to derive damping terms for the energy estimates. 
  The approach does not require an explicit profile and is robust to small perturbation, which makes it possible to combine weighted energy estimates for stability analysis,
  a numerical \textit{implicit} profile, and computer-assisted proofs to construct blowup solutions. {See for example,} \cite{chen2022stable,chen2023stable}  for applications in 3D incompressible Euler equations with smooth data and \cite{chen2021finite, chen2021HL}, \cite{chen2020singularity}, \cite{hou2024blowup} for related 1D models.


In \cite{HNWarXiv24}, the authors generalized the $L^2$-based methodology to establish a type-I
\footnote{A blowup solution to \eqref{cgl} is of Type I if it satisfies the bound $\lim_{t\to T}(T-t)^{-\frac{1}{p-1}}\|u(t)\|_\infty < \infty$, otherwise, blowup is of Type II.} 
 blowup for the semilinear heat equation beyond the self-similar setting where there is a logarithm correction in the self-similar scaling
for the spatial variable. {Compared with the mentioned works \cite{BKnon94}, \cite{HVaihn93}, \cite{MZdm97}, \cite{MZjfa08} where the authors heavily relied on a spectral analysis with detailed properties of the associated linearized operator to establish the existence and stability, we simply suppresses unstable directions {and neutral modes} via a clear characterization of weighted Sobolev spaces, without using Brouwer's fixed-point theorem or a topological argument. The correct Type I blowup rate is automatically inferred by enforcing proper vanishing conditions of the perturbation. We remark that for Type I blowup, there is indeed a link between the vanishing conditions at the origin and the vanishing coefficients projected onto the unstable {and neutral} spectral eigenfunctions; see for example \cite{HNWarXiv24} and \cite{MZdm97} in the case of the nonlinear heat equation, where the eigenfunctions correspond to Hermite polynomials. A further explanation of this connection is discussed in Step 2 in Section \ref{sec:idea}.} 

\vspace{0.1in}
\noindent \textbf{A generalized dynamical rescaling formulation with extra modulation parameters\footnote{The modulation parameters are also known as normalization constants in the dynamic rescaling formulation.}.}
In this article, we further generalize the above framework and the ideas in \cite{HNWarXiv24}. {In particular, it consists of three main steps:}

\vspace{0.1in}


\noindent \textit{Step 1 ({Renormalization with extra modulation parameters}):} {The renormalization is an essential step in the study of nonlinear PDEs with symmetries including incompressible/compressible fluids equations. For the equation \eqref{cgl}, there are trivial symmetries (see Section \ref{sec:res}) from which we introduce the following renormalization in terms of the amplitude-phase representation  $\psi(x,t) = u(x,t)e^{\imath \theta(x,t)}$,} 
\begin{equation}\label{eq:intro_dyn}
U(z,\tau) = H(\tau) u\big( \RR(\tau) z + V(\tau), t(\tau)\big), \quad \Theta(z, \tau) = 
 \theta\big( \RR(\tau) z + V(\tau), t(\tau)\big), 
 \quad 
 t(\tau) = \int_0^\tau H^{p-1}(s) ds,
\end{equation}
where $ \RR(\tau) \in \Rb^{d \times d}$ is a upper triangular matrix, $V(\tau) \in \Rb^d$ and $H(\tau) \in \Rb_+$. Here, $H$ is responsible for time, $V$ for spatial translation. 

The key novelty is that in addition to modulation parameters corresponding to the symmetries, we introduce \textit{extra modulation parameters}. Instead of applying the same rescaling to $z_i$, we rescale $z_i$ with \textit{different but similar} scalings following the ideas  \cite{HNWarXiv24}. 
See also a recent work \cite{hou2024nearly} on the generalized Navier-Stokes equations, in which the author developed a generalized dynamic rescaling formulation by using different rescalings for the $r$ and $z$ directions respectively and a self-similar blowup was observed numerically. 
In the case of \eqref{cgl}, we rescale $z_i$ with scaling slightly perturbed from the parabolic scaling. We remark that the choices of different scalings for $z_i$ violate the scaling symmetries.
Yet, in the case of \eqref{cgl}, we will show that the violation is asymptotically small and $\RR(\tau)$ converges to $c(\tau) I_d$ asymptotically for some scalar function $c(\tau)$. Therefore, the above renormalization \eqref{eq:intro_dyn} asymptotically agrees with the classical dynamical rescaling formulation \cite{chen2022stable,chen2023stable}, \cite{chen2021finite, chen2021HL}. 
These extra parameters provide us extra $d-1$ degrees of freedom, and we have crucial $1 + d + \f{d(d+1)}{2}$ degrees of freedom in total in choosing the dynamic variables $H(\tau), V(\tau), \RR(\tau)$. 
{ 
  A natural idea to represent the scaling in $z_i$ and capture the rotation symmetries is to choose $\RR(\tau) = D(\tau) Q(\tau)$ with a diagonal matrix $D$ and an orthogonal matrix $Q$. Yet, it is challenging to parametrize a time-dependent orthogonal matrix in $\Rb^{d \times d }$. 
  Instead, we use the upper triangular matrix $\RR(\tau)$ with $\f{d(d+1)}{2}$ parameters. 
}

To determine these modulation parameters, we impose normalization conditions on $\na^i U(0), i=0,1,2$, so that the perturbation of $U$ vanishes $O(|z|^3)$ near $z=0$ and we can perform weighted energy estimates mentioned above.
Note that the number of (different) equations and that of the degrees of freedom are \textit{exactly} the same.
For \eqref{cgl}, these conditions allow us to completely eliminate the unstable and neutrally stable modes of the linearized operator. See Step 2 in Section \ref{sec:idea} for more details.

\vspace{0.1in}

\noindent \textit{Step 2 (Equations of the profiles and modulation parameters)}: We derive the equations of $F = (U, \Th)$ and matrix (or vectors) 
$\Qc(\tau)$ governing the {modulation parameters} $\RR(\tau), V(\tau), H(\tau)$,
\beq\label{eq:intro_full}
 \pa_{\tau} F = \msN_F(F, \Qc), \quad  \f{d}{d \tau} \Qc = \msN_{\Qc}(F, \Qc) ,
\eeq
where $\msN_F$ is a nonlinear function and $\msN_{\Qc}$ is a matrix. Then the rescaling system is completely determined, and we further construct the approximate steady state $(\bar F, \bar \Qc)$ analytically or numerically.

\vspace{0.1in}

\noindent \textit{Step 3 (Stability analysis and the log correction):} In general, we \textit{do not} know \textit{a-priori} that the approximate steady state $(\bar F, \bar \Qc)$ is stable in some suitable topology. Nevertheless, if we can establish {stability of} $(\bar F, \bar \Qc)$ following the strategy mentioned above and $H(\tau)^{p-1}$ is integrable, then we can obtain finite time blowup using \eqref{eq:intro_dyn} and the law of blowup will then be prescribed by the normalizing constants. 

For \eqref{cgl}, we will use energy method with an energy $E$ for the perturbation $F - \bar F$ 
to establish 
\[
\f{d}{d \tau} E \leq - c_1 \cdot E + C \tr(\Qc) + l.o.t., \quad 
\quad  \f{d}{d \tau} \tr(\Qc) \leq  - c_2 \cdot \tr(\Qc)^2 + l.o.t.,
\]
for some $c_1, c_2, C >0$, where $\Qc$ further satisfies that it is a positive definite matrix and $l.o.t.$ denotes some terms that are very small.
The second ODE of $\tr(\Qc)$ further implies that $|\tr(\Qc)| \les (1 + \tau)^{-1}$. A further refinement of this algebraic decay in the self-similar time implies a log correction $\log(T-t)$ in the blowup rate. We will elaborate more in Step 2 in Section \ref{sec:idea}. 

\

One can thus hope to combine the above method for a log correction and the framework \cite{chen2022stable,chen2023stable} to problems with numerical steady states, while  spectral analysis heavily hinges on a simple and analytical approximate steady state with explicit (nonlinear heat \cite{MZdm97}) or at least asymptotical spectral information of the linearized operator (Keller-Segel \cite{CGMNcpam21}). 
For example, constructing a smooth (approximate) steady state analytically for 3D incompressible Euler or Navier-Stokes equations (NSE) is challenging and remains an open problem.
On the other hand, constructing a numerical approximate steady state with computer-assistance is much more feasible. See \cite{chen2022stable,chen2023stable}, \cite{wang2022self} for the construction in 3D Euler equations. 
For NSE, self-similar blowup with a perfect self-similar scaling has been ruled out 
\cite{tsai1998leray}, \cite{sverak1996leray}. Yet, one can construct a blowup violating these non-blowup results by adding a log correction in the spatial variable. 
See numerical evidence on the singular behavior of NSE with a potential logarithm correction in the potential blowup by the second author \cite{hou2023potentially}.

\subsection{Ideas of the blowup analysis}\label{sec:idea}

{We first discuss some of difficulties in the study of singularity formation in \eqref{cgl}. Then we follow the generalized dynamic rescaling framework to establish the existence and stability of asymptotically self-similar blowup  solutions to \eqref{cgl} by briefly discussing the strategy and main ideas of our analysis.}

\paragraph{\textbf{Difficulties:}} Compared with \eqref{nlh} or \eqref{nls}, the analysis for the complex Ginzburg-Landau equation \eqref{cgl} has the following additional challenges.

\vspace{0.1in}
  \textbf{1.} The complex Ginzburg-Landau equation \eqref{cgl} is not of a gradient form, rendering energy estimates hard. To overcome this challenge, we use the amplitude-phase representation \eqref{eq:intro_rep}, \eqref{eq:CGL_intro} to analyze \eqref{cgl}.

    \vspace{0.1in}
  \textbf{2.}  
 We remove the even symmetry assumption of the perturbation required in \cite{HNWarXiv24} to recover full stability. {
 Without the even symmetry assumption, we have more potentially unstable modes for the linearized operator. We control these unstable modes using the generalized dynamic rescaling formulation in Step 1 in Section \ref{sec:intro_dyn}.
 }

 \vspace{0.1in}
  \textbf{3.}
  We consider the whole range of the nonlinearity $p > 1$. For the analysis of the phase equation \eqref{eq:theta_intro} and the nonlinearity $u^p$ \eqref{eq:A_intro}, \eqref{eq:theta_intro}, we need to bound the rescaled amplitude $U$ from below, which we establish using the maximal principle and a weighted $L^{\infty}$ estimate. Due to the non-integer power $p$ to control $\na^K(U^p)$ in the $H^K$ estimate, which leads to terms like $ U^{p-K} (\na U)^K$, we need to obtain sharp decay estimates for $\na^i U$. This is done by 
  choosing an almost tight power in the far field of the weight for the weighted $H^k$ energy estimates and using interpolation and embedding inequalities following \cite{chen2024Euler}.  
An additional difficulty comes from the coupling between $u, \th$ in the viscous terms in \eqref{eq:A_intro}, \eqref{eq:theta_intro}. 
We design the top order energy with a special algebraic structure to cancel out the top order terms and show that the viscous terms have a good sign. See Step 3(b) in Section \ref{sec:idea}.

\paragraph{Ideas and strategy:} We briefly discuss the strategy and main ideas of our analysis.

\vspace{0.1in}
\noindent \textit{Step 1 (Dynamical rescaling formulation)}:
We follow Step 1 in Section \ref{sec:intro_dyn} to perform the rescaling \eqref{eq:intro_dyn}.
Then, we introduce the following factors governing the evolution of these parameters 
\beq\label{eq:intro_Pv}
 \frac{H_\tau}{H}=c_U, 
 \quad c_U = - \frac{1}{p-1} + c_W, 
 \quad 
 \Mc = e^{- \frac{\tau}2} \RR^{-1},
 \quad \Vc = - \RR^{-1} V_{\tau}, 
 \quad \Pc = \Mc_{\tau} \Mc^{-1}.
\eeq

\vspace{0.1in}

\noindent \textit{Step 2 (Normalization and vanishing conditions)}:
{Let $\bar U$ be the profile defined in \eqref{def:Ubar}}. To determine {the law for the parameter functions} $H(\tau), V(\tau), \RR(\tau)$, we enforce the following normalization conditions on the amplitude $U$:
\beq\label{eq:intro_vanish}
k = 0,1,2, \quad \nabla^k U(0, \tau) = \na^k \bar U(0).  
\eeq
Since $\na^2 U \in \Rb^{d \times d}$ is symmetric, we have $1 + d + \f{d(d+1)}{2}$ different equations, which match the degrees of freedom of the dynamic variables exactly. 
The above conditions determine the initial {modulation parameters} $H(0), \RR(0), V(0)$ and the leading order system of $c_W, \Pc, \Vc$ \eqref{eq:intro_Pv} 
\beq\label{eq:intro_ODE}
\bal
&c_W = \frac{2(1 - \beta \delta)}{4\flat_*} \textup{tr}(\Qc) + O(\cE_0), 
\quad \Vc = O(\cE_0),\quad \Pc =O( |\Qc| + \cE_0), 
\quad  \textup{where} \quad \Qc = H^{p-1} e^{\tau} \Mc \Mc^T, 
  \\
\eal
\eeq
where $\cE_0$ tracks some lower order terms depending on the perturbation $(W, \Phi)$ \eqref{eq:W_Intro}, \eqref{eq:Phi_Intro}  and $\Qc$. 
The main unknown $\Qc \in \Rb^{d\times d}$ (a positive definite matrix) solves the following ODE 
\beq\label{eq:intro_ODE2}
\frac{d}{d\tau} \tr(\Qc) = -\tr(\Qc^2) + O(\cE_0 |\Qc|  )
\leq - \f{1}{d} (\tr(\Qc))^2 + O(\cE_0 |\Qc|  ) .
\eeq
From \eqref{eq:intro_ODE2}, \eqref{eq:intro_ODE}, \eqref{eq:intro_Pv}, we can control all the modulation parameters. A refined estimate using $\tr(\Qc)$ and $\tr(\Qc^{-1})$ yields $\Qc = \frac{1}{\tau}\textup{I}_d + O(\tau^{-3/2 +})$, together with an asymptotic refinement of the phase yield the asymptotics in Theorem \ref{thm:blp}.  See Section \ref{sec:ode} for deriving \eqref{eq:intro_ODE}, Section \ref{sec:refineasym} for the estimates of $\Qc$ and Proposition \ref{prop:linf} for the refinement of the phase.  

Roughly speaking, imposing \eqref{eq:intro_vanish} for $U$ is equivalent to imposing local orthogonality conditions for the perturbation $W = U - \bar U$ to $1, z_i, z_i z_j, 1 \leq i, j \leq d$. These functions are all the neutrally stable and unstable modes of $L =\textup{Id} - \frac{1}{2}z \cdot \nabla$, 
which behaves similarly to the main linearized operator $\msL_{\bar U}$ in \eqref{eq:W_Intro} {for $|z|$ small}.  We then get a damping in the weighted $L^2$ energy estimate.

\vspace{0.1in}
\noindent \textit{Step 3 (Stability analysis)}: We linearize $(U, \Th)$ around  the approximate steady state  $(\bar U, \bar \Theta)$ defined in \eqref{def:Ubar} and \eqref{ss-} and obtain the equations for the perturbation $W = U - \bar U, \Phi = \Th - \bar \Th$,
\bseq\label{eq:per_Intro}
\begin{align}
{W}_\tau &= \msL_{\bar U} W +\mathscr{F}_{ U}+\mathscr{N}_U+\mathscr{D}_U ,
\quad 
 \msL_{\bar U} W = \big(-\frac{1}{p-1} + p \bar U^{p-1} - \frac{1}{2}z \cdot \nabla \big)W \, \label{eq:W_Intro}\\
\Phi_\tau &=-\frac{1}{2}z \cdot \nabla \Phi+\mathscr{F}_\Theta+\mathscr{N}_\Theta+\mathscr{D}_\Theta, \label{eq:Phi_Intro}
\end{align}
\eseq
where $\mathscr{F}_U, \mathscr{N}_U, \mathscr{D}_U, \mathscr{F}_\Theta, \mathscr{N}_\Theta, \mathscr{D}_\Theta$ are small residue, nonlinear, and viscous terms \eqref{eq:lin}, and will be treated perturbatively. 
Below, we outline the stability estimates and focus on the linear part and the viscous terms.


{
\vspace{0.1in}
\textit{(a) Estimates of the $(W, \Phi)$} : $\mathscr{L}_{\bar U}$ to the leading order is the linearized equation for Riccati equation or semilinear heat equation. With the vanishing conditions $\na^k W = 0, k=0,1,2$ \eqref{eq:intro_vanish}, we obtain its stability using weighted $H^k$ estimates with singular weights. See 
 \cite{HNWarXiv24} in the case of $p=2$ and \cite{chen2024stability}, \cite{chen2021finite,chen2021HL}, \cite{chen2020singularity}. 

The right-hand side of \eqref{eq:per_Intro} only involves $\na \Phi$, which enjoys better stability. We estimate $\na \Phi$ by performing a similarly weighted $H^k$ estimate (starting from $k=1$) on \eqref{eq:Phi_Intro} and exploiting the term $-\frac{1}{2}z \cdot \nabla \Phi$.


In the nonlinear estimates  and the estimates of the phase, we need to control $1/U$. We use the maximal principle and a weighted $L^{\infty}$ estimate to obtain a lower bound of $U$. 




\vspace{0.1in}
\textit{(b) Estimates of the viscous terms}: 
For the viscous terms $\mathscr{D}_U, \mathscr{D}_\Theta$ \eqref{eq:vis}, the main difficulty 
is the coupling between $U$ and $\Theta$. At the top $H^K$ estimate, the highest order derivative terms read 
\begin{align*}
    \na^K \msD_{U} &= \D_{\Qc} \na^K U - \be U \D_{\Qc} \na^K \Th + l.o.t. \coloneqq I_1 + I_2 + l.o.t., \\
\ \na^K \msD_{\Th} & = \be \f{\D_{\Th} \na^K U}{U} + \D_{\Qc} \na^K \Th + l.o.t.
\coloneqq I_3 + I_4 + l.o.t.,
\end{align*}
where $\D_{\Qc} F$ is a weighted elliptic operator defined in \eqref{eq:Del_Q}. 
The terms $I_1, I_4$ lead to damping terms of $\nabla^{K+1}U$ and $\nabla^{K+1}\Theta$ via integration by parts. 
To control $I_2, I_3$, we exploit their cancellation using the energy $J_1$ below 
with some weight $\rho_K$ independent of $U, \Th$ and couple their estimates in $J_2$
\begin{equation}  \label{topcouple}
    J_1 =\int(|\nabla^{K}W|^2+{U}^2|\nabla^{K}\Phi|^2)\rho_K,
      \quad  J_2 = \int ( ( - \be U \D_{\Qc} \na^K \Th) \cdot \na^K U + (\be \f{\D_{\Th} \na^K U}{U}) \cdot U^2  \na^K \Phi) \rho_K.
\end{equation} 
 Applying integration by parts, $J_2$ reduces to some lower order terms and we can close the viscous estimates. 
For estimates of intermediate-order terms, we use interpolation inequalities following \cite{chen2024Euler}.

\vspace{0.1in}
\textit{(c) {Choosing the weights}}: 
To extract damping in the energy estimates, we need to design various suitable weights. The weights \eqref{wg:rho} for $W$ are very similar to those of the semilinear heat equation \cite{HNWarXiv24}. They are singular near $z=0$ for the lower order energy estimates and regular for the top order energy estimates so that the viscous terms will have a good sign.  In addition, we choose an almost optimal rate at the infinity for these weights to obtain a sharp decay estimate for $(W, \Phi)$ using interpolation and embedding following \cite{chen2024Euler}.

\vspace{0.1in}
\paragraph{\bf{Organization of the paper:}}

The rest of the paper is organized as follows. In Section \ref{sec:dyn}, we introduce the generalized dynamic rescaling formulation using the symmetries of \eqref{cgl} and derive the ODEs governing the modulation parameters. Section \ref{sec:stab} is devoted to the stability analysis of the profile. 
In Section \ref{sec:refineasym}, we establish the asymptotics of the blowup rate. 
In Section \ref{sec:thm_proof}, 
we prove Theorem \ref{thm:blp} and Theorem \ref{cor:stab}.

\vspace{0.1in}
\paragraph{\textbf{Notations:}} {
We use $\imath$ to denote the imaginary number, $\bar{f}$ to denote approximate profiles for the variable $f$, e.g.,$\bar U$, rather than conjugates, and  $(\cdot,\cdot)$ to denote the inner product on $\mathbb{R}^d$: 
$(f,g)=\int_{\mathbb{R}^d}f{g}$. For a weight $\rho$, we denote $\|f\|_{\rho}=(|f|^2,\rho)^{1/2}$. 
}
For matrix notations, we use $\mathrm{tr}(R)$ to denote the trace of a matrix $R$, $\mathbf{T}^u(R)$ to denote the upper triangular part of $R$; namely $(\mathbf{T}^u(R))_{ij}=R_{ij} \one_{ i\leq j}$. %
We use $\delta_{ij} = \one_{i = j}$ to denote the Kronecker delta function, 
and $| \mathbf{T}| := (\sum_{i} \TT_i^2)^{1/2}$ with summation over all entries $\TT_i$ to denote the tensor norm of a tensor $\mathbf{T}$, e.g., higher-order derivatives $\nabla^k f$. We use $C$ to denote an absolute constant only dependent on the constants $p,\beta,\delta,\gamma$ and the dimension $d$, which may vary from line to line. $C(\mu)$ denotes a constant depending on $\mu$. We denote $A=O(B)$ or $A\lesssim B$ if there exists an absolute constant $C>0$, such that $|A|\leq CB$, and denote $A\approx B$ if $A \les B$ and $B \les A$. Furthermore, we denote
\beq\label{eq:intro_nota}
\Lam = z \cdot \na , \quad \langle z\rangle=\sqrt{1+|z|^2}\,.
\eeq

\noindent \textbf{Parameters and special functions}: We introduce 
 \begin{equation}\label{ss-}
    \bar{\Theta} =  \frac{\delta}{p-1}\tau + \delta\log\bar{U} \,, \quad     c_p=\frac{(p-1)^2}{4\flat_*}, \quad \s = -\f{2}{p-1}.
    \end{equation}
We choose the weights for the $H^k$ estimates as follows:
\begin{equation}
\label{wg:rho}
    \begin{aligned}
        &\rho_k=|z|^{-6+\epsilon-d+2k}+c_0|z|^{- 2 \s -\epsilon-d+2k}, \   0\leq k\leq\frac{d+5}{2}\,, \quad 
    {\rho}_k=1+c_1|z|^{ - 2 \s  -\epsilon-d+2k}, \  \frac{d+5}{2}<k \, , \\
& \mathring{\rho}_k=|z|^{2k-1-d}\,,  \ 0< k\leq\frac{d}{2}\,, \qquad \qquad \ \mathring{\rho}_k=1+|z|^{2k-1-d}\, , \ \frac{d}{2}< k<K\,, \ \qquad \qquad \mathring{\rho}_K=U^2\rho_K\,,
\end{aligned}
\end{equation}
where we determine the constants in the following order: 
\begin{align}
K & =2d+4+2\left \lceil\frac{p+1}{\min\{p-1,c_p\}}\right \rceil\,, \label{K} \\
  \epsilon &{ =\min\{\frac{p-1}{5(p+3)(K+p)},\frac{4}{5(p + 3)(K+p)}\},}
 \quad \e_2 = \f{ (p-1) \e}{4} .
   \label{eps}
 \end{align}
For $c_0, c_1$ used in \eqref{wg:rho}, we determine $c_0$ via \eqref{c0} and $c_1$ via \eqref{defered}. {Note that $c_0, c_1$ only depends on $K, \e, p, \d$. Hence, they are considered as fixed constants throughout the paper.}

\section{Generalized dynamical rescaling formulation}\label{sec:dyn}

In this section, we introduce a generalized dynamic rescaling formulation and decompose the complex Ginzburg-Landau equation into the equation of the phase and the amplitude. We will consider a linearization around the approximate profiles and {choose the modulation parameters} based on the vanishing conditions. Finally, we will estimate the ODE of the {modulation parameters}. 

\subsection{Symmetries and renormalization}\label{sec:res}

We exploit the following symmetries of equation \eqref{cgl} to study stability for general perturbation, which will motivate our choice of rescaling. 
If $\psi(x,t)$ solves \eqref{cgl}, all of the following also solves \eqref{cgl}:
\begin{enumerate}
    \item Phase shift: $\psi_{a}(x,t)\coloneqq e^{
\imath a}\psi(x,t)$, for $a\in\mathbb{R}$.
  \item Parabolic scaling for $\gamma=0$: $\psi^{l }(x,t)\coloneqq l^{1/(p-1)}\psi( l^{1/2} x, l t)$, for $ l \in\mathbb{R}$. 
    \item Translation: $\psi_{{b}}(x,t)\coloneqq \psi(x-{b},t)$, for ${b}\in\mathbb{R}^d$.
    \item Rotation: $\psi^{R}(x,t)\coloneqq \psi(Rx,t)$, for  orthogonal matrix $RR^T=I_d$.
\end{enumerate}
We use the symmetry groups of the parabolic scaling via a rescale in amplitude, of the translation via a shift in space, and of the rotation via a rotation and rescaling in the spatial variable parametrized by an upper triangular matrix. {In sum, we can exploit modulation with $1+d+\frac{(d+1)d}{2}$ degree of freedom.} 
The phase shift invariance corresponding to a constant addition in $\Theta$ is taken care of in \eqref{est:limitmu} and \eqref{eq:thm_estb}. It is irrelevant to the dynamical modulation of stability since the right-hand sides of \eqref{eq:A_intro} and \eqref{eq:theta_intro} only involve the derivatives of $\Theta$. We remark that the modulation of symmetries due to Galilean transformations, including the rotation symmetry, has been used successfully to obtain shock formation in compressible Euler equations with fine characterization \cite{buckmaster2023formation}.  
Below, we will use a general upper triangle matrix, which simplifies the parametrization of the time-dependent orthogonal matrix.

For solution $\psi$ to \eqref{cgl}, we consider the amplitude-phase form $\psi(x,t)= u(x,t)e^{\imath\theta(x,t)}$, where $u(t): x \in \Rb^d \to \Rb_+$ and $\theta(t): x \in \Rb^d \to \mathbb{R}$. For the amplitude $u$ and phase $\psi$, we introduce the generalized dynamic rescaling formulation
\beq\label{eq:dyn_scal0}
U(z,\tau)=H(\tau) u\big( \RR(\tau)z+V(\tau) ,t(\tau)\big)\,,\quad \Theta(z,\tau)= 
 \theta\big(\RR(\tau)z+V(\tau),t(\tau)\big)\,,
\eeq
where the main unknown parameter functions are $ \RR \in \Cc^1 \big([\tau_0, +\infty), \Rb^{d\times d}\big)$ a non degenerate upper triangular matrix, $V \in \Cc^1 \big([\tau_0, +\infty), \Rb^{d}\big)$ and $H$ is given by
\beq\label{eq:dyn_scal1}
{H= H(0)\exp{(\int_0^\tau {c}_U(s) ds)}\,, \quad t(\tau) =\int_0^\tau H^{p-1}(s) ds\,.}
\eeq
In a compact form, we have
\begin{equation}\label{drf-nd}
    {U(z,\tau) e^{\imath \Theta(z,\tau) }=H(\tau) (u e^{\imath \theta})\big( \RR(\tau)z+V(\tau),t(\tau)\big), \quad \psi = u e^{\imath\theta}. }
\end{equation}
We decompose the solution into the approximate steady states \eqref{ss-},  with perturbations $W, \Phi$: 
     \begin{equation}\label{per_ansatz}
        U=\bar{U}+W\,,\quad \Theta=\bar{\Theta}+\Phi\,,\quad c_U= -\frac{1}{p-1} + c_W\,,\quad H =
e^{- \f{\tau}{p-1}  } C_W.
    \end{equation}
If $\RR(\tau)$ is a scalar factor and $V = 0$, \eqref{drf-nd} reduces to the standard dynamic rescaling formulation, see e.g., \cite{chen2021HL,chen2021finite,chen2019finite2}. If 
$\RR(\tau)$ is a diagonal matrix and $V = 0$, it reduces to a formulation similar to \cite{HNWarXiv24}. We will show that $\RR$ is close to some identity matrix and $V$ is some lower order term. 

We first compute the spatial derivative
$$\nabla U= H  e^{-\imath\Theta}\nabla \psi \RR-\imath H \psi e^{-\imath\Theta}\nabla\Theta \,,$$$$\nabla^2 U=H  e^{-\imath\Theta} \RR^T\nabla^2 \psi \RR-\imath H   e^{-\imath\Theta}(\nabla \psi \RR\nabla\Theta^T+\nabla \Theta \RR^T\nabla \psi^T)-H \psi e^{-\imath\Theta}\nabla\Theta\nabla\Theta^T-\imath H \psi e^{-\imath\Theta}\nabla^2\Theta\,.$$
We then write from \eqref{cgl} the equation for $U$,
\begin{equation*}\begin{aligned}
    {U}_\tau&=-\imath\Theta_\tau U+{c}_U {U}-( \f{1}{2} z+\Pc z+\Vc)\cdot\nabla U-\imath U ( \f{1}{2} z+ \Pc z+\Vc)\cdot\nabla\Theta +(1+\imath\delta){U}^{p}-{C}_U^{p-1}\gamma U\\& \quad +(1+\imath\beta)(\Delta_\Qc U+2\imath\langle\nabla U,\nabla\Theta\rangle_\Qc -U\langle\nabla \Theta,\nabla\Theta\rangle_\Qc+\imath U\Delta_\Qc \Theta)\,,
    \end{aligned}\end{equation*}
where $\Pc, \Vc$ are related to the matrix $\RR$ as  
 \beq \label{eq:Pv} 
 \Mc^{-1} = e^{-\tau/2} \RR, \quad
\Vc =- \RR^{-1}\dot{V}, \quad  \Pc=\dot{\Mc} \Mc^{-1} , \quad \Qc \coloneqq C_W^{p-1} \Mc \Mc^T\,,
\eeq
and we use the notation 
\begin{equation}\label{eq:Del_Q}
\Delta_\Qc f\coloneqq\text{tr}(\Qc\nabla^2f)\,,\quad \langle x,y\rangle_\Qc \coloneqq x^T\Qc y\,,\, \quad \forall x,y\in\mathbb{R}^d.
\end{equation}
Taking the real and imaginary parts we arrive at the following equations for $U$ and $\Theta$:
    \begin{align}
    \label{real}{U}_\tau & ={c}_U {U}-( \f{1}{2} z+ \Pc z+ \Vc)\cdot\nabla U+{U}^{p}-{C}_U^{p-1}\gamma U+\mathscr{D}_U\,,\\
\label{imagine}\Theta_\tau & =-( \f{1}{2} z+\Pc z+\Vc)\cdot\nabla\Theta+\delta{U}^{p-1}+\mathscr{D}_\Theta\,, 
\end{align}
where $\Ds_U$ and $\Ds_\Theta$ consists of the viscous terms and the nonlinear quadratic term we define  the viscous terms as follows: 
\bseq\label{eq:vis}
\begin{align}
    \label{real_v}
    \mathscr{D}_U & =\Delta_\Qc U-2\beta\langle\nabla U,\nabla\Theta\rangle_\Qc -U\langle\nabla \Theta,\nabla\Theta\rangle_\Qc -\beta U\Delta_\Qc \Theta\,, \\
    \label{im_v}
\mathscr{D}_\Theta & = \beta\frac{\Delta_\Qc  U}{U}+2\frac{\langle\nabla U,\nabla\Theta\rangle_\Qc }{U}-\beta \langle\nabla \Theta,\nabla\Theta\rangle_\Qc +\Delta_\Qc \Theta\,.
\end{align}
\eseq

We will show that the diffusion and $\Qc, c_W, \Vc, \Pc, H^{p-1}$ are lower order terms. See Remark \ref{rmk1}. Dropping these terms and setting $\pa_{\tau} U = 0$, we obtain the leading order parts of \eqref{real} and \eqref{imagine}:
$$
 -\frac{1}{p-1}\bar{U}- \f{1}{2} \Lambda\bar{U}+\bar{U}^p=0\,,\quad \bar{\Theta}_\tau =- \f{1}{2} \Lambda\bar{\Theta} +\delta{\bar{U}}^{p-1}\,,
 $$
 whose solution are given by the approximate profiles  $(\bar U, \bar \Th)$ defined in \eqref{def:Ubar}, \eqref{ss-}.

  \subsection{Initial rescaling and normalization conditions}\label{sec:nom}
    We will choose some proper initial modulation parameters $H(0), V(0), \Mc(0)$ and the dynamic variables {$c_W, \Vc, \Pc$} such that the perturbation $W$ vanishes to the third order. We denote the following constants 
    \begin{equation} \label{eq:dU_cons}
       {\kappa_0=\bar{U}(0)=(p-1)^{-\frac{1}{p-1}}\,,\quad \kappa_2=\partial_{1}^2\bar{U}(0)=-\frac{2c_p\kappa_0}{(p-1)^2}\,,\quad \kappa_4={\partial_{1}^4\bar{U}(0)}=\frac{12p c^2_p\kappa_0}{(p-1)^4}\,.}
\end{equation}

Given initial data $(u, \th)$ \eqref{drf-nd} satisfying \eqref{eq:ass_init}, we first define $\Mc_0, V_0, H(0), \Th_0, U_0$ using \eqref{eq:init_res}. Then we determine other initial rescalings and initial data using  
\[
V(0) = V_0, \quad \Mc(0) = \Mc_0, 
\quad \RR(0) = \Mc_0^{-1}, 
\quad \Th(z, 0) = \Th_0(z), \quad 
   U(z, 0) = U_0(z)  .
\]
We impose the following normalization conditions in time as
  $$
    U(0,\tau)=\bar{U}(0)=\kappa_0\,,\quad \nabla U(0,\tau)=\nabla \bar{U}(0)=0\,,\quad \nabla^2 U(0,\tau)=\nabla^2 \bar{U}(0)=\kappa_2 I_d\,. 
  $$
From \eqref{eq:init_res2}, the above holds for $\tau =0$. By the ansatz \eqref{drf-nd}, it reduces a dynamical condition in time $$\partial_\tau \nabla^k U(0,\tau)=0\,,\quad k=0,1,2\,,$$ which we can use \eqref{real} to simplify as
\beq\label{ode:v}
\bal
    & {c}_U+\kappa_0^{p-1}-H^{p-1}\gamma +\frac{\mathscr{D}_U(0)}{\kappa_0} =0 , \\
           & \kappa_2\Vc =\nabla \mathscr{D}_U(0), \\
     & ({c}_U- 1 +p\kappa_0^{p-1}-H^{p-1}\gamma)\delta_{ij}-(\Pc_{ij}+\Pc_{ji})+\frac{\partial_{ij}\mathscr{D}_U(0)-\Vc\cdot\nabla\partial_{ij}U(0)}{\kappa_2} =0\,,
      \eal
\eeq
for any indices $i,j$. Notice that the inverse of an upper-triangular matrix is still upper-triangular, and as a consequence $\Pc={\Mc}_\tau \Mc^{-1}$ is upper-triangular. We can further simplify the equations for $c_U$ and $\Pc$ by the ansatz \eqref{per_ansatz} as follows:
\begin{equation}
    \label{cq}c_W=-\frac{\mathscr{D}_U(0)}{\kappa_0} + H^{p-1}\gamma\,,\quad (1+\delta_{ij})\Pc_{ij}=-\frac{\mathscr{D}_U(0)}{\kappa_0}\delta_{ij}+\frac{\partial_{ij}\mathscr{D}_U(0)-\Vc\cdot\nabla\partial_{ij}W(0)}{\kappa_2}\,,\end{equation}
    for any $i
    \leq j$. We will estimate equations \eqref{cq} and \eqref{ode:v} in the next subsection.

   \subsection{ODE for the {modulation parameters}}\label{sec:ode}
   In this subsection, we simplify equations \eqref{cq} and \eqref{ode:v} to derive a leading order ODE. 
   {
     We will treat the perturbations $W, \Phi$ as low-order terms and estimate them in Section \ref{sec:stab}. Denote
 \begin{equation}
\label{ee}
\Ga
=\max_{0\leq k\leq 5, 1\leq l\leq 5}(\|\nabla^k W\|_{\infty},\|\nabla^l\Phi\|_{\infty})\, ,
\quad \cE_0 = |\Qc| (\Ga + \Ga^4) + H^{p-1}. 
 \end{equation} 
Clearly, we have $ \Ga^i |Q| \les \cE_0, 1 \leq i \leq 4 $. We use $\cE_0$ to track some lower order terms.
 }
 \begin{lemma}\label{lem1}
     We have the following estimates for the modulation parameters:
     \begin{equation}
         \label{norm_con}
           c_W=\frac{2(1-\beta\delta)}{(p-1)^2}c_p\textup{tr}(\Qc)+ \Oc(\cE_0)\,,\quad \Vc = \Oc( \cE_0)\,,
     \end{equation}
and
     \begin{equation}
         \label{norm_p}
         \Pc= \Oc(|\Qc|(1+ \Ga^4)+H^{p-1})\,, \quad
         {\Qc}_\tau=-(\Qc_u+\frac{1}{2}\Qc_d)\Qc-\Qc(\Qc_u^T+\frac{1}{2}\Qc_d)+ \Oc(\Ec_0|\Qc|),
     \end{equation}
     where $\Qc_u, \Qc_d$ are the strictly upper part and diagonal part of $\Qc$.

    
 \end{lemma}
 {
    \begin{remark}\label{rmk1}
    Formally to the leading order when we take the trace, we have 
    $$
    \textup{tr}(\Qc)_\tau\approx-\textup{tr}(\Qc^2)\,,\quad\textup{tr}(\Qc^{-1})_\tau\approx-d\,.
    $$ 
    Recall that $\Qc=C_W^{p-1}\Mc\Mc^T$ is positive. 
    We can estimate $\tr(\Qc), \tr(\Qc^{-1})$ and obtain 
    $\Qc\approx \tau^{-1} \textup{I}_d$. Therefore $\Qc, c_W, \Vc, \Pc$ are indeed small and the viscous terms can be treated perturbatively. We will make this heuristic rigorous by choosing $H(0)=C_W(0)$ small; see Section \ref{sec:boot}. Although we allow anisotropicity in the initial data, the profile will converge to a isotropic one, i.e. $\Qc$ will converge to a diagonal matrix.
    \end{remark}
    }
 \begin{proof}
 Notice that by \eqref{real_v}, we have 
\begin{equation}
\label{du_exp}\mathscr{D}_U=\text{tr}(\Qc S_1)\,,\quad S_1=\nabla^2 U-2\beta\nabla U\nabla\Theta^T-U\nabla \Theta\nabla\Theta^T-\beta U\nabla^2\Theta\,.
\end{equation} We can simplify \eqref{cq} by an asymptotic expansion of $S_1$ near the origin up to the second order.
\[
\bal
\bar{U} & =\kappa_0+\frac{\kappa_2}{2}|z|^2+\frac{\kappa_4}{24}|z|^4+O(|z|^6)\,,\quad \nabla\bar{\Theta}=\frac{\delta}{\bar{U}}{\nabla\bar{U}}=\delta\frac{\kappa_2}{\kappa_0}z+O(|z|^3)\,, \\
\nabla^2\bar{U}& =\kappa_2I_d+\frac{\kappa_4}{6}|z|^2I_d+\frac{\kappa_4}{3}zz^T+O(|z|^4)\,, \\
\nabla^2\bar{\Theta} & =\frac{\delta}{\bar{U}}{\nabla^2\bar{U}}-\frac{\delta}{\bar{U}^2}{\nabla\bar{U}\nabla\bar{U}^T}=\delta(\frac{\kappa_2}{\kappa_0}I_d+(\frac{\kappa_4}{6\kappa_0}-\frac{\kappa_2^2}{2\kappa_0^2})(|z|^2I_d+2zz^T))+O(|z|^4)\,. 
\eal
\]
{Decomposing $U = \bar U + W, \Th = \bar \Th + \Phi$ \eqref{per_ansatz} and using $\Ga$ \eqref{ee} to control the perturbation $W, \Phi$, we expand}
$$\begin{aligned}
    S_1&=\kappa_2(1-\beta\delta)I_d+(\frac{\kappa_4}{3}+(-2\beta-\delta)\delta\frac{\kappa_2^2}{\kappa_0}-\beta\delta(\frac{\kappa_4}{3}-\frac{\kappa_2^2}{\kappa_0}))zz^T\\&+(\frac{\kappa_4}{6}-\beta\delta(\frac{\kappa_2^2}{2\kappa_0}+\frac{\kappa_4}{6}-\frac{\kappa_2^2}{2\kappa_0}))|z|^2I_d+O(|z|^4+  \Ga + \Ga^3 )\,.
\end{aligned}$$
{The estimates for $\na^i S_1$ are similar. We have chosen $\Ga$ \eqref{ee} to control $\na^k W, \na^{k+1} \Phi$ with high enough order $k$. In particular, for an error term $I$ in $S_1$ bounded by $|z|^4 + \Ga + \Ga^3$, e.g.,$(U - \bar U) \na^2 (\Theta - \bar \Theta)$, we have 
\[
 |\na^i I| \les |z|^{4-i} + \Ga + \Ga^3, \ i = 1, 2.
\]  
Since we only need the expression at $z=0$, the error term $O(|z|^j), j \geq 1$ vanishes in the following derivations. For this reason, we do not track the constant associated with $|z|^j$. 
}

As a consequence, we have the expressions for derivatives of $\mathscr{D}_U$ as 
\[
\bal
 \mathscr{D}_U(0) & = \kappa_2(1-\beta\delta)\text{tr}(\Qc)+ {O(\cE_0)}
\,,\quad \nabla \mathscr{D}_U(0)= { O(\cE_0)} \,, \\
 \partial_{ij}\mathscr{D}_U(0) & = \delta_{ij}(1-\beta\delta)\frac{\kappa_4}{3}\text{tr}(\Qc)+2((1-\beta\delta)\frac{\kappa_4}{3}-(\beta+\delta)\delta\frac{\kappa_2^2}{\kappa_0})\Qc_{ij}+ { O(\cE_0) }\, .
\eal
\]

We plug the estimates into \eqref{cq} and \eqref{ode:v} and get 
\[
\bal
  & c_W=-\frac{\kappa_2}{\kappa_0}(1-\beta\delta)\text{tr}(\Qc)+ {O(\cE_0)}\,,\quad v={ O(\cE_0) }\,, \\
   & \Pc=\mathbf{T}^u\Big[\frac{\kappa_4}{6\kappa_2}(1-\beta\delta)\textup{tr}(\Qc)I_d+((1-\beta\delta)\frac{\kappa_4}{3\kappa_2}-\big(\beta+\delta)\delta\frac{\kappa_2}{\kappa_0}\big)\big(2\Qc-\textup{diag}(\Qc)\big)\Big] \\
   & \qquad \qquad -\mathbf{T}^u(\frac{\kappa_2}{2\kappa_0}(1-\beta\delta)\textup{tr}(\Qc)I_d)+ {O(\cE_0)} \,,
\eal
\]
where we recall that $\mathbf{T}^u$ is the upper-triangular part of the matrix.   Notice that by \eqref{eq:dU_cons}, we have the relationship $$\frac{\kappa_4}{6\kappa_2}=p\frac{\kappa_2}{2\kappa_0}\,,\quad(1-\beta\delta)\frac{\kappa_4}{3\kappa_2}-(\beta+\delta)\delta\frac{\kappa_2}{\kappa_0}=(p-\delta^2-\beta\delta(1+p))\frac{\kappa_2}{\kappa_0}=-\frac{1}{2}\,.$$
Therefore, we collect 
\begin{equation}
    \label{P} \Pc+\frac{p-1}{2}c_WI_d=-\mathbf{T}^u(\Qc-\frac{1}{2}\textup{diag}(\Qc))+ O(\cE_0)\,.
\end{equation}

Recall that ${\Mc}_\tau=\Pc \Mc$ by definition, and we can compute 
$$
{\Qc}_\tau=(p-1)c_W \Qc+C_W^{p-1}(\Pc \Mc \Mc^T+\Mc \Mc^T\Pc^T)=\big(\Pc+\frac{p-1}{2}c_WI_d\big)\Qc+\Qc\big(\Pc+\frac{p-1}{2}c_WI_d\big)^T\,.
$$
If we decompose $Q$ into the strictly upper, lower, and diagonal parts as $\Qc=\Qc_u+\Qc_u^T+\Qc_d$, then we can simplify 
$$
{\Qc}_\tau=-(\Qc_u+\frac{1}{2}\Qc_d)\Qc-\Qc(\Qc_u^T+\frac{1}{2}\Qc_d)+ \cE_\Qc  \Qc + \Qc \cE_\Qc ^T,
\quad \cE_\Qc  = O(\cE_0).
$$ 
and we conclude the proof of the lemma.
\end{proof}
\begin{remark}
{
In the above ODE of $\Qc$, we get $-1$ for the coefficient of $\Qc^2$ since we have 
normalized the profile \eqref{ss-}. The factor $p-\delta^2-\beta\delta(1+p) > 0$ (corresponding to the subcritical case) appears in the constant $c_p$. 
    For the critical case $$p-\delta^2-\beta\delta(1+p)=0\,,$$ without such a normalization}, 
 we can see that if we do something similar, the coefficient 
of $\Qc^2$ will be zero.  We could keep track of the next order terms of size $\Ga$ to derive a system  of size $|\Qc|^3$. This can potentially help us establish a result similar to \cite{DNZmems23} but we do not pursue it here.
\end{remark}

 \section{Stability analysis and finite time blowup}\label{sec:stab}

{
 In this section, we perform stability analysis and establish nonlinear stability of the perturbation around the approximate steady state following the ideas and strategy outlined in 
 Section \ref{sec:idea}. 

We linearize \eqref{real} and \eqref{imagine} around the approximate profile as in ansatz \eqref{per_ansatz}  and obtain the equations of the perturbations as follows:
\begin{equation}
    \label{realp}{W}_\tau=\mathscr{L}_{\bar{U}}W+\mathscr{F}_U+\mathscr{N}_U+\mathscr{D}_U\,,\end{equation}     
    \begin{equation}
\label{imaginep}\Phi_\tau =- \f{1}{2} \Lambda \Phi+\mathscr{F}_\Theta+\mathscr{N}_\Theta+\mathscr{D}_\Theta\,,\end{equation}
{where we recall from \eqref{real_v} and \eqref{im_v} the definition of $\Ds_U$ and $\Ds_\Theta$, and define the linear, residue, and nonlinear parts respectively as}
\beq\label{eq:lin}
\bal
 \mathscr{L}_{\bar{U}}W & =c_{\bar{U}} W- \f{1}{2} \Lambda W+p\bar{U}^{p-1}W\,,
 \quad \Lambda=z\cdot\nabla , \\
 \mathscr{F}_U & =c_W U-(\Pc z+\Vc)\cdot\nabla U-{C}_U^{p-1}\gamma U\,,\quad\mathscr{N}_U=(\bar{U}+W)^p-\bar{U}^p-p\bar{U}^{p-1}W\,,  \\
\mathscr{F}_\Theta &=-(\Pc z+\Vc)\cdot\nabla {\Theta}\,,\quad\mathscr{N}_\Theta=\delta((\bar{U}+W)^{p-1}-\bar{U}^{p-1})\,. 
\eal
\eeq
We will group the terms by integrability:  $\mathscr{L}_U$ and $\mathscr{N}_U$ vanish to the third order at the origin. Recall that by Lemma \ref{lem1} and Remark \ref{rmk1}, we know that $\Qc, c_W, \Vc, \Pc$ are small. The viscous terms $\mathscr{D}_U$, $\mathscr{D}_\Theta$ are small of order $|\Qc|$ with a typical size of $1/\tau$. 
Also obviously $H=C_We^{-\tau/(p-1)}$ is small.

 We define the weighted $H^k$ energy  as follows 
\begin{equation}\label{norm:Ek}
\bal
    & E_k^2=(|\nabla^k W|^2,\rho_k)\,,\quad 0\leq k\leq K\,, \quad 
    F_k^2=(|\nabla^k \Phi|^2,\mathring{\rho}_k)\,,\quad 0< k\leq K\,.
    \eal
\end{equation}
Our goal is to prove the following nonlinear stability results.
  \begin{theorem}\label{thm:stab}
  Denote $E_{\Qc} = \tr(\Qc)$. There exists $ 0< E_* < 1$ sufficiently small and $\mu_3 > 0$, such that for any initial perturbation satisfying 
  \beq\label{eq:thm_stab_ass}
U \rho > 2 C_b, \quad E(0) < E_*, \quad E_{\Qc}  < E_*, \quad H^{p-1}(0) < E_*,
  \eeq
we have the following estimates for all $\tau > 0$
\beq\label{eq:non_stab}
\bal
 & E_\Qc (\tau)  \leq  \min( 2 E_\Qc (0), 4 d / \tau ), \quad U(\tau) \rho > C_b , 
 \quad H^{p-1}(\tau) < H^{p-1}(0) e^{-\tau / 2} , \\
& E(\tau) \leq e^{-\lam \tau/2} (E(0) + \mu_3 H^{p-1}(0) +\mu_3  E_\Qc (0)) + \mu_3  \min(E_\Qc (0), 1 / \tau).
\eal
\eeq
Moreover, the boostrap assumptions \ref{asss}, \ref{ass:lower}, \ref{ass:refine} (introduced below) hold for all $\tau > 0$. 
  \end{theorem}

Note that the parameters $\mu_1, \mu_2$ will be introduced in the proof in Section \ref{sec:boot}.

We impose the following weak bootstrap assumptions for nonlinear estimates. To control $1/U$, which will appears in the estimate of $\msN_{U}, \msN_{\Th}$, we impose a lower bound on $U$, which is almost comparable with $\bar{U}$, up to a small power. To simplify the notations in the nonlinear estimates and to ensure that $\RR, \Mc$ are invertible \eqref{eq:Pv}, \eqref{per_ansatz}, we impose the following weak assumptions on the energy $E_i, F_j$ and $\det(\Qc)$.
\begin{assumption}
\label{asss}
Let  $\e_2$ be defined in \eqref{eps}. We impose the following bootstrap assumptions
\begin{align}
    & U\geq C_b\bar{U}^{1+\epsilon_2} , \quad  C_b = \min_{|z|\leq 1} \bar U^{-\e_2}(z) / 4 > 0, \label{eq:boot_U}  \\
&  \max(E_K, E_0, F_K, F_1) \leq 1 ,
\quad \quad \det(\Qc) > 0.\label{eq:boot_EE}
\end{align}
\end{assumption}

Below, we will first establish some functional inequalities in Section \ref{sec:ineq}. 
We will start with the $L^2$ analysis of perturbations $W$ and $\nabla\Phi$ in Sections \ref{l2sec}, \ref{h1sec}, and then build higher-order estimates in Section \ref{hksec}. 
We obtain a lower bound of the amplitude $U$ via the maximal principle in Section \ref{sec:linfty}, 
inspired by \cite{chen2021HL}.
Then we close the nonlinear estimates and prove Theorem \ref{thm:stab} via a bootstrap argument in Section \ref{sec:boot}.

}

\subsection{Functional inequalities}\label{sec:ineq}

In this section, we establish a few functional inequalities, which will be used to 
estimate the decay of the solution and close the nonlinear estimates. We introduce the following norms 
\beq\label{norm:Hk}
|| f ||_{\dot \cH^k} := || \na^k f g_k^{1/2} ||_{L^2}, \quad 
 || f ||_{\cH^k} := || f ||_{\dot \cH^k}  + || f ||_{\dot \cH^0}, \quad 
 g_k = \la z \ra^{ - 2 \s - \e - d + 2k} .
\eeq
By definition of $\rho_k$ \eqref{wg:rho} and \eqref{norm:Ek} for $E_k$, we have 
\beq\label{eq:norm_comp}
g_k \les \rho_k, \quad 
|| f ||_{\dot \cH^k} \les  E_k, 
\quad  || f ||_{ \cH^k} \les E_0 + E_k.
\eeq
We define the low-order terms $\cE_1$, $\cE_2$ that we later show to be small: 
\beq\label{eq:lot}
\bal
\cE_1 & :=   |\Qc|+H^{p-1}+ \sum_{ i \leq K - 1} E_i + \sum_{ 1\leq j \leq  K-1}  F_j 
+ ( F_K+F_1+E_K+E_0)^2,  \\
\cE_2  & := |\Qc |+H^{p-1} .
\eal
\eeq
We treat $\cE_1$  as a lower order term since it either contains nonlinear terms or energy with order lower than $E_K, F_K$. Note that $\cE_2$ is the low-order term of order $P$ in Lemma \ref{lem1} by asuming \eqref{eq:boot_EE}, which implies $\Ga \les 1$. See \eqref{eq:EE_embed}.

Following Lemma C.4 in \cite{chen2024Euler}, we have the following weighted interpolation and embedding inequalities.

\begin{proposition}[Interpolation ]
    \label{prop:inter}
    Let $\s = - \f{2}{p-1}$ and $\e$ be the constants defined in \eqref{eq:dU_cons}, \eqref{eps}. For any $\mu>0$, there exists a constant $C(\mu)$, such that the following interpolation  inequalities hold:    
\begin{subequations}\label{eq:inte}
\begin{align}
        E_k &  \leq \mu E_l+C(\mu)E_0\,,\quad \forall 0\leq k< l \leq K\,, \label{eq:inte:1}  \\
     F_k &  \leq \mu F_l+ C( \mu)F_1\,,\quad \forall 1\leq k<l \leq K\,,  \label{eq:inte:2}  \\
     || f  ||_{\dot \cH^k} & \leq \mu || f||_{\dot \cH^l} + C(\mu) || f ||_{\cH^0} , \quad \forall 0\leq k< l \leq K .     \label{eq:inte:Hk}  
\end{align}
\end{subequations}
Moreover, we have the following embedding 
\bseq\label{eq:embed_linf}
\begin{align}
     |\nabla^l W| & \les  \la z \ra^{-l +\s + \e/ 2} \min( E_{l+d} , || W||_{\cH^{l+d}} )   , \ 
\forall 0\leq l\leq K - d \ \label{eq:inte:3} , \\
 | \nabla^l\Phi | & \les   \langle z\rangle^{-l + 1/2} F_{l + d}  ,\quad \forall 1\leq l\leq K - d-1\, \label{eq:inte:4}  . 
 \end{align}
 \eseq

As a result, for $k < l$, we have the following product rule for $i\leq n, j \leq m$ with $ i+ d \leq n$, or $j + d \leq m$,
\beq\label{eq:prod}
|| \na^i F \na^j G g_{i+j}^{1/2} ||_{L^2} 
\les || F ||_{\cH^{n}} || G ||_{\cH^m} .
\eeq

Finally, assuming \eqref{eq:boot_EE}, for $0\leq i\leq K, 1 \leq j \leq K$ and $\Ga$ defined in \eqref{eq:EE_ODE2}, we have 
\beq\label{eq:EE_embed}
E_i \les 1, \quad F_j \les 1, \quad  \Ga \les 1.
\eeq


\end{proposition}
{Since $K$ \eqref{K} is absolute, $l, k\leq K$, and the parameters $c_0, c_1$ in the weights \eqref{wg:rho} and norms depend on absolute constants $K, \e$ \eqref{eps}, we only need to track the constants related to $\mu$.}

\begin{proof}

\textbf{(a) Interpolation inequalities.}
To prove \eqref{eq:inte}, we use integration by parts. For \eqref{eq:inte:1}, we compute for $K>k>0$ that
    $$E_k^2=-\sum_i\int(\partial^2_i\nabla^{k-1}W\cdot\nabla^{k-1}W\rho_k+\partial_i\nabla^{k-1}W\cdot\nabla^{k-1}W\partial_i\rho_k)\,.$$
    Notice that the weights \eqref{wg:rho} satisfy 
    $$
    \rho_k^2\lesssim  \rho_{k+1}\rho_{k-1}\,,\quad(\partial_i\rho_k)^2\lesssim \rho_k\rho_{k-1}\,.
    $$
    Combined with a Cauchy-Schwarz inequality, we obtain  
    $$E_k^2\lesssim  E_{k-1}(E_{k}+E_{k+1})\,.
    $$
    {Since $\e$ only depends on $K$ \eqref{eps},}  by a weighted AM-GM inequality, for any $\mu>0$, we have 
    $$E_k^2\leq C(\mu)E^2_{k-1}+\mu E^2_{k+1}\,.
    $$
    From here, to conclude the first inequality, since $\mu>0$ is arbitrary, we only need to that show it holds for $k= l-1$, which we can combine the above estimates for $k = 1, 2, .., l-1$ to establish.

    The proof of \eqref{eq:inte:Hk} follows from the same argument.

    For the second inequality \eqref{eq:inte:2}, we can repeat the same procedure to conclude, provided that the weights $\mathring{\rho}_k$ satisfy the same inequalities for $K>k>1$ \eqref{wg:rho}:
    $$\mathring{\rho}_k^2\lesssim  \mathring{\rho}_{k+1}\mathring{\rho}_{k-1}\,,\quad(\partial_i\mathring{\rho}_k)^2\lesssim  \mathring{\rho}_k\mathring{\rho}_{k-1}\,.$$
    When $k+1< K$, this is obvious. For $k=K-1$, we only need to show  
    $$
    \langle z\rangle^{2K-1-d}\lesssim\mathring{\rho}_{K}\approx U^2
        \langle z\rangle^{-2 \s+2K-\epsilon-d}\,,
    $$
    which is true, since by the choice of $\epsilon, \e_2$ \eqref{eps} and the bootstrap Assumption \ref{asss} we have 
    $$
    U\geq C_b\bar{U}^{1+\epsilon_2}\approx \la z \ra^{\s - \e/2} , \quad
    2  \e < 1.
    $$
    
\noindent{\textbf{(b) Embedding \eqref{eq:embed_linf}.}
 To prove the $L^\infty$ estimates \eqref{eq:embed_linf},  one can proceed as in \cite{HNWarXiv24} and invoke the weighted Morrey-type inequality. Below, we present a simpler proof. By a density argument, we can assume that $W \in C_c^{\infty}$.  
Without loss of generality, we fix $z  \in \Rb^d$ with $z_i \geq 0$ and estimate $\na^l W (z)$. 
Consider the region $\Om(z) = \{ y  \in \Rb^d, y_i \geq z_i \}$. We have $|y| \geq |z|$ for any $y \in \Om(z)$. Denote $\d = - 2 \s - \epsilon - d > -d $ \eqref{eps}.
We have 
\[
\bal
 |\na^l W(Z) | \les_l  \int_{\Om(z)} | \pa_1 \pa_2 .. \pa_d \na^l W (y) | d y
& \les_l || \la y \ra^{ l + d + \d / 2}  \na^{l + d} W ||_{L^2}
 \B( \int_{|y| \geq |z|} \la y \ra^{ - 2 l  -  2 d - \d } d y \B)^{1/2} . \\
 \eal
\]
For $ 0 \leq l \leq K - d - 1$, using $\rho_{ l+ d} \gtr \la y \ra^{ 2 (l+d) + \d}$ \eqref{wg:rho},
$-2 l - d - \d - 1 < -1$, and the first inequality \eqref{eq:inte:1} with $k = l + d < K$, for any $\mu > 0$, we further obtain 
\[
 |\na^l W(y) | \les_K E_{l+d}
 (\int_{ R \geq |y|}  \la R \ra^{-2 l - 2 d - \d } R^{d-1} d R)^{1/2}
\les_K 
E_{l + d}
 \la y \ra^{- l - (d + \d) /2 } .
\]
Rearranging the power on both sides, we prove \eqref{eq:inte:3}.

The proof of \eqref{eq:inte:4} with $ 1 \leq l \leq K- d - 1$ is similar by replacing $W$ by $\na \Phi$  
in the above argument 
and using $\mathring \rho_{ l+ d} \gtr \la y \ra^{ 2 (l+d) + \d_2} $  for $ \d_2 = -d - 1$ 
\eqref{wg:rho} and $ - 2l - d - \d_2 - 1 = - 2 l< -1$ for $l \geq 1$. 
}

\noindent \textbf{(c) Other inequalities.}
For \eqref{eq:prod}, without loss of generality, we assume that $i+ d \leq n$. Using $g_{i+j}^{1/2} \les \la z \ra^i g_j^{1/2}$, $\s + \f{\e}{2} <0$ \eqref{eq:inte:Hk}, we prove 
\[ 
 || \na^i F \na^j G g_{i+j}^{1/2} ||_{L^2}
 \les || F||_{\cH^{i+d}} || \na^j G g_j^{1/2}||_{L^2}
 \les  || F||_{\cH^{n}} || G||_{\cH^m}.
\]
The inequalities \eqref{eq:EE_embed} follow from \eqref{eq:inte:1}, \eqref{eq:inte:2} and the assumption \eqref{eq:boot_EE}.
\end{proof}

{

Combining Assumption \ref{asss} and the decay estimates \eqref{eq:inte:3}, \eqref{eq:inte:4}, we have the following estimates.

\begin{cor}\label{cor:decay}

Denote $W = \bar U - U$. Under the Assumption \ref{asss}, for any $\al \in [0, 1]$ and any $q \in \Rb$ we have 
\beq\label{eq:U_decay}
\bal
  \la z \ra^{ \s - \e / 2} & \les U \les \la z \ra^{\s + \e / 2}, \\
   |(\bar U + \al W)(z)|^q & \les C(|q|) \la z \ra^{ \s q +|\e q| / 2 } , \\
   || \bar U + \al W ||_{\cH^K}, || U ||_{\cH^K} & \les 1 , 
  \eal
\eeq
where $\s, \e$ are defined in \eqref{eq:dU_cons}. As a result, 
for $i+j \leq K$, we have the following estimates for the weights
\beq\label{eq:comp_wg}
\bal
& \mr \rho_i  \les \rho_i \la z \ra^{\s + \e/2-1/2},  i\leq K-1,\quad 
\mr \rho_K \les \rho_K \la z \ra^{\s + \e/2},  \\
& \rho_{i+j} \les \rho_i \la z \ra^{2j}, \quad \mr \rho_{i+j} \les \mr \rho_i
 \la z \ra^{2j} .
 \eal
\eeq

\end{cor}

\begin{proof}

\textbf{(a) Estimate of $U$.} By definition of $\e_2$ \eqref{eps}, we get $- \f{2}{p-1}( 1 +  \e_2)  = - \f{2}{p - 1} -  \f{\e}{2} $. Thus, under Assumption \ref{asss}, we yield 
\[
 U \gtr \bar U^{1 + \e_2} \gtr \la z \ra^{- \f{2}{p-1} - \f{\e}{2} } .
\]


Since $\bar U + \al W = \bar U + \al( U - \bar U) 
= \al U + (1- \al) \bar U$, which is between $U, \bar U$, and $\bar U, U > 0$, using \eqref{eq:inte:3}, $K > d$ \eqref{K} and the above estimate, we prove
\[
\bal
|(\bar U + \al W)| & \les \bar U + U \les (1 + E_{K} + E_0)  \la z \ra^{  \s + \e/2  } ,\quad
|(\bar U + \al W)|^{-1}  \les \min(\bar U, U )^{-1} \les \la z \ra^{ -\s  + \e / 2}. 
\eal
\]

The first estimate with $\al = 1$ implies the upper bound for $U$ in \eqref{eq:U_decay}. Raising the above estimates to $|q|$-th power proves the second estimate in \eqref{eq:U_decay}. 

For the last estimate in \eqref{eq:U_decay}, using triangle inequality, $|\na^i \bar U |
\les \la z \ra^{\s -i }$, and \eqref{eq:norm_comp}, we prove 
\[
 || \bar U + \al W ||_{\cH^K}
 +  || \bar U +  W ||_{\cH^K}
\les || \bar U ||_{\cH^K} + || W ||_{\cH^K} 
\les 1 + E_0 + E_K \les 1.
\]
\textbf{(b) Estimate of weights.}  We consider \eqref{eq:comp_wg}. Using $U  \les \la z \ra^{\s + \e/2}$, clearly, we have 
\[
\bal
& |z| \leq 1: \  \mr \rho_i \les \rho_i \les \rho_i \la z \ra^{ 2 \s + \e - 1} , \\ 
& |z| \geq 1: \ \mr \rho_i \les |z|^{2k-1-d}  \les \rho_i \la z \ra^{ 2 \s + \e-1} ,  i \leq K-1,
\quad \mr \rho_K \les U^2 \rho_K 
\les \rho_K \la z \ra^{ 2 \e + \e}.
\eal
\]
For $(f, \tau) = (\rho, - 2 \s - \e )$ or $( \mr \rho,  -1)$ and $i+j < K$, from the definition of $ f_i$ \eqref{wg:rho}, we have
\[
\bal
 f_{i+j} \les f_i \les f_i \la z \ra^{2j}, \quad |z| \leq 1,
f_{i+j} \approx |z|^{2 i + 2 j + \tau - d} 
\les |z|^{2i + \tau - d} \la z \ra^{2j} \les f_i \la z \ra^{2j}, |z| \geq 1.
 \eal
\]
For $i+j=K$, the above estimate still holds for $f = \rho$. For $ \mr \rho_i$ and $\mr \rho_K$, using $ U \les \la z \ra^{\s + \e/2}$, we obtain 
\[
\mr \rho_K \les \la z \ra^{2 \s + \e  } \rho_K
\les \la z \ra^{2K - 1 - d}
= \la z \ra^{2i-1 - d} \la z \ra^{2j} 
\les \mr \rho_i \la z \ra^{2 j}.
\]
We complete the proof of \eqref{eq:comp_wg}.
\end{proof}

\begin{proposition}\label{prop:Hk}

Suppose that \eqref{eq:boot_EE} holds true. For $ 0 \leq k \leq K, j_1 + j_2\leq k$ and any $\al \in [0, 1]$, denote 
\[
V = \bar U + \al W , \quad I_{(j_1, j_2)} =  \na^{j_1} W_1 \cdot \na^{ j_2} W_2 \na^{k-j_1-j_2} V^{p-2}.
\]
 We have the following product estimates 
\beq\label{eq:Hk_est1}  
|| I_{(j_1, j_2 )} ||_{g_k}  \les || W_1 ||_{ \cH^{ \max{ (j_1, K-1)} } } || W_2 ||_{ \cH^{\max(j_2, K-1)} } , 
\eeq
and 
\begin{subequations}
\begin{align}
  || W_1 W_2 V^{p-2} ||_{\cH^K} & \les ( || W_1 ||_{\cH^{K-1}} || W_2 ||_{\cH^K} 
  +  || W_1 ||_{\cH^{K}} || W_2 ||_{\cH^{K-1}} )  , \label{eq:Hk_est2_Hk} 
   \\
 ||W_1 W_2 V^{p-2} ||_{\rho_0} 
& \les || W_2 ||_{\rho_0} || W_1 ||_{\cH^d} 
\label{eq:Hk_est2_L2} . 
\end{align}
\end{subequations}

Moreover, we have 
\begin{align}
    || \la z \ra^{\s + \e/2}  \na^{l +1} U \na^{m}( U^{-1}) ||_{g_k} & \les || U ||_{\cH^{\max(l+1, m , k)}}, 
       \quad l + m = k \leq K , \label{eq:Hk_estU} \\
       | \na^{l +1} U \cdot \na^{m}( U^{-1}) | 
       & \les \la z \ra^{-l-m}  , \quad l , m \leq K-1 - d, 
       \label{eq:Hk_estU_linf}
 \end{align}
\end{proposition}

\begin{proof}
A direct computation yields
\beq\label{eq:Hk_est_pf1}
\bal
I_{j_1, j_2}
 \leq \sum_{2 \leq q \leq k+2,} \sum_{ \sum_{l=1}^q j_l = k, j_l \geq 0} I_{\vec j, q}, \quad 
 I_{\vec j, q} = |V|^{p-q}  \cdot |\na^{j_1} W_1 | |\na^{j_2} W_2|  \prod_{l=3}^q  |\na^{j_l} V| .
 \eal
\eeq

For a fixed $(j, q)$, we denote 
\[
J_1 =  W_1, \quad J_2 =  W_2,
\quad J_l =  V, \quad l\geq 3,
\quad i = \arg\max_{l\leq q} j_l.
\]
If there are more than one indices $a$ with $ j_a =  \arg\max_{l\leq q} j_l$, we just pick one of them.
Clearly, we have $j_l \leq k/2, l \neq  i$ \eqref{K}.
By Proposition \ref{prop:inter}  \eqref{eq:inte:3}, we have
\[
\bal
|\na^{j_l} J_l | \les \la z \ra^{ - j_l + \s + \e/2}  || J_l ||_{\cH^{j_l + d}}
\eal
\]
Applying the above $L^{\infty}$ estimates to $J_l, l\neq i$, and Corollary \ref{cor:decay} for $
V = \bar U + \al W$, we obtain 
\beq\label{eq:Hk_est_pf2}
I_{\vec j, q} \les |\na^{j_i} J_i|  \prod_{l\neq i} \la z \ra^{ - j_l +  \s +  \e / 2}  || J_l ||_{\cH^{j_l + d}}  \la z \ra^{(p-q) \s + |p-q| \e / 2}. 
\eeq
Combining the exponents of the $\la z \ra^{\cdot}$ terms, and using the definitions of $\s = -\f{2}{p-q}$ \eqref{eq:dU_cons}  and $\e$ \eqref{eps}, we yield
\beq\label{eq:Hk_est_pf3}
\bal
\xi & = \sum_{ 1 \leq l \neq i \leq q} ( - j_l + \s +  \e  / 2) + (p-q) \s + |p-q| \e / 2 \\
& = - (k - j_i) + (p-q + q-1) \s + (p+q + q-1) \e / 2
= -(k - j_i) - 2 + ( p + K ) \e <  -(k - j_i).
\eal
\eeq
Since $\rho_k^{1/2} \la z \ra^{ - (k-j_i)} \les \rho_{k-j_i}$ \eqref{wg:rho}, applying weighted $L^2$ bound to $\na^{j_i } J_i$, we further obtain 
\[
I_{\vec j ,q} \les || J_i ||_{\cH^{j_i}} \prod_{l\neq i} || J_l ||_{\cH^{j_l+d}} .
\]
Since $j_l + d\leq k /2 + d \leq K-1 $ for $l \neq i, k \leq K$ , $j_i \leq k$, and $ || V ||_{\cH^K} \les 1 + E_0 + E_K$, we obtain $ j_l + d \leq \max( j_l, K-1)$ and thus 
\[
I_{\vec j ,q} \les || W_1 ||_{ \cH^{\max(j_1, K-1)}} 
|| W_2 ||_{ \cH^{\max(j_2, K-1)}} . 
\]
Using $\max(p-q,0) + q \leq p + K + q$ and summing the estimates of $I_{\vec j, q}$, we conclude the proof of \eqref{eq:Hk_est1}. 

The estimate \eqref{eq:Hk_est2_Hk} follows from summing the estimates of $I_{(j_1, j_2)}$
\eqref{eq:Hk_est1} over $(j_1, j_2) $ with $j_1 + j_2 \leq K$, and using the fact that we have $j_1 \leq K-1$ or $j_2 \leq K-1$. The estimate \eqref{eq:Hk_est2_L2} follows from applying $L^2(\rho_0)$ estimate to $W_2$ and $L^{\infty}$ estimate to $W_1, V$ similar to the above.

For the last estimate \eqref{eq:Hk_estU}, we note that $U = \bar U + W$. Applying the Leibniz rule, we obtain 
\[
|\na^{l+1} U  \na^{m} (U^{-1} ) )| 
\les   \sum_{ 1 \leq q \leq k+1,} \sum_{ \sum_{l=1}^q j_l = k+1 , j_1 \geq 1, j_l \geq 0} T_{\vec j, q}, \quad 
 T_{\vec j, q} = |U|^{-q}  \cdot    \prod_{l=1}^q  |\na^{j_l} U| .
\]
Denote $i = \arg\max_l j_l$. Applying the above estimates of $I_{\vec j, q}$ with $(J_1,.. ,J_l), V, p-q$ replaced by 
$ (U, U,.., U), U, -q$ \eqref{eq:Hk_est_pf1}-\eqref{eq:Hk_est_pf3}, and using $\max(-q, 0) = 0, q \leq K+1 $, we obtain 
\[
|\la z \ra^{\s + \e/ 2} T_{\vec j, q} |
\les |\na^{j_i} U| \la z \ra^{\xi} \prod_{l\neq i} || U ||_{\cH^{j_l + d}}, 
\]
where 
\[
\bal
\xi & =\s + \f{\e}{2} +  \sum_{ 1 \leq l \neq i \leq q} ( - j_l + \s +  \f{\e}{2} ) + (-q) \s + |-q| \f{\e}{2}\\
& = - (k + 1 - j_i) + (-q + q) \s + (q + q) \e / 2
= -(k - j_i) - 1 + ( 1 + K ) \e <  -(k - j_i).
\eal
\]
Using $j_l + d \leq K/2 + d < K, || U||_{\cH^{j_l + d}}  \les || U||_{\cH^K} $ for $l \neq i$, 
$\rho_k^{1/2} \les \rho_{j_i}^{1/2} \la z \ra^{k-j_i}, q, j_i  \leq k+1$, and applying weighted $L^2$ estimate to $\na^{j_i} J_i$, we establish
\[
 ||\la z \ra^{\s + \e/ 2} T_{\vec j, q} ||_{g_k}
\les || U ||_{\cH^{j_i}} (1 + || U ||_{\cH^K})^{q-1}
\les || U ||_{\cH^{k+1}}. 
\]
Combining the estimates for $ T_{\vec j, q} $ with different $\vec j , q$, we  conclude the proof of \eqref{eq:Hk_estU}.

For \eqref{eq:Hk_estU_linf}, denote $k = l + m$. Applying $L^{\infty}$ estimate to each term $\na^{j_l}U$ in $T_{\vec j, q}$ and noting that $j_l \leq \max(l+1, m) \leq K-d$, we prove
\[
T_{j, q} \les \la z \ra^{\xi} 
\les \la z \ra^{-k - 1/2} ,
\]
where we have used 
\[
\xi = \sum_{ 1 \leq l  \leq q} ( - j_l + \s +  \f{\e}{2} ) + (-q) \s + |-q| \f{\e}{2}
= - (k+1) + q \e < - k - 1/2.
\]
We complete the proof.
\end{proof}

}

\subsection{$L^2$ stability analysis of the amplitude}\label{l2sec}

{In this section, we estimate the weighted $L^2$ energy $E_0^2=(W,W\rho_0)\,$ \eqref{norm:Ek}. In the following energy estimates, without specification, we will assume that the bootstrap assumption \ref{asss} holds true. We will show that the following lemma holds.

}
\begin{lemma}[Weighted $L^2$ estimate]\label{lem:l2}
Under the bootstrap assumption \ref{asss}, it holds
\begin{equation}
    \label{l2-collecy}
        \frac{1}{2} {  \f{d}{d \tau } } E_0^2\leq (-\frac{\epsilon}{8}+ C \cE_1)E_0^2+ C \cE_2
        E_0\,.
\end{equation}
or some absolute constant $C>0$. 
\end{lemma}
\begin{proof}
Notice that $W$ vanishes at the origin to the third order so this choice of singular weight induces a well-defined energy. We  have by \eqref{realp} that 
$$\frac{1}{2} { \f{d}{d \tau}} E_0^2=(\mathscr{L}_{\bar{U}}W,W\rho_0)+(\mathscr{N}_U,W\rho_0)+(\mathscr{F}_U+\mathscr{D}_U,W\rho_0)\,.$$

For the leading order linear term, we have via integration by parts 
 that $$(\mathscr{L}_{\bar{U}}W,W\rho_0)=(d_0 W,W\rho_0)\,,$$
where we calculate the damping 
$$\begin{aligned}
    d_0&\coloneqq c_U+\frac{1}{2\rho_0}\nabla\cdot(d_{\bar{U}}z\rho_0)+p\bar{U}^{p-1}=-\frac{1}{p-1}+p\bar{U}^{p-1}+\frac{1}{4\rho_0}\nabla\cdot(z\rho_0)
    \\&=-\frac{1}{p-1}+\frac{p}{p-1+c_p|z|^2}+\frac{(-6+\epsilon)|z|^{-6+\epsilon}+(\frac{4}{p-1}-\epsilon)c_0|z|^{\frac{4}{p-1}-\epsilon}}{4(|z|^{-6+\epsilon}+c_0|z|^{\frac{4}{p-1}-\epsilon})}\leq-\frac{\epsilon}{8}\,.
\end{aligned}$$
 The last inequality amounts to $$\begin{aligned}&4p(|z|^{-6+\epsilon}+c_0|z|^{\frac{4}{p-1}-\epsilon})+((-6-\frac{4}{p-1}+\frac{3\epsilon}{2})|z|^{-6+\epsilon}-\frac{\epsilon}{2}c_0|z|^{\frac{4}{p-1}-\epsilon})(p-1+c_p|z|^2)\\&\leq-(p-1)|z|^{-6+\epsilon}-\frac{\epsilon}{2}c_0c_p|z|^{2+\frac{4}{p-1}-\epsilon}+4pc_0|z|^{\frac{4}{p-1}-\epsilon}
\leq0\,.\end{aligned}$$ 
The last inequality is implied by a weighted AM-GM inequality provided that \begin{equation}\label{c0}
    \Big(\frac{\epsilon c_0c_p}{2(6+\frac{4}{p-1}-2\epsilon)}\Big)^{6+\frac{4}{p-1}-2\epsilon}(\frac{p-1}{2})^2\geq \Big(\frac{4pc_0}{8+\frac{4}{p-1}-2\epsilon}\Big)^{8+\frac{4}{p-1}-2\epsilon}\,.
\end{equation}
Notice that $\epsilon\leq 1/2$ is fixed to be small. We can choose a sufficiently small constant $c_0 > 0$ such that we can conclude the linear estimate\begin{equation}
    \label{damp-l2}\big(\mathscr{L}_{\bar{U}}W,W\rho_0\big)\leq-\frac{\epsilon}{8}E_0^2\,.
\end{equation} 

The nonlinear estimate is more subtle due to the general nonlinearity $p$. We use Taylor's expansion or Newton-Leibniz's formula twice to derive 
\beq\label{eq:Newton}
\mathscr{N}_U=W^2p(p-1)\int_0^1(1-\alpha)(\bar{U}+\alpha W)^{p-2}d\alpha\,.
\eeq
{Using Proposition \ref{prop:Hk} \eqref{eq:Hk_est2_L2} with $(W_1, W_2) = (W, W)$ and $\cE_1$ defined in \eqref{eq:lot}, we obtain 
\[
 || W^2 (\bar{U}+\alpha W)^{p-2} ||_{\rho_0}
 \les || W||_{\cH^d} || W ||_{\rho_0}
\les (E_0 + E_d) E_0
\les \cE_1 E_0,
\]
which implies 
}



\begin{equation}
    \label{non-l2}
    | (\mathscr{N}_U,W\rho_0) | \leq\|\mathscr{N}_U\|_{\rho_0}E_0 
    \les \cE_1 || W ||_{\rho_0} E_0 \les 
        \cE_1 E_0^2\,.
\end{equation} 

Finally, we estimate the viscous and residue terms together. We group the terms to make them integrable. Consider a fixed 1D smooth cutoff function $\chi$ such that it equals $1$ in $[-1,1]$ and $0$ outside of $[-2,2]$. We use the notation $\widetilde{f}$ to denote functions only differing from $f$ near the origin, where they equal the residue of $f$ when expanded until its second-order Taylor's expansion at the origin, via the cutoff function $\chi$.  For illustrative purposes, we will explicitly write down   $\widetilde{f}$ by the expansions 
$$f=(f(0)+z^T\nabla f(0)+\frac{1}{2}z^T\nabla^2f(0)z)\chi(|z|)+\widetilde{f}\,,$$
where $\widetilde{f}$ vanishes to the third order at the origin.  By the choice of 
{the modulation parameters} in \eqref{ode:v}, it's easy to see that\footnote{Note that $\tilde U$ does not denote the perturbation.} 
\begin{equation} \label{reor}\mathscr{F}_U+\mathscr{D}_U=c_W\widetilde{U}-\cP z\cdot\nabla\widetilde{U}-v\cdot\widetilde{\nabla{U}}-H^{p-1}\gamma\widetilde{U}+\widetilde{\mathscr{D}_U}\coloneqq\widetilde{\mathscr{F}_U}+\widetilde{\mathscr{D}_U}\,.
\end{equation}
Each of the terms in $\widetilde{\mathscr{F}_U}$ vanishes to the third order at the origin. Notice that $\rho_k=\rho_0 |z|^{2k}$, for $k=1,2,3$. Recall the definition of $\cE_1, \cE_2$ from \eqref{eq:lot}. By Lemma \ref{lem1}, we have that 
$$
\|\widetilde{\mathscr{F}_U}\|_{\rho_0}\les {\cE_2} (1+E_0+E_1+\|\widetilde{\nabla{W}}\|_{\rho_0})\,.
$$
We can decompose the integral region into the near field $I=[0,1]^d$ and the rest of the outer region $I^c$ to estimate 
$$\|\widetilde{\nabla{W}}\|_{\rho_0}\lesssim(\int_{z\in I}|z|^{\epsilon-d})^{1/2}\sup_{z\in I}|\widetilde{\nabla{W}}/|z|^3|+\|\widetilde{\nabla{W}}|z|\|_{\rho_0}\les   \Ga +E_1\,.
$$
Combined with  Proposition \ref{prop:inter}, we have the residue estimate
\begin{equation}
    \label{res-l2}
      { | (\widetilde{\mathscr{F}_U},W\rho_0) | \les \cE_2 E_0\,.}
\end{equation} 

For the viscous term, we notice as in \eqref{du_exp}, we can write $$\widetilde{\mathscr{D}_U}=\text{tr}(\Qc\widetilde{S_1})\,,\quad S_1=S_{11}+S_{12}+S_{13}+S_{14}\,.$$
We estimate the four terms respectively. We compute
$$ 
 \|\widetilde{S_{11}}\|_{\rho_0}=\|\widetilde{\nabla^2U}\|_{\rho_0}\les 1 +\|\widetilde{\nabla^2W}\|_{\rho_0}\les 1+ \Ga +E_2 \,,
$$
where in the last inequality we use again the decomposition of the integral into the near and far fields.
For the remaining three viscous terms, we estimate similarly as follows:
{
\[
\bal
 \|\widetilde{S_{12}}\|_{\rho_0}&=2|\beta|\|\widetilde{\nabla U\nabla \Theta^T}\|_{\rho_0}\ \les  2|\beta|\|\widetilde{\nabla W}\widetilde{\nabla \Phi^T}\|_{\rho_0}+ (1+ \Ga)^2 \les  \Ga( \Ga +E_1) + (1+ \Ga )^2\,, \\
  \|\widetilde{S_{13}}\|_{\rho_0}&=\|\widetilde{U\nabla \Theta\nabla \Theta^T}\|_{\rho_0} \les \|\widetilde{W}\widetilde{\nabla \Phi}\widetilde{\nabla \Phi^T}\|_{\rho_0}+ (1+ \Ga)^3\les \Ga^2(\Ga+E_0) +(1+ \Ga)^3 \,, \\
   \|\widetilde{S_{14}}\|_{\rho_0}&=\|\widetilde{U\nabla^2\Theta}\|_{\rho_0} \les 
\|\widetilde{W}\widetilde{\nabla^2\Phi}\|_{\rho_0}+ (1+ \Ga )^2 \les \Ga(\Ga+E_0) +(1+\Ga)^2 .
\eal
\]
}

We can collect the viscous estimate by Proposition \ref{prop:inter} and Assumption \ref{asss}:
\begin{equation}
    \label{vis-l2} { | (\widetilde{\mathscr{D}_U},W\rho_0) | \les \cE_2} (1+E_0+E_1+E_2 + \Ga)^3 E_0\les {\cE_2} (1+E_0+E_K)^3E_0 \les \cE_2 E_0\,.
\end{equation} 
We thereby conclude the proof of Lemma \ref{lem:l2} using \eqref{damp-l2}, \eqref{non-l2}, \eqref{res-l2}, and \eqref{vis-l2}.  
\end{proof}
One sees that we already have leading order damping in the $L^2$ estimates. However, to close the nonlinear estimates, we will need higher order estimates to control the $L^\infty$ norms. 
\subsection{$H^1$ stability analysis of the phase}\label{h1sec}

We consider the weighted $H^1$ norm of the phase $F_1^2=(\nabla\Phi,\nabla\Phi\mathring{\rho}_1)\,$ \eqref{norm:Ek}.  
We choose this norm since $\Phi$ does not decay at the origin. We will show that the following lemma holds. 
\begin{lemma}[Weighted $H^1$ estimate] \label{lem:h1}
    Under the bootstrap assumption \ref{asss}, it holds 
    \begin{equation}    
    \label{h1-collecy}
            \frac{1}{2} \f{d}{d \tau} F_1^2\leq - \frac{1}{8}F_1^2+ { C ( E_{K-1}+E_0)^2
          } + {C \cE_2} F_1\,,   
\end{equation}
for some absolute constant $C>0$.
\end{lemma}\begin{proof}
    We  have by \eqref{imaginep} that 
$$\frac{1}{2} { \f{d}{d \tau} }  F_1^2=(\nabla(- \f{1}{2} \Lambda \Phi)+\nabla\mathscr{N}_\Theta+\nabla\mathscr{F}_\Theta+\nabla\mathscr{D}_\Theta,\nabla\Phi\mathring{\rho}_1)\,.$$

For the leading order linear term, we have via integration by parts
 that 
 $$(\partial_i(- \f{1}{2} \Lambda \Phi),\partial_i\Phi\mathring{\rho}_1)=- \f{1}{4}(\partial_i\Phi,\partial_i\Phi\mathring{\rho}_1)\,.
 $$
 Therefore we have the linear estimate\begin{equation}
\label{lin-theta}
(\nabla(- \f{1}{2} \Lambda \Phi),\nabla\Phi\mathring{\rho}_1)=-\frac{1}{4}F_1^2\,.
 \end{equation}
 
 For the nonlinear estimate, we again use Newton-Leibniz's formula to get 
 \beq\label{eq:newton_th}
 \mathscr{N}_\Theta= 
\d ( ( \bar U + W)^{p-1 } - \bar U^{p-1} )  = \delta(p-1)W\int_0^1(\bar{U}+\alpha W)^{p-2}d\alpha\,.
 \eeq
{It is not difficult to see that $ \la z \ra^{\s  + \f{\e-1}{2}} \in \cH^i$ for any $i\geq 0$ since $\e-1 < 0$. Note that $\mr \rho_1$ \eqref{wg:rho} is locally integrable and $\mr \rho_1
\les \la z \ra^{\s  + \f{\e-1}{2}} g_1$. Applying $L^{\infty}$ estimate for $W, \bar U + \al W$ from Corollary \ref{cor:decay}, \eqref{eq:inte:3} in Proposition \ref{prop:inter},  and Proposition \ref{prop:Hk} \eqref{eq:Hk_est1} with $W_1 =\la z \ra^{\s  + \f{\e-1}{2}} , W_2 = W, j_1 = 0, j_2 = 1, k=1$, we obtain 
\[
\bal
|| \nabla\mathscr{N}_\Theta ||_{\mr \rho_1}
& \les || \nabla\mathscr{N}_\Theta  ||_{L^{\infty}(|z|\leq 1)}
+ || W_1 \nabla\mathscr{N}_\Theta ||_{ g_1}  
  \les  E_0 + E_{K-1} + || W_1 \nabla\mathscr{N}_\Theta ||_{  \rho_1}  \\
 & \les  E_0 + E_{K-1}  + || W_2 ||_{\cH^{K-1}} 
 \les E_0 + E_{K-1}  .
 \eal
\]
}
We can collect the nonlinear estimate, via an AM-GM inequality as follows:\begin{equation} \label{non-theta}
|(\nabla\mathscr{N}_\Theta,\nabla\Phi\mathring{\rho}_1) | \leq C \|\nabla\mathscr{N}_\Theta\|_{\mathring{\rho}_1}F_1\leq C (  E_0+ E_{K-1} )^2 +\frac{1}{8}F^2_1\,.\end{equation}

For the residue estimate, we have 
\begin{equation}
\label{res-theta}
| (\nabla\mathscr{F}_\Theta,\nabla\Phi\mathring{\rho}_1) | \leq\|\nabla\mathscr{F}_\Theta\|_{\mathring{\rho}_1}F_1\les  \cE_2(1+F_1+\|\nabla^2\Phi\|_{\mathring{\rho}_1})F_1
\les \cE_2 F_1 
\,.
\end{equation}

For the viscous estimate, we have 
\beq\label{eq:vis_H1_1}
\begin{aligned}
|\nabla\mathscr{D}_\Theta\|_{\mathring{\rho}_1}&\les \cE_2 \B(\|\frac{\nabla^3U}{U}\|_{\mathring{\rho}_1}+\|\frac{\nabla^2U}{U}\|_{\mathring{\rho}_1}\|\frac{\nabla U}{U}\|_{\infty }+\|\frac{\nabla U}{U}\|_{\infty}\|\nabla ^2\Theta\|_{\mathring{\rho}_1}\\&+(\|\frac{\nabla^2U}{U}\|_{\infty}+\|\frac{\nabla U}{U}\|_{\infty}^2)\|\nabla \Theta\|_{\mathring{\rho}_1}+1+F_1+\|\nabla^3\Phi\|_{\mathring{\rho}_1} \B)\,.
\end{aligned}
\eeq
To estimate the integral $L^2(\mr \rho_1)$, we apply $L^{\infty}$ estimate in the region $|z| \leq 1$ and \eqref{eq:Hk_estU} and $ \mr \rho_1 \les \la z \ra^{\s + \e/2} g_1 \les \la z \ra^{\s + \e/2} g_2 $  to the region $|z| \geq 1$: 
\[
\bal
|| \na^l U / U ||_{\mr \rho_1} 
\les || \na^l U / U ||_{L^{\infty}(|z|\leq 1)}
+ || \la z \ra^{\s + \e/2} \na^l U / U ||_{g_1} 
\les 1 +  E_0 + E_{K} \les 1 , \quad l = 2,3 .
\eal
\]
Applying \eqref{eq:Hk_estU_linf} with $(l, m) =(1,0),(0,0)$, we get 
\beq\label{eq:vis_H1_2}
|\na^{l+1} U / U| \les 1.
\eeq
As a consequence, we can simplify the viscous estimate as follows:
\begin{equation}
    \label{h1-vis}|\nabla\mathscr{D}_\Theta\|_{\mathring{\rho}_1}\les \cE_2 (1+F_1+\|\nabla ^2\Phi\|_{\mathring{\rho}_1} +\|\nabla^3\Phi\|_{\mathring{\rho}_1})\,.
\end{equation}
Finally, since $\mr \rho_1$ is $L^1$ integrable and $\mr \rho_i \les \la z \ra^{2i-2} \mr \rho_1$
\eqref{eq:comp_wg}, we can decompose the integral region into $I=[0,1]^d$ and the rest of the outer region $I^c$ as in the $L^2$ estimate of the amplitude to compute 
$$\|\nabla^l\Phi\|_{\mathring{\rho}_1} \lesssim\sup_{z\in I}|\nabla^l\Phi|+\|\nabla^l\Phi|z|^{2l-2}\|_{\mathring{\rho}_1} \les \Ga+F_l\,,
          \quad  || \na^l \bar \Th ||_{\mathring{\rho}_1} \les 1, \quad  l=2,3\,.
$$
We use Proposition \ref{prop:inter} and the bootstrap assumption \eqref{eq:boot_EE} to further obtain 
\[
\|\nabla^l\Phi\|_{\mathring{\rho}_1} + || \na \bar \Th |_{\mathring{\rho}_1}
 \les 1 + F_1 + F_K \les 1.
\]
Plugging in the estimate in \eqref{h1-vis} and combined with \eqref{lin-theta}, \eqref{non-theta}, and \eqref{res-theta}, we establish Lemma \ref{lem:h1}.
\end{proof}
\subsection{$H^K$ stability analysis}\label{hksec}

For the estimate at the highest order, we consider the weighted $H^K$ energies \eqref{norm:Ek}
$$E_K^2=(|\nabla^K W|^2,\rho_K)\,,\quad F_K^2=(|\nabla^K \Phi|^2,U^2\rho_K)\,.$$ 
In this section, we will establish the following lemma.
\begin{lemma}[Weighted $H^K$ estimate]\label{lem:hk}
Under the bootstrap assumption \ref{asss}, we have
    \begin{equation}
     \label{hk-collected}
     \frac{1}{2} \f{d}{d \tau }  (E_K^2+F_K^2)\leq-\frac{\epsilon}{8}(E_K^2+F_K^2)+ \mu_0 \cE_1(E_K+F_K)\,\end{equation}
     for some absolute constant $\mu_0>0$.
\end{lemma}

\subsubsection{Estimates of the amplitude}

Recall the definitions of $\Ls_{\bar U}, \Ns_U, \Fs_U$ from \eqref{eq:lin}. We have  
$$\frac{1}{2} \f{d}{d \tau} E_K^2=(\nabla^K (\mathscr{L}_{\bar{U}}W)+\nabla^K \mathscr{N}_U+{\nabla^K\mathscr{F}_U}+{\nabla^K\mathscr{D}_U},\nabla^KW\rho_K)\,.
$$
For the leading order linear term, we can calculate the damping similarly as in the $L^2$ estimates.
{
A direct computation yields $\bar U^{p-1} \in \cH^i$ \eqref{ss-} for any $i \geq 0$. 
Using the Leibniz rule, the product rule \eqref{eq:prod} in Proposition \ref{prop:inter} with $i + j = K, j \leq K-1, m =K-1, n = i + d$, and $g_K \approx \rho_K$ \eqref{wg:rho}, \eqref{norm:Hk}, we yield 
\[
|| \na^K(\bar U^{p-1} W) - \bar U^{p-1} \na^K W ||_{\rho_K}
\les \sum_{ j \leq K-1} \na^{K-j} || \bar U^{p-1}||_{\cH^{i+d}} || W ||_{\cH^j}
\les \sum_{ j \leq K-1} E_j \les \cE_1.
\]
}
Therefore, we can compute 
$$
\nabla^K(\mathscr{L}_{\bar{U}}W)=c_{\bar{U}} \nabla^KW- \f{1}{2} \sum_iz_i\nabla^K\partial_iW-
\f{K}{2} \nabla^KW+p\bar{U}^{p-1}\nabla^KW+ { O(\cR_{\mathscr{L}, K}),  \quad 
|| \cR_{\mathscr{L}, K } ||_{\rho_K} \les \cE_1.}
$$
We can calculate the damping similar to the $L^2$ case as follows:
$$\begin{aligned}
    d_K & \coloneqq-\frac{1}{p-1}-\frac{K}{2}+p\bar{U}^{p-1}+\frac{1}{4\rho_K}\nabla\cdot(z\rho_K)\\&=\frac{p}{p-1+c_p|z|^2}-\frac{1}{p-1}-\frac{K}{2}+\frac{d+(2K+\frac{4}{p-1}-\epsilon)c_1|z|^{\frac{4}{p-1}-\epsilon-d+2K}}{4(1+c_1|z|^{\frac{4}{p-1}-\epsilon-d+2K})}\leq-\frac{\epsilon}{8}\,,
\end{aligned}$$
where the last inequality holds for a sufficiently small $c_1$, which we defer till \eqref{defered} where we combine this damping with the estimates of the nonlinear term in the phase equation. 

For the nonlinear term, we use Netwon-Leibniz's formula twice as in the $L^2$ estimate \eqref{eq:Newton}, to derive  
$$
 | (\nabla^K\mathscr{N}_U,\nabla^K W\rho_K) | \lesssim \sup_{\alpha\in[0,1]}(1-\alpha)\|\nabla^K(W^2(\bar{U}+\alpha W)^{p-2})\|_{\rho_K}E_K\,.
$$
Since $|| f||_{\rho_K} \les || f ||_{\cH^K}$  \eqref{wg:rho}, \eqref{norm:Hk}, 
using the product estimate \eqref{eq:Hk_est2_Hk} with $(W_1, W_2) = (W, W)$ and $ || f ||_{\cH^K} 
\les E_0 + E_K$ \eqref{eq:norm_comp}, we obtain 
\beq\label{non-hk} 
\bal
 | (\nabla^K\mathscr{N}_U,\nabla^K W\rho_K) | & \les E_K || W||_{\cH^K} || W ||_{\cH^{K-1}} 
 \les E_K  ( E_K + E_0) ( E_0 + E_{K-1}) 
\les E_K \cE_1. 
\eal
\eeq
Recall $\cE_2$ from \eqref{eq:lot}. For the residue term, we have via integration by parts that 
$$
| (\nabla^K\mathscr{F}_U,\nabla^KW\rho_K) | {\les \cE_2}(E_K^2+ E_K+ (|\nabla^KW|^2,|z\cdot\nabla\rho_K|+|\nabla\rho_K|))\,.
$$
Since we have $|\nabla\rho_K|\langle z\rangle\lesssim\rho_K$, we can conclude the residue estimate
\begin{equation}
    \label{res-hk} | (\nabla^K\mathscr{F}_U,\nabla^KW\rho_K) |\les \cE_2(E_K+E_K^2) {\les \cE_1 } E_K\,.
\end{equation}

\subsubsection{Estimates of the phase}
We have 
\begin{align*}
     \frac{1}{2} \f{d}{d \tau } F_K^2 & =\Big(\nabla^K (- \f{1}{2} \Lambda \Phi)+\nabla^K \mathscr{N}_\Theta+{\nabla^K\mathscr{F}_\Theta}+{\nabla^K\mathscr{D}_\Theta},\nabla^K\Phi U^2\rho_K\Big) \\
     & \qquad \qquad  \qquad +\Big(\mathscr{L}_{\bar{U}}W+ \mathscr{N}_U+{\mathscr{F}_U}+{\mathscr{D}_U},  U|\nabla^K\Phi|^2\rho_K\Big)\,.
\end{align*}
Notice that the weight is time-dependent.
We remark that it is essential to pair the two linear terms and the two residue terms together to cancel out the leading order term via integration by parts. For the leading order linear term, we have via integration by parts that 
$$
(\nabla^K (- \f{1}{2} \Lambda \Phi),\nabla^K\Phi U^2\rho_K)+(\mathscr{L}_{\bar{U}}W+\mathscr{N}_U,  U|\nabla^K\Phi|^2\rho_K)=(\mathring{d}_{K},|\nabla^K\Phi|^2U^2\rho_K)\,,
$$
where we can calculate the damping $$\mathring{d}_{K}=\frac{-K}{2}+\frac{1}{4\rho_K}\nabla\cdot(z\rho_K)-\frac{1}{p-1}+\frac{{U}^{p}}{U}< d_K+{U}^{p-1}-\bar{U}^{p-1}\,.$$
Notice that by \eqref{eq:Newton}, \eqref{eq:inte:3} in Proposition \ref{prop:inter} with $l = 0$, and Corollary \ref{cor:decay}, we can further estimate
\beq\label{eq:main_Up}
     |{U}^{p-1}-\bar{U}^{p-1}|\les \sup_{0\leq\alpha\leq1}|W(\bar{U}+\alpha W)^{p-2}|\les 
     (E_0 + E_{K-1} ) (1+  E_K+E_0 )^{p + 2}) \les \cE_1.
    \eeq
For the residue term, similarly via integration by parts, we have \begin{equation}
\label{res_hk1}
\begin{aligned}
     | (\nabla^K \mathscr{F}_\Theta,\nabla^K\Phi U^2\rho_K) | +| (\mathscr{F}_U,  U|\nabla^K\Phi|^2\rho_K) | & {  \les \cE_2}(F_K+F_K^2+(\frac{\nabla\cdot((Pz+v)\rho_K)}{2\rho_K},|\nabla^K\Phi|^2U^2\rho_K))\\ & \les \cE_2(F_K+F_K^2)\les \cE_1F_K\,,
\end{aligned}\end{equation}
where the inequality is again by the fact that $|\nabla\rho_K|\langle z\rangle\lesssim\rho_K$.

	For the nonlinear term, using Newton-Leibniz's formula \eqref{eq:newton_th}, we obtain 
\[
 |  \na^K \Ns_{\Th} | 
\leq I_{0, K} + C \sum_{j\leq K-1} I_{0, j} 
,\quad I_{i, j} = \d(p-1)  \cdot  \na^j W \cdot \na^{K-j} ( U + \al \bar W ) . 
\]
Applying \eqref{eq:Hk_est1} in Proposition \ref{prop:Hk} with $(W_1, W_2, j_1, j_2, k) = (U, W, 0, j, K-j)$ and $\bar U \in \cH^i$ \eqref{eq:U_decay}, we obtain 
\[
|| U I_{0, j} ||_{\rho_K} \les || U ||_{\cH^{K}} || W ||_{\cH^{K-1}} 
\les E_0 + E_{K-1}  \les \cE_1.
\]
Recall $\mr \rho_K = U^2 \rho_K$ \eqref{wg:rho}. For $ j\leq K-1$, the above estimate implies 
\beq\label{eq:non_HK_Th1}
 ||  I_{0, j} ||_{\mr \rho_K} = || U I_{0, j} ||_{\rho_K} \les \cE_1,
 \quad | ( I_{0, j}, \na^K \Phi \mr \rho_K ) | 
 \les  ||  I_{0, j} ||_{\mr \rho_K}  F_K \les \cE_1 F_K.
\eeq
The term $I_{0, K}$ is trickier and we need to estimate by an AM-GM inequality: 
{
$$\begin{aligned}
    ((\bar{U}+\alpha W)^{p-2}\nabla^K W,\nabla^K\Phi U^2\rho_K)&\leq\frac{1}{2}\| ( U^{\f{1}{2}}(\bar{U}+\alpha W)^{\frac{p-2}{2}}\nabla^K \Phi \|^2_{\mr \rho_K} +\frac{1}{2}\|U^{ \f{1}{2}}(\bar{U}+\alpha W)^{\frac{p-2}{2}}\nabla^K W\|^2_{\rho_K}\, ,
\end{aligned}$$
where we pair one of $U$ in $U^2$ with $\rho_K^{1/2}$ to get $\mr \rho_K^{1/2}$. Applying $U = (1-\al) W + \bar U + \al W$, Newton-Leibniz's rule for $(\bar U + \al W)^{p-1} - \bar U^{p-1}$, Proposition \ref{prop:inter} for $W$, and Corollary \ref{cor:decay} for $\bar U + s W, s \in [0, 1]$, which are similar to the estimate of $\mathscr{N}_{\Th}$ \eqref{eq:newton_th}, we obtain 
\[
\bal
U (\bar U + \al W)^{p-2} 
& \leq (\bar U + \al W)^{p-1} + C |W (\bar U + \al W)^{p-2} |
\leq \bar U^{p-1} + C \sup_{s \in [0, 1]}   |W (\bar U + s W)^{p-2} \\
&\leq \bar U^{p-1} + C (  E_{K-1} + E_0) 
\leq 
  (\min\{p-1,c_p\})^{-1}\langle z\rangle^{-2}+ C \cE_1 ,
\eal
\]
}where we extract a decay at the far field for the leading order term.
We can calculate the damping 
\begin{align*}
    &\d(p-1)(\min\{p-1,c_p\})^{-1}\langle z\rangle^{-2}+d_K \\
    & \quad  \leq \frac{ \d(p-1)(p+1) }{\min\{p-1,c_p\}(1+|z|^2)}-\frac{1}{p-1}-\frac{K}{2}+\frac{d+(2K+\frac{4}{p-1}-\epsilon)c_1|z|^{\frac{4}{p-1}-\epsilon-d+2K}}{4(1+c_1|z|^{\frac{4}{p-1}-\epsilon-d+2K})}\,.
\end{align*}
Recall the definition of $K$ in \eqref{K} and
similar to the $L^2$ damping, we can use a weighted AM-GM inequality to conclude for a sufficiently small positive $c_1$, we have  
\begin{equation}
    \label{defered}
   (p-1) \d (\min\{p-1,c_p\})^{-1/2}\langle z\rangle^{-1/2}+d_K\leq-\frac{\epsilon}{8}\,.
\end{equation}
As a consequence, we collect the linear and nonlinear estimates of the phase, and the linear estimate of the amplitude together as follows:
\begin{align}
   \Big(\nabla^K(\mathscr{L}_{\bar{U}}W),\nabla^K W\rho_K\Big)+ \Big(\nabla^K (- \f{1}{2} \Lambda \Phi+\mathscr{N}_\Theta)&+\frac{\mathscr{L}_{\bar{U}}W+\mathscr{N}_U}{U},\nabla^K\Phi U^2\rho_K\Big) \nonumber \\
    & \quad \leq  -\frac{\epsilon}{8}(E^2_K+F^2_K)+C \cE_1( E_K+F_K)
    \,. \label{non-hk1}
\end{align}

\subsubsection{Estimates of the viscous terms}
Finally, we estimate the viscous terms. The simpler term  can be estimated as follows: \begin{equation}
    \label{vos-ls}({\mathscr{D}_U},  U|\nabla^K\Phi|^2\rho_K)\leq\|\frac{\mathscr{D}_U}{U}\|_{\infty}F_K^2\les \cE_1 F_K\,.
\end{equation}
The last inequality is derived similarly to the $H^1$ viscous estimates in \eqref{eq:vis_H1_1}, \eqref{eq:vis_H1_2}.

We group leading order viscous terms as follows and estimate them together: $$(\nabla^K\mathscr{D}_U,\nabla^K W\rho_K)+(\nabla^K\mathscr{D}_\Theta,\nabla^K\Phi U^2\rho_K)\,,$$
and we will use integration by parts to cancel out the leading order terms and extract damping.  
Recall the definition of the viscous terms in \eqref{eq:vis}. For any tensor $f$, we define $$|f|^2_{\Qc}=\sum_{i}(\nabla f_i)^T\Qc\nabla f_i\,,$$ where we sum over its scalar entry components $f_i$. 

Notice that $|\nabla\rho_K|\lesssim\rho_K$. We compute the damping of the amplitude using integration by parts and the Cauchy-Schwarz inequality as 
$$\begin{aligned}
    (\nabla^K\Delta_\Qc  U,\nabla^K W\rho_K) & \leq C \cE_1 E_K -(|\nabla^{K} W|^2_{\Qc},\rho_K) +C|\Qc|^{1/2}\|\nabla^{K} W\|_{\rho_K}\||\nabla^{K} W|_{\Qc}\|_{\rho_K}  \\
    & \leq C \cE_1 E_K-\frac{1}{2}(|\nabla^{K} W|^2_{\Qc},\rho_K)\,.
\end{aligned}$$
Using \eqref{eq:inte:3} from Proposition \ref{prop:inter} and \eqref{eq:U_decay}, we get 
\[
|\na U| \les U 
, \quad  |\nabla(U^2\rho_K)|\lesssim  
|\na U| U \rho_K + U^2 | \na \rho_K | \les  
U^2\rho_K. 
 \]
Similarly, we compute the damping of the phase as
$$
  (\nabla^K\Delta_\Qc  \Theta,  \nabla^K\Phi U^2\rho_K)\leq C \cE_1F_K-\frac{1}{2}(|\nabla^{K} \Phi|^2_{\Qc},U^2\rho_K)\,.
$$

For the four intermediate terms in the viscous terms 
$$
I_1 = \langle\nabla U,\nabla\Theta\rangle_\Qc \,,\quad I_2 = U\langle\nabla \Theta,\nabla\Theta\rangle_\Qc \,,\quad I_3 = \frac{\langle\nabla U,\nabla\Theta\rangle_\Qc }{U}\,,\quad I_4 = \langle\nabla \Theta,\nabla\Theta\rangle_\Qc \,,
$$
we can simply control their weighted norms using the diffusion term. 

{We consider the most challenging term $I_3$. Using the Leibniz rule, \eqref{eq:Hk_estU} in Proposition \ref{prop:Hk}, we obtain 
\[
|| I_3 ||_{\mr \rho_K}
\les \sum_{0\leq i\leq K} || I_{3, i} ||_{\mr \rho_K}, \quad  I_{3, i} = (\na^i \f{ \na U }{U}, \na^{K-i+1} \Th )_\Qc .
\]
For $ 1 \leq i \leq K-1$, applying  \eqref{eq:Hk_estU_linf} to $\na U / U$ if $i \leq  K/2 + 1 < K-d-1$ and \eqref{eq:inte} to $\Th$ if $ i > K/2$, which implies $K-i+1 \leq K/2+1 < K-d-1$, we obtain
\[
|I_{3, i} |
\les |\Qc| ( \la z \ra^{-i}  | \na^{K-i+1} \Th|
  + \la z \ra^{-(K-i)} |  \na^i (\na U / U)|    ).
\]
Since $ i \leq K-1, K-i+1 \leq K$, using the estimate \eqref{eq:comp_wg} for weights
\[
 \mr \rho_K^{1/2} \les \la z \ra^{i-1} \mr \rho^{1/2}_{K-i+1}, \quad \la z \ra^{-(K-i)} \mr \rho_K^{1/2}
 \les \mr \rho_{i}^{1/2}
 \les \la z \ra^{\s + \e/2} \rho_{i}^{1/2}, 
 \]
and \eqref{eq:Hk_estU}, we obtain 
\[
\bal
 ||I_{3, i}||_{ \mr \rho_K}
& \les |\Qc| 
(||\na^{K-i+1} \Th  ||_{ \mr \rho_{K-i}} 
+ || \la z \ra^{\s +\e/2}  \na^i (\na U / U) ||_{\rho_i})  \les |\Qc|  \les \cE_1.
\eal
\]
}
For $I_{3,0}, I_{3, K}$, we use the Cauchy-Schwarz inequality to compute that its $\rho_K$ norm is bounded by $$ \cE_1+|\Qc|^{1/2}(\|\nabla\Theta\|_{\infty}\||\nabla^{K} W|_{\Qc}\|_{\rho_K}+\|\frac{\nabla U}{U}\|_{\infty}\|U|\nabla^{K} \Phi|_{\Qc}\|_{\rho_K})\,.$$
Similarly, we have the estimates for the other three viscous terms $I_1, I_2, I_4$.
Combined with \eqref{eq:vis_H1_2}, we can use the Cauchy-Schwarz inequality to derive that 
\begin{align*}
&\Big(-2\beta\nabla^K\langle\nabla U,\nabla\Theta\rangle_\Qc -\nabla^K(U\langle\nabla \Theta,\nabla\Theta\rangle_\Qc ),\nabla^K W\rho_K\Big)+\Big(2\nabla^K\frac{\langle\nabla U,\nabla\Theta\rangle_\Qc }{U}-\beta\nabla^K{\langle\nabla \Theta,\nabla\Theta\rangle_\Qc },\nabla^K\Phi U^2\rho_K\Big)\\
& \qquad \leq C\cE_1(E_K+F_K)+\frac{1}{8}((|\nabla^{K} W|^2_{\Qc},\rho_K)+(|\nabla^{K} \Phi|^2_{\Qc},U^2\rho_K))\,.
\end{align*}

Finally, for the last two viscous terms, we use integration by parts to cancel out the leading order terms. 
{
Applying estimates similar to the those for $I_3$ in the above, we can extract the leading order terms, which involve $\na^{K+2} U$ or $\na^{K+2} \Th$,
\[
\bal
-\beta \B(\nabla^K(U\Delta_\Qc  \Theta),\nabla^K W\rho_K \B)
& = - \beta ( U  \Delta_\Qc  \na^K \Th, \na^K W \rho_K) + O(\cE_1 E_K) + \f{1}{16} (|\nabla^{K} \Phi|^2_{\Qc},U^2\rho_K) , \\
\beta\B(\nabla^K\frac{\Delta_\Qc  U}{U},\nabla^K \Phi U^2\rho_K\B)& = \beta \B( \f{ \D_\Qc  \na^K U }{U} , \na^K \Phi U^2 \rho_K \B) + O(\cE_1 F_K) + \f{1}{16} (|\nabla^{K} U|^2_{\Qc},\rho_K)
.  \\
\eal
\]
}
Now, we use $U = \bar U + W,  \Th = \bar \Th + \Phi$ and integration by parts to cancel out the leading order terms. 
$$
\begin{aligned}
    &-(\nabla^K\Delta_\Qc  \Theta,\nabla^K WU\rho_K)+({\nabla^K\Delta_\Qc  U},\nabla^K \Phi U\rho_K) \\
 = &  -(\nabla^K\Delta_\Qc  \Phi,\nabla^K WU\rho_K)+ ({\nabla^K\Delta_\Qc  W},\nabla^K \Phi U\rho_K) + O(\cE_1(E_K+F_K)) \\
 = & \sum_{i,j} \B(  \Qc_{ij} ( -\pa_i ( \pa_j \na^K \Phi \cdot \na^K W)  + \pa_j ( \pa_i \na^K W \cdot \na^K \Phi ) ), U \rho_K \B) \\
     \leq &  C \cE_1(E_K+F_K)+\frac{1}{16\beta}((|\nabla^{K} W|^2_{\Qc},\rho_K)+(|\nabla^{K} \Phi|^2_{\Qc},U^2\rho_K))\,.
\end{aligned}
$$ 
We notice that the remaining terms from integration by parts are controlled since $|\nabla(U\rho_K)|\lesssim U\rho_K$.

Combining the viscous estimates with the estimates \eqref{non-hk}, \eqref{res-hk}, \eqref{res_hk1}, and \eqref{non-hk1}, we conclude the proof of Lemma \ref{lem:hk}.
\subsubsection{Summary of the $H^K$ estimates}
 Using \eqref{eq:inte:1}, \eqref{eq:inte:2} in Proposition \ref{prop:inter}, for any $\mu > 0$, we obtain 
\[
\cE_1 \leq |\Qc| + H^{p-1}
+ C(\mu)(E_0 + F_1) + \mu( E_K + F_K) +( E_0 + F_1 +E_K + F_K)^2. 
\]
By Lemma \ref{lem:hk},
choosing $\mu < \f{\e}{16 \mu_0}$ and then collecting \eqref{hk-collected}, \eqref{h1-collecy}, and \eqref{l2-collecy}, we obtain that there exists a sufficiently small constant $1 > \nu_1>\nu_2>0$, $\nu_2$ determined after $\nu_1$, such that for the energy \begin{equation}   \label{eneq}E^2=E_K^2+F_K^2+1/\nu_1 F_1^2+ 1/\nu_2 E_0^2\,,
 \end{equation}
 the following estimate holds 
 \begin{equation}
 \label{booten}\frac{1}{2} \f{d}{d \tau} E^2\leq-\frac{\epsilon}{16}E^2+C(|\Qc|+H^{p-1})
E+CE^3 \,, 
\iff 
 \f{d}{d \tau} E \leq-\frac{\epsilon}{16}E +\mu_1(|\Qc|+H^{p-1})
+\mu_1 E^2 \,, 
 \end{equation}
 for some absolute constant $\mu_1>0$. Here, the constant $C$ would depend on $\nu_1, \nu_2$. Once we fix $\nu_1, \nu_2$, then $C$ becomes a fixed constant $\mu_1$. The estimate holds provided that Assumption \ref{asss} is valid.

\subsection{Lower bound of the amplitude}\label{sec:linfty}




We now prove the bootstrap Assumption \ref{asss} by estimating the lower bound of $U\rho$,
for the  weight $\rho=\bar{U}^{-1-\epsilon_2}$. We will proceed with a maximal principle argument and a barrier argument.  Notice that $$\nabla U=\frac{\nabla(U\rho)-U\nabla\rho}{\rho}\,,\quad\nabla^2 U=\frac{\nabla^2(U\rho)-U\nabla^2\rho}{\rho}-\frac{\nabla(U\rho)\nabla\rho^T+\nabla\rho\nabla(U\rho)^T-2U\nabla\rho\nabla\rho^T}{\rho^2}\,. $$ 
We compute by \eqref{real} that 
\begin{equation}\label{max-p}
\bal
& \partial_\tau(U\rho)= \cP_U (U \rho) ,
\quad \cP_U  f = A_0 f + A_1 \cdot \na f + \tr( Q \na^2 f ) .  \\
\eal
 \end{equation}
where the coefficients $A_0, A_1$ of the parabolic operator $\cP_U$ are:  
$$
\begin{aligned}
A_0 &=c_U-H^{p-1}\gamma+U^{p-1}+( \f{1}{2} z+\Pc z+\Vc)\cdot\frac{\nabla\rho}{\rho}-\frac{\Delta_\Qc \rho-2\beta\langle\nabla \rho,\nabla\Theta\rangle_\Qc }{\rho}+2\frac{\langle\nabla \rho,\nabla\rho\rangle_\Qc }{\rho^2}-\langle\nabla \Theta,\nabla\Theta\rangle_\Qc -\beta\Delta_\Qc \Theta\,, \\
A_1 & = -(  \f{1}{2} z+\Pc z+\Vc+2\frac{\Qc\nabla\rho}{\rho}+2\beta \Qc ), 
\end{aligned}$$

Notice that $\bar{U}$ is the approximate steady state and $|\nabla\rho|\langle z\rangle\lesssim\rho$. We can calculate the damping using the nonlinear estimate \eqref{eq:main_Up} and Lemma \ref{lem1}
that:
\[
A_0= O(\cE_1)-\epsilon_2 \frac{z\cdot\nabla\bar{U}}{\bar{U}}\,,
\quad A_1 = - ( \f{1}{2} z + \Pc z) + O(\cE_1)  ,\quad |\Pc| \les \cE_1.
\]
Next, we define a barrier function $F = \bar U^{ - 4 \e_2}$.  
Since $|z \cdot \na F | \les F, |\na^i F| \les F, i=1,2 $, we get
\[
\bal
\cP_U F & = ( O(\cE_1)- \e_2 \f{z \cdot \na \bar U}{\bar U}  
+ \f{ A_1 \cdot \na F }{F} ) F  + \tr(\Qc \na^2 F ) \\
& = (  O(\cE_1)   -\e_2 \f{ z \cdot \na \bar U}{\bar U} - \f{1}{2} \f{z \cdot \na F}{F}  ) F
=  ( O(\cE_1) -  (\e_2 - 2 \e_2) \f{z \cdot \na \bar U}{\bar U} )F 
= (O(\cE_1) + \e_2 \f{z \cdot \na \bar U}{\bar U} )F.
\eal
\]
- For $|z|\geq 1$, we derive by the form of $\bar{U}$ in \eqref{ss-}  the lower bound $-\frac{z\cdot\nabla\bar{U}}{\bar{U}}\geq \mu_{U,2}$ for some positive constant $\mu_{U,2}$\,. Recall the definition of $\cE_1$ \eqref{eq:lot} and $E \les 1$ from \eqref{eneq} and Assumption \eqref{eq:boot_EE}. Since $|\Qc| \les \tr(\Qc)$, for some positive constant $\mu_{U,1}$, we have 
\begin{equation}\label{max-g1}
A_0\geq \mu_{U,2}\epsilon_2 - \mu_{U,1}( \tr(\Qc)+H^{p-1} + E) \,,
\quad \cP_U F \leq (  \mu_{U,1}( \tr(\Qc)+H^{p-1} + E) - \mu_{U, 2} \e_2 ) F.
 \end{equation}

\noindent - For $|z|\leq 1$, since $\rho$ is bounded on the interval and we recall the definition of $\Gamma$ \eqref{ee}, we can estimate 
\begin{equation}\label{max-l1}
U\rho=\bar{U}^{-\epsilon_2}+W\rho\geq 4 C_b-C\Gamma\geq 4 C_b- \mu_{U,3} E\,,
 \end{equation}
for some positive constant $\mu_{U,3}$. Here we use the definition of $C_b$ in Assumption \ref{asss}.

Hence, by enforcing $E, |\Qc|+H^{p-1}$ sufficiently small, we will verify the following bootstrap assumption.

\begin{assumption}\label{ass:lower}
\beq\label{eq:boot_lower}
A_0 > 0, \  \cP_U F < 0, \quad  |z| \geq 1, \quad   \ U \rho > 2 C_b,  \ |z| \leq 1.
\eeq

\end{assumption}

Now, we consider $\Om_{ c} = U \rho + c F$ for $c >0$. From Corollary \ref{cor:decay} and the choice of $\e_2$, we obtain 
\[
 U \rho \les \la z \ra^{\s + \e /2} \la z \ra^{-\s + \e / 2 } 
 = \la z \ra^{\e},
 \quad F = \bar U^{-4 \e_2} \gtr \la z \ra^{ 8 \e_2/(p-1)} \gtr \la z \ra^{2 \e}.
\]

Under the assumption \eqref{eq:boot_lower} and \ref{asss},  for any $c > 0$, we have 
\[
\Om_c(z) > 2 C_b,\quad  |z| \leq 1, \quad \lim_{|z| \to \infty} \Om_c = \infty.
\]

Using the above estimates of $\cP_U$, we get 
\[
\pa_{\tau} \Om_c = \pa_{\tau} (U \rho) 
= \cP_U( U\rho + c F) - c \cP_U F = \cP_U \Om_c - c \cP_U F > \cP_U \Om_c.
\]

By choosing initial data with $ U \rho > 2 C_b$ and then applying the maximal principle to the operator $\cP_U$ on $|z| \geq 1$, we obtain 
\[
\Om_c > 2 C_b, \quad U \rho + c F \geq 2 C_b,\quad |z| \geq 1.
\]

Since $c$ is arbitrary, taking $c \to 0$, we prove $U \rho > 2 C_b$ for $|z| \geq 1$, which along with \eqref{eq:boot_lower} for $U \rho$ concludes 
$U \rho \geq 2 C_b, \forall z \in \Rb^d$ and strengthens \eqref{eq:boot_U} in Assumption \ref{asss}.

In Section \ref{sec:boot}, we prove Assumption \ref{ass:lower}.


\subsection{Bootstrap argument and blowup}\label{sec:boot}

In this section, we prove Theorem \ref{thm:stab} by combining previous estimates and use a bootstrap argument.


Recall the ODE of $Q$ from Lemma \ref{lem1} and $\cE_0, \Ga$ from \eqref{ee}
\beq\label{eq:EE_ODE}
  {\Qc}_\tau=-(\Qc_u+\frac{1}{2}\Qc_d)\Qc-\Qc(\Qc_u^T+\frac{1}{2}\Qc_d)+ \Oc(|\Qc| \Ec_0), \quad \cE_0 = |\Qc|(\Ga + \Ga^4) + H^{p-1}. 
\eeq

Since the parameters $\nu_i$ in the energy $E$ \eqref{eneq} have been chosen as some absolute constants, under the bootstrap assumption \ref{asss}, we get 
 \beq\label{eq:EE_ga}
E \les 1, \quad 
\Ga \les E_0 + E_K \les E \les 1,  \quad \cE_0 \les |\Qc| \Ga + H^{p-1}
\les |\Qc| E + H^{p-1}.
\eeq

Taking trace on both side of \eqref{eq:EE_ODE} and then using $\Qc = \Qc_u + \Qc_u^T + \Qc_d$,
\[
\bal
& \tr( (\Qc_u + \f{1}{2} \Qc_d) \Qc + \Qc(\Qc_u + \f{1}{2} \Qc_d )^T )
= \tr( (\Qc_u + \f{1}{2} \Qc_d + \f{1}{2} \Qc_d + \Qc_u^T  )  \Qc ) = \tr(\Qc^2), \\
 & |\Qc| \approx \tr(\Qc), \quad 
\tr(\Qc^2) = 
\sum \lambda_{\Qc,i}^2 \geq \f{1}{d}( \sum \lambda_{\Qc,i})^2 =
 \f{1}{d} ( \tr (\Qc))^2  ,
 \eal
 \]
 where $\lambda_{\Qc, i}$ is the eigenvalue of $\Qc$,  and the above estimates, we get for a constant $\mu_2$:
\beq\label{eq:EE_ODE2}
\pa_{\tau} \tr (\Qc) \leq - \tr(\Qc^2) + \mu_2 ( E \tr(\Qc)^2  + H^{p-1} \tr(\Qc))
\leq  - \f{1}{d}  (\tr \Qc)^2 + \mu_2 ( E (\tr \Qc)^2 + H^{p-1} \tr(\Qc) ).
\eeq

Recall $c_W$ from \eqref{per_ansatz}. To simplify the nonlinear estimates, in addition to bootstrap assumption \ref{asss}, we impose the following assumption 

\begin{assumption}\label{ass:refine}
\beq\label{eq:boot_refine}
|c_W| < \f{1}{2} \min( (p-1)^{-1}, 1), \quad E(\tau) < \min( \f{1}{ 4 d \mu_1} , \f{\e}{ 32 \mu_2} ) , 
\eeq
where $\mu_1$ is the constant in \eqref{eneq}. We denote 
\beq\label{eq:EE_nota}
 \e_1 = \mu_2 H^{p-1}, \quad a(\tau) = \exp( \mu_2 \int_0^\tau H^{p-1}(s) ds),
 \quad \lam = \f{\e}{32}. 
\eeq
\end{assumption}

\noindent 
\textbf{Consequence of bootstrap assumptions.} 
We perform the energy estimates under the assumptions \eqref{eq:boot_refine} and \ref{asss} and show that these estimates can be strengthened.

Using \eqref{eq:dyn_scal1}, \eqref{per_ansatz}, we obtain 
$(p-1) c_U(s) = -1 + (p-1) c_W(s) \leq -\f{1}{2}$ and 
\beq\label{cUUU}
H^{p-1}(\tau) \leq H^{p-1}(0) \exp(\int_0^{\tau} (p-1)c_U(s) d s)
\leq H^{p-1}(0) \exp( - \tau / 2),
\quad - \f{1}{d} + \mu_2 E(\tau) < - \f{1}{2 d}.
\eeq

We can solve the ODE of $(\tr(Q))^{-1}$ using the above estimate and \eqref{eq:EE_ODE2} to obtain 
\[
\pa_{\tau} E_\Qc ^{-1} \geq \f{1}{2d} - \mu_2 H^{p-1} E_\Qc ^{-1},\quad E_\Qc  := \tr(\Qc). 
\]

By choose $H^{p-1}(0)$ small enough such that $ \exp(2  \e_1 )  < 2$, for any $0\leq s \leq \tau$, we get 
\[
a(\tau) a(s)^{-1} \leq e^{ \e_1 \int_0^\tau \exp(-s/2) d s  } \leq e^{ 2 \e_1}  < 2,
\quad a(\tau)^{-1} a(s) > \f{1}{2},  \quad a(0) = 1.
\]

Solving the above ODE, we yield 
\beq\label{eq:EE_ODE3}
\bal
  E_\Qc ^{-1}(\tau) & \geq a(\tau)^{-1} E_\Qc ^{-1}(0) + \f{1}{2 d} \int_0^\tau a(\tau)^{-1}  a(s)  d s 
\geq \f{1}{2} ( E_\Qc ^{-1}(0) + \f{1}{2 d} \tau  ) ,\\
 E_\Qc (\tau) & \leq  \min( 2 E_\Qc (0), 4 d / \tau ).
\eal
\eeq

Using \eqref{eq:boot_refine}, the above estimates, and $-\f{\e}{16} + \mu_1 E < - \f{\e}{32} = \lam$ \eqref{eq:EE_nota}, we obtain 
\[
 \f{d}{d \tau} E \leq -\lam E + C ( E_\Qc   + H^{p-1}(0) e^{-\tau / 2} ).
\]

Solving the ODE and using \eqref{eq:EE_ODE3}, we obtain 
\[
E(\tau) \leq e^{-\lam \tau} E(0) 
+ C \int_0^\tau e^{-\lam(\tau-s)} ( \min( E_\Qc (0) , \f{1}{s}) + H^{p-1}(0) e^{-s / 2} ) d s  ,
\]
where $C$ is some absolute constant and can depend on $\e, \lam$. Since $\lam < 1/2$, by decomposing the integral into $s < \tau / 2$ and $s \geq \tau/2$, we obtain 
\beq\label{eq:EE_E}
\bal
E(\tau) & \leq e^{-\lam \tau} ( E(0) + C H^{p-1}(0) )
+ C \B( E_\Qc (0) \int_0^{\tau/2} e^{-\lam(\tau- s)} d s +  \int_{\tau/2}^\tau e^{-\lam(\tau-s)} \min(E_\Qc (0), 1 / \tau) d s \B) \\
& \leq e^{-\lam \tau/2} (E(0) + \mu_3 H^{p-1}(0) +\mu_3  E_\Qc (0)) + \mu_3  \min(E_\Qc (0), 1 / \tau)
\eal
\eeq
for some absolute constant $\mu_3 > 0$.

Plugging the above estimates and \eqref{eq:EE_ga} into Lemma \ref{lem1}, and using $E \les 1$ \eqref{eq:boot_refine}, we get for some $\mu_4>0$:
\beq\label{eq:EE_cw}
|c_W(\tau)| < C ( E_\Qc (\tau) + E_\Qc (\tau) E(\tau) + H^{p-1}(\tau) ) 
< \mu_4 ( \min( E_\Qc (0), 1/ \tau) + H^{p-1}(0) e^{-\tau/2}).
\eeq

\noindent 
\textbf{Continuation of the bootstrap assumptions.} For initial data satisfying 
\beq\label{eq:cond_init1}
E(0) < E_*, \quad  E_\Qc (0) < E_*, \quad  H^{p-1}(0) < E_* ,
\eeq
with $E_*$ sufficiently small, we obtain from \eqref{eq:EE_ODE3}, \eqref{eq:EE_E}, \eqref{eq:EE_cw}  the following estimates 
\[
\bal
E(\tau) & \leq e^{-\lam \tau/2} E_* (1 + 2\mu_3) + \mu_3 \min(E_* , 1/\tau) < E_*(1 + 3 \mu_3),\quad 
E_\Qc (\tau)  < 2 E_* , \quad  |c_W| < \mu_4 E_* , \\
H^{p-1}(\tau) & \leq H^{p-1}(0) < E_*,
\quad E(\tau) + \tr(\Qc) + H^{p-1} < (4 + 3 \mu_3 ) E_*. 
\eal
\]

Therefore, there exists $\nu_3 >0$ such that for $E_* < \nu_3$, the bootstrap assumption \eqref{eq:boot_refine} can be strengthened and continued. Plugging the above estimates into \eqref{max-g1}, \eqref{max-l1}  we obtain 
\[
A_0 \geq \mu_{U,2} \e_2 - \mu_{U, 3}(  3 \mu_3 + 4 ) E_*,
\quad \cP_U F \leq ( \mu_{U,2} (4 + 3 \mu_3) E_* - \mu_{U,2} \e_2 ) F,
\quad U \rho \geq 4 C_b- \mu_{U,3} (1 + 3 \mu_3) E_*.
\]

By further requiring $E_*$ to be sufficiently small, the Assumption \ref{ass:lower}  can be strengthened and continued. The $L^{\infty}$ estimate in Section \ref{sec:linfty} 
strengthens \eqref{eq:boot_U} in the Assumption \ref{asss}. Using the definition \eqref{eneq} and the above estimate for $E$, we obtain 
\[
(E_0 + E_K + F_1 + F_K)(t) \leq C(\nu_1, \nu_2) E_*,
\]
which strengthens the first inequality in \eqref{eq:boot_EE} in Assumption \ref{asss} by further choosing $E_*$ to be small enough. 

For the second inequality in \eqref{eq:boot_EE}, applying the Jacobi's formula 
$ \f{d}{d \tau} \det( \Qc(\tau)) = \det (\Qc(\tau)) \tr( \Qc^{-1} \f{d}{d\tau} \Qc)$ to \eqref{norm_p} and using $\tr(A B) = \tr(B A), \Qc = \Qc_u +\Qc_d + \Qc_u^T$, we obtain 
\[
 \pa_{\tau} \det(\Qc)  = \det(\Qc) \cdot \tr(  - \Qc + O(\cE_{0}) ) .
\]
From the above estimates,  $|\Qc|$ and $\cE_0$ remain uniformly bounded for all $\tau>0$. Since $\det(\Qc(0)) > 0$, we prove $\det(\Qc) \geq \det(\Qc(0)) e^{-C \tau}$, which 
strengthens the second inequality in \eqref{eq:boot_EE}. This concludes the proof of Theorem \ref{thm:stab}.

\section{Refined asymptotics} \label{sec:refineasym}

In this section, building on Theorem \ref{thm:stab}, we obtain sharp asymptotics stated in Theorem \ref{thm:blp}. 
In Section \ref{sec:asym_amp}, we estimate the sharp blowup rates for the amplitude similarly as in \cite{HNWarXiv24}. In Section \ref{sec:asym_phase}, we estimate the asymptotics related to the phase and prove $L^{\infty}$ convergence. 
In Section \ref{sec:thm_proof}, we combine Theorem \ref{thm:stab}, Propositions \ref{prop:asym} and \ref{prop:linf} to prove Theorem \ref{thm:blp}.

\subsection{Asymptotics of the amplitude and blowup rate}\label{sec:asym_amp}

We use $O_{in}$ and $C_{in}$ to track any constant depending on the norm of the initial data $\tr(\Qc(0)), \tr(\Qc^{-1}(0))$. We have the following results for the asymptotics.
\begin{proposition}\label{prop:asym}
Suppose that the initial data $(U, \Th, \Qc, H)$ satisfy the assumption in Theorem \ref{thm:stab}. We have the following asymptotics for the {modulation parameters}
\beq\label{eq:rate_asym1}
\B| \f{H(\tau)^{p-1}}{T-t(\tau) } - 1 \B| \les C_{in} \la \tau \ra^{-1}, 
\quad \lim_{\tau \to \infty} \f{\tau}{|\log( T-t(\tau))|}
= 1,
\quad \lim_{t\to T}\frac{\RR(t)}{\sqrt{(T-t)|\log(T-t)|}}= I_d .
\eeq

\end{proposition}

{
We consider $\tau \geq 2$. 
Note that $E_* < 1$. We focus on the asymptptics as $\tau \to \infty$ and the decay rate in $\tau$. 

}

\paragraph{Refined estimate of $\Qc$.}
By inserting \eqref{eq:EE_E} and \eqref{cUUU} into \eqref{eq:EE_ODE2}, we get 
$$
\pa_{\tau} E_\Qc  \leq - \f{1}{d}  E_\Qc ^2+ C((\frac{1}{\tau}+  e^{-\lambda\tau/2}  )E_\Qc ^2+  E_\Qc e^{-\tau/2})\,,
$$
for some absolute constant $C>0$. 
Since $E_Q > 0$, we arrive at the ODE 
$$
\pa_{\tau} E_\Qc ^{-1} \geq  \f{1}{d}  - C (  \frac{1 }{\tau}+e^{-\lambda\tau/2}  )- C  E_\Qc^{-1}e^{-\tau/2}\,.
$$

{
  By introducing the integrating factor $a(\tau) = \exp( -C E_* \int_1^{\tau} e^{-s /2 } d s )$, and using the fast convergence $|a(\tau) / a(s) - 1| \les E_* e^{-s/2}, a(\tau) \geq e^{-C E_*} $ for $1\leq s < \tau$, we can solve the above ODE and obtain 
}
$$
{E_\Qc ^{-1}}\geq \frac{1}{d}\tau+O(\log\tau) 
+ E_{\Qc}^{-1}(2) e^{- C E_*} 
\geq \frac{1}{d}\tau + O_{in}(\log \tau).
$$

Since $\tr(\Qc) = \sum \lam_{\Qc, i}$, 
we know that 
\beq\label{trq}
\min(\lambda_{\Qc,i})\leq\frac{1}{d}E_\Qc \leq \frac{1}{\tau}+O_{in}(\frac{\log\tau}{\tau^2}  ).
\eeq

Next we estimate $\text{tr}(\Qc^{-1})$. From \eqref{eq:EE_ODE}, we have by \eqref{eq:EE_ga} that 
$$
\pa_{\tau}\text{tr}(\Qc^{-1})=d-2\text{tr}(\cE_\Qc  \Qc^{-1})\leq d+\mu_2(EE_\Qc +H^{p-1})\text{tr}(\Qc^{-1}).
$$
By the above estimates of $E_\Qc $, and the same estimates of $E$ and $H^{p-1}$ in \eqref{eq:EE_E} and \eqref{cUUU}, we have that for sufficiently large $\tau$, there exists a $\mu_5$ such that 
$$
\pa_{\tau}\text{tr}(\Qc^{-1})\leq d+\frac{C}{\tau^2}\text{tr}(\Qc^{-1}).
$$
We conclude that 
$$
\tr(\Qc^{-1}) \leq {d}\tau+O(\log\tau) + \tr(\Qc^{-1}(2))
\leq {d}\tau+O_{in}(\log\tau) 
$$
Using $\tr(\Qc^{-1}) = \sum_i \lam_{\Qc, i}^{-1}$, we obtain
\beq\label{trq-1}
\max(\lambda_{\Qc,i})\geq\frac{d}{\text{tr}(\Qc^{-1})}\geq \frac{1}{\tau}+O_{in}(\frac{\log\tau}{\tau^2}).
\eeq

Combining the above estimates, we obtain 
$$
{\text{tr}(\Qc^{-1})}\text{tr}(\Qc)\leq d^2+O_{in}(\frac{\log\tau}{\tau}).
$$ 
Using $\tr(\Qc^{\al}) = \sum \lam_{\Qc, i}^{\al}, \al = 1, -1$, we derive
$$
\text{tr}(\Qc^{-1})\text{tr}(\Qc)=\sum \lambda_{\Qc,i}\sum \lambda^{-1}_{\Qc,i}=d^2+\sum_{i<j}(\sqrt{\frac{\lambda_{Q,i}}{\lambda_{\Qc,j}}}-\sqrt{\frac{\lambda_{\Qc,j}}{\lambda_{\Qc,i}}})^2.
$$
It follows 
$$\Big(\sqrt{\frac{\lambda_{\Qc,i}}{\lambda_{\Qc,j}}}-\sqrt{\frac{\lambda_{\Qc,j}}{\lambda_{\Qc,i}}}\Big)^2=O_{in}\Big(\frac{\log\tau}{\tau}\Big), \quad \forall i < j, 
\quad  \f{\max(\lambda_{\Qc,i})}{\min(\lambda_{\Qc,i})}=1 + O_{in}( (\frac{\log\tau}{\tau})^{ 1/2 } ).
$$
Combining the above estimate with \eqref{trq} and \eqref{trq-1}, we have that  each one of the eigenvalue satisfies 
$$
\lambda_{\Qc,i}=\frac{1}{\tau}+ O_{in}(\tau^{-3/2}\sqrt{\log \tau} )
= \frac{1}{\tau} +  O_{in}(a_{\tau}),\quad a_{\tau} = \tau^{-3/2 + \e_3},
\quad \e_3 = \f{1}{10}.
$$

{
  Since $\Qc$ is symmetric and $\Qc(\tau) = R(\tau) \Lam R(\tau)^T$ for $\Lam = \mathrm{diag}(\lam_{\Qc, 1}, .., \lambda_{\Qc, d})$ and some orthogonal matrix $R$, which satisfies 
 $| R(\tau)| \leq C$ for $C$ independent in $\tau$, the above estimates further imply, 
 \beq\label{eq:asymp_Q}
 \Qc= R(\tau) \B( \f{1}{\tau} I_d + O_{in}(a_{\tau})  \B) R(\tau)^T =  \frac{1}{\tau}I_d+O_{in}(a_{\tau}) .
 \eeq
}

\paragraph{Estimate of $\RR$ and blowup rate.}
Recall from \eqref{per_ansatz}, \eqref{eq:Pv} 
\[
\Mc = e^{- \tau / 2}
\RR^{-1}, \quad 
\Qc = C_W^{p-1} \Mc \Mc^T = C_W^{p-1} e^{-\tau/2}  \RR^{-1} ( e^{-\tau/2} \RR^{-1})^T =
M_Q M_Q^T, \quad M_Q : = {C}_U^{(p-1)/2} \RR^{-1}.
\]
Note that $\RR, \Mc,  M_Q$ are upper triangular matrices. Due to $M_{Q, ii}(0) > 0$ and 
the non-degeneracy $ 0 < \det(\Qc) = \det(M_Q)^2 = \prod M_{Q, ii}^2 $ for all $\tau$ from \eqref{eq:boot_EE}, by continuity, we have $M_{Q, ii}(\tau) > 0$, which are the eigenvalues of $M_Q, M_Q^T$. For each real eigenpair $(\lam, v)$ of $M_Q^T$ with $ || v||_{l^2}^2 = 1$, we obtain 
\[
\lam^2 = \lam^2 || v||_{l^2}^2 =  v^T M_Q M_Q^T v = v^T \Qc v 
= \tau^{-1} || v||_{l^2}^2 + O_{in}( a_{\tau} )
= \tau^{-1} +  O_{in}( a_{\tau} ).
\]

Since $M_{Q, ii} > 0$ is a eigenvalue of $M_Q$, we obtain 
\[
M_{Q, ii} =  \tau^{-1/2} ( 1 + \tau O(a_{\tau}))^{1/2}
= \tau^{-1/2} + O( \tau^{1/2} a_{\tau})
= \tau^{-1/2} + O( (\log \tau)^{1/2} / \tau).
\]

Next, we estimate the strictly upper part of $M_Q$: $M_{\Qc}^u$. Taking trace, we get 
\[
 \sum_{i \neq j} M_{Q, ij}^2
= \tr( M_Q M_Q^{T}) - \sum_i M_{Q, ii}^2
 = \tr(\Qc) -\sum_i M_{Q, ii}^2 = d /\tau - d / \tau + O(a_{\tau}) = O(a_{\tau})
\]
which implies $M_{\Qc}^u = O( a_{\tau}^{1/2})$.
Comparing the strictly upper part $\Qc = (M_{Q,d} + M_{Q}^u) 
  (M_{Q,d} + M_{Q}^u)^T$, we get 
\[
  M_{Q}^u  M_{Q, d}  =   \Qc^u -  (M_{Q, u} M_{Q, u}^T )^u
  = O(a_{\tau}), \quad  M_Q^u =  O( a_{\tau}  ) M_{Q, d}^{-1}
  =  O( a_{\tau}  \tau^{1/2} ) .
\]

Therefore, we conclude,
\beq\label{eq:asym_M}
{C}_U^{(p-1)/2} \RR^{-1} = M_Q=\frac{1}{\sqrt{\tau}}I_d+ O_{in}( \tau^{1/2} a_{\tau})
= \frac{1}{\sqrt{\tau}}I_d+ O_{in}( \tau^{-1 +\e_3}) .
\eeq




{
Using \eqref{eq:dyn_scal1}, we define the blowup time as $T = t(\infty)$. Using \eqref{eq:dyn_scal1} for $t(\tau)$, Lemma \ref{lem1} and \eqref{eq:boot_refine} for $c_W$, and \eqref{eq:EE_cw}, \eqref{eq:non_stab}, \eqref{eq:asymp_Q} for $\tr(\Qc), \cE_0$, we obtain 
\begin{equation}\label{tau-t}
t_{\tau} = H^{p-1}, 
\quad (p-1) c_W =  \f{\mu_5}{\tau}  + O_{in}(a_{\tau}),  
\quad |(p-1) c_W| < \f{1}{2}, \quad a_{\tau} = \tau^{-3/2 + \e_3}, 
 \quad \mu_5 = \f{ 2(1 - \be \delta) d c_p}{p-1} .
\end{equation}

Using \eqref{eq:dyn_scal1}, \eqref{per_ansatz}, for $\tau \geq 2, s>0$, we obtain
\[
\f{H^{p-1}(\tau +s) }{ H^{p-1}(\tau) }
= e^{-s} F(\tau, s), 
\quad F(\tau, s) := e^{ (p-1) \int_{\tau}^{\tau+s} c_W(z) d z }. 
\]

Since $|(p-1)c_W(s)| < \min( \f{1}{2}, C  \tau^{-1})$, using $|e^x - 1| \les |x| (e^x + 1),
\pa_s F(\tau, s) = (p-1) c_W(\tau + s) F(\tau, s), F(\tau, 0)=1$, for $0\leq z  \leq s$,  we obtain 
\[
\bal
 |F(\tau , s) -1| & \les s / \tau ( F(\tau, s) + F(\tau,0)) \les e^{s/2} s / \tau, \\
 |c_W(\tau+z) - \mu_5 \tau^{-1} |
& \les | (\tau + z)^{-1} - \tau^{-1}| + a_{\tau}
\les z \tau^{-2} + a_{\tau}, 
\eal
\]
which implies 
\[
\bal
 | c_W(\tau + z) F(\tau, z) - \mu_5 \tau^{-1}|
& \les | c_W(\tau + z) ( F(\tau, z) - 1) + (c_W(\tau+z) - \mu_5 \tau^{-1})|
\les a_{\tau} + e^{z/2} z \tau^{-2} + z\tau^{-2} , \\
 |F(\tau, s) - 1  - s \mu_5 \tau^{-1}| 
&\les s \max_{0\leq z \leq s}  | c_W(\tau + z) F(\tau, z) - \mu_5 \tau^{-1}|
\les e^{s/2} s^2 \tau^{-2} + s \tau^{-3/2 + \e_3}.
\eal
\]

Therefore, integrating $\f{H^{p-1}(\tau +s) }{ H^{p-1}(\tau) }$  for $s$ from $0$ to $\infty$ and using the above estimates, we get 
\beq\label{eq:t_tau}
\bal
\f{T-t(\tau) }{H^{p-1}(\tau)}
 & = \int_{0}^{\infty} \f{H^{p-1}(\tau +s) }{ H^{p-1}(\tau) } d s
= \int_{0}^{\infty} e^{-s} F(\tau, s) d s
=  \int_0^{\infty}  (1 + \mu_5 s \tau^{-1}) e^{-s} d s 
+O_{in}(\tau^{-3/2 + \e_1}) \\
& = 1 + \mu_5 \tau^{-1} +O_{in}(\tau^{-3/2 + \e_1}),
\eal
\eeq
where we use $\int_0^{\infty} s e^{-s} ds=1$. Since $(p-1) c_U = (p-1) (\bar c_U + c_W ) = - 1 + O(\tau^{-1})$, we further obtain 
\beq\label{eq:t_tau2}
\log ( T-t(\tau))
= (1 + O(\tau)^{-1}) \log( H^{p-1})
= (1 + O(\tau)^{-1} ) ( O_{in}( 1 ) + \int_0^{\tau} (p-1) c_U ) 
= - \tau + O_{in}(\log(\tau)) .
\eeq

Combining \eqref{eq:asym_M}, \eqref{eq:t_tau}, \eqref{eq:t_tau2}, we prove \eqref{eq:rate_asym1} and Proposition \ref{prop:asym}.
}


\subsection{Asymptotics of phase and $L^{\infty}$ convergence}\label{sec:asym_phase}

{
In this section, our goal is to prove the following convergence.

\begin{proposition}\label{prop:linf}

Suppose that the initial data $(U, \Th)$ satisfy the assumption in Theorem \ref{thm:stab}. We have 
\beq\label{eq:linf_conv}
| U(z, \tau) e^{ \imath (\Th -  A(\tau))  } - \bar U^{1 + \imath \d} |   \les 
 \min( \la z\ra^{\s +\e/2} \la \tau \ra^{-\e}, \la \tau \ra^{\max(-1, 2 \s)} ).
\eeq
where $A(\tau)$  satisfies the following estimate for some constant $C_{in}$ depending on the initial data
\beq\label{eq:phi0_asym1}
|A -  \bar A| \leq C_{in} ,
\quad 
A(\tau) := \frac{\delta}{p-1}\tau+\Phi(\tau, 0) ,
\quad 
\bar A(\tau):= - \frac{\delta \log (T-t(\tau))}{p-1}-\frac{d\beta(1+\delta^2)\log|\log (T-t(\tau))|}{2 
\flat_*} .
\eeq

\end{proposition}

}

\begin{proof}

We will first estimate $\Phi(0)$ and then prove convergence. 

\paragraph{Estimate of $\Phi(0)$.}

Firstly, we compute $A_{\tau}, \bar A_{\tau}$. Using \eqref{imaginep}, we perform similar computations to the proof of Lemma \ref{lem1} to compute that 
$$
\Phi_\tau(0)=-v\cdot\nabla\Phi(0)+\mathscr{D}_\Theta(0)= O( \cE_0)+(\beta+\delta)\frac{\kappa_2}{\kappa_0}\text{tr}(Q)\,.
$$

Applying \eqref{eq:asymp_Q} and the estimates \eqref{eq:non_stab}, \eqref{eq:EE_ga} for $\cE_0$, we yield
\[
\bal
 & A_{\tau}(\tau)=\frac{\delta}{p-1}-\frac{d(\beta+\delta)}{2 \flat_* \tau}+
 O_{in}(\tau^{-3/2 + \e_1}), \quad 
  \bar A_{\tau}(\tau)=\frac{\delta}{p-1}\frac{t_\tau}{T-t}-\frac{d\beta(1+\delta^2)t_\tau}{2 \flat_* |\log (T-t)|(T-t)}\,.
\eal
\]

{
Using $t_{\tau} = H^{p-1}$  \eqref{tau-t}, \eqref{eq:t_tau}, and \eqref{eq:t_tau2}, we yield 
\[
 \f{t_{\tau}}{ T- t(\tau)}
 = \f{H^{p-1}}{T- t(\tau)}
  = 1 - \f{\mu_5} {\tau} + O_{in}(\tau^{-3/2 + \e_1}),  \quad
   \f{t_{\tau}}{ (T-t)\log( T- t)}
   = \f{1}{\tau} + O_{in}(\log(\tau) \tau^{-2}).
\]
} 
Using the definition of $\mu_5$ \eqref{tau-t}, $c_p$ \eqref{ss-}, we conclude
\begin{align*}
A_{\tau} - \bar A_{\tau} &=  ( - \f{ d( \beta + \d)}{2 \flat_*} + \f{\mu_5 \d}{p-1} 
 + \f{ d \beta(1 + \d^2) }{2 \flat_*} ) \f{1}{\tau} + O_{in}(\tau^{- \f{3}{2} + \e_3})\\
 &= \f{\d}{\tau}( \f{2 (1 - \beta \d) d c_p }{(p-1)^2} - \f{d (1-\be \d)}{2 \flat_*} ) 
 + O_{in}(\tau^{- \f{3}{2} + \e_3}).
\end{align*}
The first term vanishes due to \eqref{ss-} for $c_p$. Since the error term is integrable in $\tau \geq 2$, we conclude the asymptotics of the phase \eqref{eq:phi0_asym1}.




\paragraph{$L^{\infty}$ convergence.}

Recall $ \Th = \bar \Th + \Phi$ \eqref{per_ansatz}. Integrating \eqref{eq:inte:4} with $l=1$, we obtain 
\beq\label{eq:linf_phi}
|\Phi(z) - \Phi(0)| \les \int_0^{1}  |\na \Phi( t z ) | d t 
\les E \int_0^1 \la t z \ra^{-1/2} d t \les E \la z \ra^{1/2}.
\eeq

Using $U = \bar U + W$ \eqref{per_ansatz}, we decompose
\[
J := U e^{ \imath (\Th - \bar \Th - \Phi(0))} - \bar U
 = U e^{ \imath (\Phi - \Phi(0))} - \bar U
 = W e^{ \im (\Phi - \Phi(0))} 
 + \bar U (e^{ \im (\Phi - \Phi(0)) }  - 1) =  I + II.
\]

Appyling \eqref{eq:inte:3} to $I$, \eqref{eq:linf_phi} and $|e^{\im x } - 1| \les \min(|x|, 1)$ to $II$, and $E \les \la \tau \ra^{-1}$ \eqref{eq:non_stab}, we prove
\[
\bal
|J| & \les \la z \ra^{\s + \e/2} E 
+ \la z \ra^{\s} \min( E \la z \ra^{1/2}, 1)
\les \min( \la z \ra^{\s + \e/2} ( E + E^{ \e} )  , E + E^{- 2\s})  \\
& \les 
 \min( \la z\ra^{\s +\e/2} \la \tau \ra^{-\e}, \la \tau \ra^{\max( -1, 2 \s)} ).
\eal
\]
Since $\Th - \bar \Th - \Phi(0) + \d \log \bar U 
= \Th - A(\tau)$ \eqref{ss-} and \eqref{eq:phi0_asym1}, we get $ U^{ \imath \d} J  =   U e^{ \imath (\Th - A(\tau))} - \bar U^{1 + \imath \d} $. Since $| U^{ \imath \d}| = 1$, the above estimate conclude the proof of \eqref{eq:linf_conv}.
\end{proof}

\subsection{Proof of Theorem \ref{thm:blp}}\label{sec:thm_proof}

In this section, we prove Theorem \ref{thm:blp} with the open set $\Oc$ prescribed in Remark 
\ref{rem:blowup}. 

\

\noindent \textbf{Verification of assumptions in Theorem \ref{thm:stab}.} We first choose $\nu < 1$ in \eqref{eq:ass_init2} for Theorem \ref{thm:blp}. Following the proof of Corollary \ref{cor:decay}, we obtain $U_0 \les \la z \ra^{\s + \e/2}$. Using the definitions \eqref{norm:Ek} and \eqref{def:EKFKbar} of $F_K$ and $\bar {\mfr F_K}$, and \eqref{wg:rho}, we get 
\[
 \mr \rho_K = \rho_K U^2 \leq  C \la z \ra^{2 \s + \e}\rho_K \geq C (1 + |z|^{ - d - 2 K } ),\quad
F_K \leq C \|\Phi\|_{\bar {\mfr F_K} }.
\]
From  \eqref{def:EkFk}, \eqref{eneq}, we obtain $E \leq C E_{in}$. 
By choosing $\nu = c E_*$ in \eqref{eq:ass_init2} with $c>0$ sufficiently small, the assumptions \eqref{eq:ass_init2} implies the assumptions \eqref{eq:thm_stab_ass} in Theorem \ref{thm:stab} except for $E_{\Qc} < E_*$. Using the definitions of $\Mc_0, H(0)$ \eqref{eq:init_res} 
$\cR_0, \Qc$ \eqref{eq:Pv},  and $C_W(0) = H(0)$ \eqref{per_ansatz}, we obtain 
\[
\tr(\Qc(0)) = H(0)^{p-1} \tr(\Mc \Mc^T) = 
H(0)^{p-1} \tr( \Mc^T \Mc) =  C u_0(V_0)^{-p} \tr( \na^2 u_0(V_0))
\]
for some absolute constant $C$. Therefore, by further choosing $ c$ small in $\mu = c E_*$, we obtain $|E_{\Qc}| < E_*$ from the last assumption in \eqref{eq:ass_init2}.  
We verify the assumptions in Theorem \ref{thm:stab} and can use the results in Theorem \ref{thm:stab}, and Propositions \ref{prop:asym} and \ref{prop:linf}.

For the time $t$ in Theorem \ref{thm:blp}, we use the change of variables $t = t(\tau)$ \eqref{eq:dyn_scal1}. Then we only need to prove Theorem \ref{thm:blp} in terms of the self-similar time $\tau$.

\

\noindent \textbf{Proof of estimates \eqref{est:asymptoticEKFK}, \eqref{eq:thm_estb}.}
Using Theorem \ref{thm:stab}, we obtain the estimates \eqref{eq:non_stab}, which along with the relation between $(U , \Th), (u, \th)$ and $W = U - \bar U, \Phi = \Th - \bar \Th$ prove \eqref{est:asymptoticEKFK} in Theorem \ref{thm:blp}. 

To obtain \eqref{eq:thm_estb} and \eqref{est:limitmu}, we choose $\mu(t(\tau)) = A(\tau), \hat \mu(t) = A(\tau) - \bar A(\tau) $. Using 
$ \Th - A = \Th - \bar A -  \hat \mu $ and the formula of $\bar A$ \eqref{eq:phi0_asym1}, we obtain  
\[
\bal
J: & = |\log(T-t)|^{\imath  \frac{d\beta (1 + \delta^2)}{2\flat_*}} (T-t)^{\frac{1 + \imath \delta}{p-1}}  \psi( \RR(t) z + \Vc(t) , t )  e^{ -\imath \hat \mu(t) }
= (T-t)^{ \f{1}{p-1} } H^{-1} U(z, \tau) e^{ \imath ( \Th(z, \tau) - \bar A - \hat \mu ) }\\
& = (T-t)^{ \f{1}{p-1} } H^{-1} U(z, \tau) e^{ \imath ( \Th(z, \tau) - A(\tau) ) } .
 \eal
\]

We denote
\[
J_2 =  U(z, \tau) e^{ \imath ( \Th(z, \tau) - A(\tau) ) } .
\]

Using the limits $H(\tau) / (T-t)^{1/(p-1)} \to 1$ and $\tau / |\log(T-t(\tau))| \to 1$ as $\tau \to \infty$ \eqref{eq:rate_asym1},  and the estimate \eqref{eq:linf_conv}, we prove 
\begin{align*}
|J  - \bar U^{1 + i \d}|
&\les | (T-t)^{ \f{1}{p-1} } H^{-1} - 1| \cdot 
 |\bar U^{1 + i \d}| 
 + (T-t)^{ \f{1}{p-1} } H^{-1} |J_2 -\bar U^{1 + i \d} | \\
 &\les C_{in}  (1 + \tau)^{\eta}
 \les C_{in}(1 + |\log( T-t(\tau))| )^{\eta},
\end{align*}
where $\eta = \max(-1, 2 \s)$. 

\ 

\noindent \textbf{Proof of rates \eqref{est:limitHRV}, \eqref{est:limitmu}.} {Next, we show that $V(\tau)$ converges as $\tau \to \infty$ in \eqref{est:limitHRV}. From Lemma \ref{lem1} for $\Vc$, the decay estimates for $\Qc, E$ in Theorem \ref{thm:stab}, and $\RR(\tau) \to 0$ as $\tau \to \infty$ in Proposition \ref{prop:asym}, we obtain 
\[
 | \RR(\tau) \Vc(\tau) | \les (1 + \tau)^{-2}.
\]
Since the upper bound is integrable in $\tau$, using $| \dot V(\tau) | = | \RR(\tau) \Vc(\tau)| $ \eqref{eq:Pv}, we prove that $V(\tau)$ converges as $\tau \to \infty$.} The asymptotics \eqref{est:limitmu} follows from the definition of $\mu$ and \eqref{eq:rate_asym1}.This ends the proof of Theorem \ref{thm:blp}.

\subsection{Proof of Theorem \ref{cor:stab}}\label{sec:cor_pf}

To prove Theorem \ref{cor:stab}, we only need to show that for $\e_0 = \e_0(u_0)$, assumption \eqref{eq:ass_cor} implies that $\td u_0$ is in the open set $\Oc$ in Theorem \ref{thm:blp}. 

\

\noindent  \textbf{Estimates of $\td V_0, \td \Mc_0, \td H_0$.}  Since $V_0$ is the unique maximizer and $\na^2 u_0(V_0) \succ 0$, for $\d_1$ sufficiently small, 
we obtain that $\td u_0$ admits a global non-degenerate maximizer $\td V_0$ close to $V_0$ 
\footnote{Since for any $\d > 0$, there exists a $r > 0$, such that 
$u_0(V_0) > u_0(V_0 + z) + \d, |z| > r$, we obtain $|\td V_0 - V_0| < r$ when $\d_1 < \d / 2$.
}
with $|V_0 - \td V_0| \to 0$ as $||\td u_0 - u_0||_{L^{\infty}} \to 0$. Using embedding \eqref{eq:inte:3} in Proposition \ref{prop:inter}, we obtain 
\[
||u_0 - \td u_0||_{C^2} \les || u_0 - \td u_0 ||_{\cH^K} . 
\]
Denote 
\[
\d_1 := || u_0 - \td u_0 ||_{\cH^K} ,
\quad \d_2 := |V_0 - \td V_0| .
\]
Using continuity and the above embedding, we obtain
\bseq
\beq\label{eq:cor_small1}
\lim_{\e_0 \to 0}  |\td V_0 - V_0| = 
\lim_{\e_0 \to 0}  \d_2 =0, \quad | \td u_0(\td V_0) - u_0(V_0) | \les \d_1 + \d_2, 
\quad | \na^2 \td u_0(\td V_0) - \na^2 u_0(V_0) | \les \d_1  + \d_2. 
\eeq

Upon choosing $\e_0 > 0$ small, we can define the initial {modulation parameters} \eqref{eq:init_res} $\td H_0, \td \Mc_0$ associated with $\td u_0$ and obtain 
\beq\label{eq:cor_small2}
\lim_{\e_0 \to 0} |\td \Mc_0 - \Mc_0 | + | H_0 - \td H_0| = 0.
\eeq

\eseq

We denote by $\td U_0$ the rescaled variables for $\td u_0$: 
\beq\label{eq:cor_pf1}
\td U_0 = \td H_0 \td u_0( \td \Mc_0^{-1} z + \td V_0 ) .
\eeq

\noindent  \textbf{Verification of assumptions.}  We show that $( \td u_0, \td U_0)$ satisfies assumptions \eqref{eq:ass_init2}. The implicit constants can depend on $u_0$.

Firstly, assumptions \eqref{eq:ass_init}, \eqref{eq:ass_init2} for $\td u_0$ except for 
\beq\label{eq:cor_pf2}
\td U_0 \bar U^{-1 - \e_2} > 2 C_b
\eeq
follow from continuity and choosing $\d_1$ small. Condition \eqref{eq:cor_pf2} follows from 
the assumption \eqref{eq:cor_pf2} for $U_0$, \eqref{eq:ass_cor}, the triangle inequality, and choosing $\e_0$ small.

Next, we verify \eqref{def:EKFKbar} for $\td U_0$, i.e. 
 \beq\label{eq:cor_small0}
|| \td U_0 - \bar U_0 ||_{\cE_K} < \nu .
 \eeq
 Using the definition \eqref{eq:cor_pf1}, we decompose $\td U_0 - \bar U$ as follows 
\beq\label{eq:cor_J}
 \td U_0 - \bar U
=  \td H_0 \B(  \td u_0( \td \Mc_0^{-1} z + \td V_0 )
 - u_0( \td \Mc_0^{-1} z + \td V_0 ) \B)  
 + \B( \td H_0 
  u_0( \td \Mc_0^{-1} z + \td V_0 ) - \bar U_0  \B) := J_1 + J_2. 
\eeq

For $J_1$, using a change of variable, \eqref{eq:cor_small1}, the assumption \eqref{eq:ass_cor} for $u_0 - \td u_0$, and the embedding \eqref{eq:inte:3}, we obtain 
\beq\label{eq:cor_J1}
|| J_1 ||_{\cH^K} \les || u_0 - \td u_0 ||_{\cH^K} \les \e_0,
\quad \lim_{\e_0 \to 0} \max_{|z| \leq 1} | \na^i J_1 |  = 0, \quad \mathrm{for \ } i = 0,1, 2,3.
\eeq

Denote $H_1 =\td H_0 H_0^{-1}, \Mc_1 = \Mc_0 \Mc_0^{-1}, V_1 = \Mc_0 ( \td V_0 - V_0 ) $. For $J_2$, using $u_0(z) = H_0^{-1}U_0( \Mc_0 ( z - V_0))$ \eqref{eq:init_res} and a change of variable, we obtain 
\[
\bal
J_2 & = \td H_0 H_0^{-1} U_0( \Mc_0 \Mc_0^{-1} z  + \Mc_0 ( \td V_0 - V_0 ) ) - \bar U_0
= H_1 U_0( \Mc_1 z + V_1) - \bar U \\
& =  H_1 (U_0 - \bar U) (\Mc_1 z + V_1)
+ \B( H_1 \bar U(\Mc_1 z + V_1) - \bar U \B) := J_{21} + J_{22}. 
\eal
\]

From \eqref{eq:cor_small2}, we obtain that $\Mc_1 - 1 =o(1), H_1 - 1 = o(1), V_1 = o(1)$. 
Since $|| U_0 - \bar U_0 ||_{\cE_K} < \nu$, by choosing $\e_0$ small enough, we yield 
\beq\label{eq:cor_J21}
|| J_{21} ||_{\cE_K} < \nu .
\eeq

Using the smoothness of $\bar U$ and the embedding \eqref{eq:inte:3}, we obtain 
\beq\label{eq:cor_J22}
\lim_{\e_0 \to 0} || J_{22} ||_{\cH^K} + \max_{|z| \leq 1, i\leq 3} | \na^i J_{22} |   = 0.
\eeq

Next, we show that  for $f$ with $\na^j f(0) = 0, j \leq 2$ and $i\leq K$, we have 
\beq\label{eq:cor_pf3}
 || \na^i f ||_{\rho_i} \les || f ||_{\cH^K} .
\eeq

Using the definition of \eqref{wg:rho} $\rho_K$, $\cE_K$ \eqref{norm:Ek},  $g_k, \cH^K$ \eqref{norm:Hk}, and embedding \eqref{eq:inte:3}, for $f$ with $\na^j f(0) = 0, j \leq 2$, we have 
\[
\bal
 & ||  \na^i f ||_{\rho_i}
 \les || \na^3 f ||_{L^{\infty}} || \ |x|^{ 3 - i } \one_{|x| \leq 1} ||_{\rho_i}
  +  ||  \na^i f ||_{g_i}
  \les  || f ||_{\cH^K} , \quad i \leq 3,  \\
  & ||  \na^i f ||_{\rho_i} 
\les ||  \na^i f ||_{L^{\infty}} || \one_{|x| \leq 1} ||_{\rho_i}
 +  ||  \na^i f ||_{ g_i} \les || f||_{\cH^K}, \quad i \leq (d+5)/2. 
\eal
\]

For $ \f{d+5}{2} < i \leq K$, $\rho_i$ and $g_i$ are equivalent and \eqref{eq:cor_pf3} follows from \eqref{eq:inte:Hk}. From the definition of $\td U_0$ and $J_i$, for $l\leq 2,i=1,2$, we obtain $\na^l J_{21}(0) = 0, \na^l( J_1 + J_{22})(0) = 0 $. Applying \eqref{eq:cor_pf3}, we obtain 
\[
 || J_1 + J_{22} ||_{\cE_K} \les || J_1 + J_{22} ||_{\cH^K} ,
 \]
 which goes to $0$ as $\e_0 \to 0$. Since the inequality \eqref{eq:cor_J21} is strict, for $\e_0$ small enough, we prove \eqref{eq:cor_small0}. Condition \eqref{def:EKFKbar} follows from a similar argument, we conclude the proof of Theorem \ref{cor:stab}.

\vspace{0.2in}

\paragraph{Acknowledgments:} The research of Y.Wang and T.Y.Hou was in part supported by NSF Grant DMS-2205590 and the Choi Family Gift Fund.
The work of J.Chen was in part supported by the NSF Grant DMS-2408098 and by the Simons Foundation. V.T. Nguyen is supported by the National Science and Technology Council of Taiwan.

\bibliographystyle{abbrv}

\bibliography{VTNbib}

\def\cprime{$'$}
\begin{thebibliography}{10}

\bibitem{AKaps02}
I.~S. Aranson and L.~Kramer.
\newblock The world of the complex {Ginzburg-Landau} equation.
\newblock {\em Rev. Mod. Phys.}, 74:99--143, Feb 2002.

\bibitem{bethuel1994ginzburg}
F.~Bethuel, H.~Brezis, F.~H{\'e}lein, et~al.
\newblock {\em {Ginzburg-Landau} vortices}, volume~13.
\newblock Springer, 1994.

\bibitem{BLNphysd86}
H.~R. Brand, P.~S. Lomdahl, and A.~C. Newell.
\newblock Benjamin-feir turbulence in convective binary fluid mixtures.
\newblock {\em Physica D: Nonlinear Phenomena}, 23(1):345--361, 1986.

\bibitem{BKnon94}
J.~Bricmont and A.~Kupiainen.
\newblock Universality in blow-up for nonlinear heat equations.
\newblock {\em Nonlinearity}, 7(2):539--575, 1994.

\bibitem{buckmaster2023formation}
T.~Buckmaster, S.~Shkoller, and V.~Vicol.
\newblock Formation of point shocks for 3d compressible {Euler}.
\newblock {\em Communications on Pure and Applied Mathematics},
  76(9):2073--2191, 2023.

\bibitem{BRWejam13}
C.~J. Budd, V.~Rottschafer, and J.~F. Williams.
\newblock Asymptotic analysis of a new type of multi-bump, self-similar, blowup
  solutions of the {Ginzburg–Landau} equation.
\newblock {\em European Journal of Applied Mathematics}, 24(1):103–129, 2013.

\bibitem{CDFjee14}
T.~Cazenave, J.-P. Dias, and M.~Figueira.
\newblock Finite-time blowup for a complex {Ginzburg--Landau} equation with
  linear driving.
\newblock {\em Journal of Evolution Equations}, 14(2):403--415, 2014.

\bibitem{CDFWsiam13}
T.~Cazenave, F.~Dickstein, and F.~B. Weissler.
\newblock Finite-time blowup for a complex {Ginzburg--Landau} equation.
\newblock {\em SIAM Journal on Mathematical Analysis}, 45(1):244--266, 2013.

\bibitem{chapman1992macroscopic}
S.~J. Chapman, S.~D. Howison, and J.~R. Ockendon.
\newblock Macroscopic models for superconductivity.
\newblock {\em Siam Review}, 34(4):529--560, 1992.

\bibitem{chen2020singularity}
J.~Chen.
\newblock Singularity formation and global well-posedness for the generalized
  {C}onstantin--{L}ax--{M}ajda equation with dissipation.
\newblock {\em Nonlinearity}, 33(5):2502, 2020.

\bibitem{chen2024Euler}
J.~Chen, G.~Cialdea, S.~Shkoller, and V.~Vicol.
\newblock Vorticity blowup in 2d compressible {Euler} equations.
\newblock {\em arXiv preprint arXiv:2407.06455}, 2024.

\bibitem{chen2019finite2}
J.~Chen and T.~Y. Hou.
\newblock Finite time blowup of {2D} {Boussinesq} and {3D} {Euler} equations
  with ${C}^{1,\alpha}$ velocity and boundary.
\newblock {\em Communications in Mathematical Physics}, 383(3):1559--1667,
  2021.

\bibitem{chen2022stable}
J.~Chen and T.~Y. Hou.
\newblock Stable nearly self-similar blowup of the 2{D} {B}oussinesq and 3{D}
  {E}uler equations with smooth data {I}: Analysis.
\newblock {\em arXiv preprint arXiv:2210.07191}, 2022.

\bibitem{chen2023stable}
J.~Chen and T.~Y. Hou.
\newblock Stable nearly self-similar blowup of the 2{D} {B}oussinesq and 3{D}
  {E}uler equations with smooth data {I}{I}: rigorous numerics.
\newblock {\em arXiv preprint arXiv:2305.05660}, 2023.

\bibitem{chen2024stability}
J.~Chen and T.~Y. Hou.
\newblock On stability and instability of c 1, $\alpha$ singular solutions to
  the 3d euler and 2d boussinesq equations.
\newblock {\em Communications in Mathematical Physics}, 405(5):112, 2024.

\bibitem{chen2021finite}
J.~Chen, T.~Y. Hou, and D.~Huang.
\newblock On the finite time blowup of the {D}e {G}regorio model for the 3{D}
  {E}uler equations.
\newblock {\em Communications on pure and applied mathematics},
  74(6):1282--1350, 2021.

\bibitem{chen2021HL}
J.~Chen, T.~Y. Hou, and D.~Huang.
\newblock Asymptotically self-similar blowup of the {H}ou--{L}uo model for the
  3{D} {E}uler equations.
\newblock {\em Annals of PDE}, 8(2):24, 2022.

\bibitem{CGMNcpam21}
C.~Collot, T.~Ghoul, N.~Masmoudi, and V.~T. Nguyen.
\newblock Refined description and stability for singular solutions of the {2D
  Keller-Segel} system.
\newblock {\em Communications on Pure and Applied Mathematics}, 2021.

\bibitem{DHSjfm74}
A.~Davey, L.~M. Hocking, and K.~Stewartson.
\newblock On the nonlinear evolution of three-dimensional disturbances in plane
  poiseuille flow.
\newblock {\em Journal of Fluid Mechanics}, 63(3):529–536, 1974.

\bibitem{PSbook85}
R.~C. Di~Prima and H.~L. Swinney.
\newblock {\em Instabilities and transition in flow between concentric rotating
  cylinders}, pages 139--180.
\newblock Springer Berlin Heidelberg, Berlin, Heidelberg, 1985.

\bibitem{diprima1971non}
R.~DiPrima, W.~Eckhaus, and L.~Segel.
\newblock Non-linear wave-number interaction in near-critical two-dimensional
  flows.
\newblock {\em Journal of Fluid Mechanics}, 49(4):705--744, 1971.

\bibitem{du1992analysis}
Q.~Du, M.~D. Gunzburger, and J.~S. Peterson.
\newblock Analysis and approximation of the {Ginzburg--Landau} model of
  superconductivity.
\newblock {\em Siam Review}, 34(1):54--81, 1992.

\bibitem{DNZarxiv23}
G.~Duong, N.~Nouaili, and H.~Zaag.
\newblock Flat blow-up solutions for the complex ginzburg landau equation.
\newblock {\em arXiv:2308.02297}, 2023.

\bibitem{DNZmems23}
G.~K. Duong, N.~Nouaili, and H.~Zaag.
\newblock {\em Construction of Blowup Solutions for the Complex
  {Ginzburg-Landau} Equation with Critical Parameters}, volume 285.
\newblock Memmoires of AMS, 2023.

\bibitem{elgindi2021finite}
T.~M. Elgindi.
\newblock Finite-time singularity formation for ${C}^{1,\alpha}$ solutions to
  the incompressible {E}uler equations on ${R}^3$.
\newblock {\em Annals of Mathematics}, 194(3):647--727, 2021.

\bibitem{Fbook15}
G.~Fibich.
\newblock {\em The nonlinear {S}chr\"{o}dinger equation}, volume 192 of {\em
  Applied Mathematical Sciences}.
\newblock Springer, Cham, 2015.
\newblock Singular solutions and optical collapse.

\bibitem{HDphysA98}
G.~Fibich and D.~Levy.
\newblock Self-focusing in the complex ginzburg-landau limit of the critical
  nonlinear schrödinger equation.
\newblock {\em Physics Letters A}, 249(4):286--294, 1998.

\bibitem{GNZihp18}
T.~E. Ghoul, V.~T. Nguyen, and H.~Zaag.
\newblock Construction and stability of blowup solutions for a non-variational
  semilinear parabolic system.
\newblock {\em Annales de l'Institut Henri Poincaré C, Analyse non linéaire},
  35(6):1577 -- 1630, 2018.

\bibitem{GVphysd96}
J.~Ginibre and G.~Velo.
\newblock The {C}auchy problem in local spaces for the complex
  {G}inzburg-{L}andau equation. {I}. {C}ompactness methods.
\newblock {\em Phys. D}, 95(3-4):191--228, 1996.

\bibitem{GVmr97}
J.~Ginibre and G.~Velo.
\newblock The {C}auchy problem in local spaces for the complex
  {G}inzburg-{L}andau equation.
\newblock In {\em Differential equations, asymptotic analysis, and mathematical
  physics ({P}otsdam, 1996)}, volume 100 of {\em Math. Res.}, pages 138--152.
  Akademie Verlag, Berlin, 1997.

\bibitem{GVcmp97}
J.~Ginibre and G.~Velo.
\newblock The {C}auchy problem in local spaces for the complex
  {G}inzburg-{L}andau equation. {II}. {C}ontraction methods.
\newblock {\em Comm. Math. Phys.}, 187(1):45--79, 1997.

\bibitem{ginzburg2009theory}
V.~L. Ginzburg and L.~Landau.
\newblock {\em On the theory of superconductivity}.
\newblock Springer, 2009.

\bibitem{hagan1982spiral}
P.~S. Hagan.
\newblock Spiral waves in reaction-diffusion equations.
\newblock {\em SIAM journal on applied mathematics}, 42(4):762--786, 1982.

\bibitem{HVaihn93}
M.~A. Herrero and J.~J.~L. Vel{\'a}zquez.
\newblock Blow-up behaviour of one-dimensional semilinear parabolic equations.
\newblock {\em Ann. Inst. H. Poincar{\'e} Anal. Non Lin{\'e}aire},
  10(2):131--189, 1993.

\bibitem{hou2023potentially}
T.~Y. Hou.
\newblock Potentially singular behavior of the 3d {Navier--Stokes} equations.
\newblock {\em Foundations of Computational Mathematics}, 23(6):2251--2299,
  2023.

\bibitem{hou2024nearly}
T.~Y. Hou.
\newblock Nearly self-similar blowup of generalized axisymmetric
  {N}avier-{S}tokes and {B}oussinesq equations.
\newblock {\em arXiv preprint arXiv:2405.10916}, 2024.

\bibitem{HNWarXiv24}
T.~Y. Hou, V.~T. Nguyen, and Y.~Wang.
\newblock L2-based stability of blowup with log correction for semilinear heat
  equation.
\newblock {\em arXiv:2404.09410}, 2024.

\bibitem{hou2024blowup}
T.~Y. Hou and Y.~Wang.
\newblock Blowup analysis for a quasi-exact {1D model of 3D Euler and
  Navier--Stokes}.
\newblock {\em Nonlinearity}, 37(3):035001, 2024.

\bibitem{iooss1992time}
G.~Iooss and A.~Mielke.
\newblock Time-periodic {Ginzburg-Landau} equations for one dimensional
  patterns with large wave length.
\newblock {\em Zeitschrift f{\"u}r angewandte Mathematik und Physik ZAMP},
  43:125--138, 1992.

\bibitem{KBSaps88}
P.~Kolodner, D.~Bensimon, and C.~M. Surko.
\newblock Traveling-wave convection in an annulus.
\newblock {\em Phys. Rev. Lett.}, 60:1723--1726, Apr 1988.

\bibitem{KSALphysD95}
P.~{Kolodner}, S.~{Slimani}, N.~{Aubry}, and R.~{Lima}.
\newblock {Characterization of dispersive chaos and related states of
  binary-fluid convection}.
\newblock {\em Physica D Nonlinear Phenomena}, 85(1):165--224, Jan. 1995.

\bibitem{LPSSphysA88}
M.~J. Landman, G.~C. Papanicolaou, C.~Sulem, and P.-L. Sulem.
\newblock Rate of blowup for solutions of the nonlinear {S}chr\"{o}dinger
  equation at critical dimension.
\newblock {\em Phys. Rev. A (3)}, 38(8):3837--3843, 1988.

\bibitem{MZjfa08}
N.~Masmoudi and H.~Zaag.
\newblock Blow-up profile for the complex {G}inzburg-{L}andau equation.
\newblock {\em J. Funct. Anal.}, 255(7):1613--1666, 2008.

\bibitem{mclaughlin1986focusing}
D.~McLaughlin, G.~Papanicolaou, C.~Sulem, and P.~Sulem.
\newblock Focusing singularity of the cubic {S}chr{\"o}dinger equation.
\newblock {\em Physical Review A}, 34(2):1200, 1986.

\bibitem{MZdm97}
F.~Merle and H.~Zaag.
\newblock Stability of the blow-up profile for equations of the type
  {$u_t=\Delta u+\vert u\vert ^{p-1}u$}.
\newblock {\em Duke Math. J.}, 86(1):143--195, 1997.

\bibitem{Mmono02}
A.~Mielke.
\newblock The {Ginzburg-Landau} equation in its role as a modulation equation.
\newblock In {\em Handbook of dynamical systems}, volume~2, pages 759--834.
  Elsevier, 2002.

\bibitem{MHKnonl98}
A.~Mielke, P.~Holmes, and J.~N. Kutz.
\newblock Global existence and uniqueness for an optical fibre laser model.
\newblock {\em Nonlinearity}, 11(6):1489, nov 1998.

\bibitem{sverak1996leray}
J.~Necas, M.~Ruzicka, and V.~Sverak.
\newblock On {L}eray's self-similar solutions of the {N}avier-{S}tokes
  equations.
\newblock {\em Acta Math.}, 176(2):283--294, 1996.

\bibitem{newell1971review}
A.~Newell and J.~Whitehead.
\newblock Review of the finite bandwidth concept.
\newblock In {\em Instability of Continuous Systems: Symposium Herrenalb
  (Germany) September 8--12, 1969}, pages 284--289. Springer, 1971.

\bibitem{NWjfm69}
A.~C. Newell and J.~A. Whitehead.
\newblock Finite bandwidth, finite amplitude convection.
\newblock {\em Journal of Fluid Mechanics}, 38(2):279–303, 1969.

\bibitem{NZarma18}
N.~Nouaili and H.~Zaag.
\newblock Construction of a blow-up solution for the {C}omplex
  {G}inzburg-{L}andau equation in some critical case.
\newblock {\em Arch. Ration. Mech. Anal., to appear}, 2018.

\bibitem{RScpam01}
P.~Plechač and V.~Šverák.
\newblock On self-similar singular solutions of the complex {Ginzburg-Landau}
  equation.
\newblock {\em Communications on Pure and Applied Mathematics},
  54(10):1215--1242, 2001.

\bibitem{QSbook07}
P.~Quittner and P.~Souplet.
\newblock {\em Superlinear parabolic problems}.
\newblock Birkh{\"a}user Advanced Texts: Basler Lehrb{\"u}cher. [Birkh{\"a}user
  Advanced Texts: Basel Textbooks]. Birkh{\"a}user Verlag, Basel, 2007.
\newblock Blow-up, global existence and steady states.

\bibitem{scheel1998bifurcation}
A.~Scheel.
\newblock Bifurcation to spiral waves in reaction-diffusion systems.
\newblock {\em SIAM journal on mathematical analysis}, 29(6):1399--1418, 1998.

\bibitem{schneider1998hopf}
G.~Schneider.
\newblock Hopf bifurcation in spatially extended reaction—diffusion systems.
\newblock {\em Journal of Nonlinear Science}, 8:17--41, 1998.

\bibitem{SSjfm71}
K.~Stewartson and J.~T. Stuart.
\newblock A non-linear instability theory for a wave system in plane poiseuille
  flow.
\newblock {\em Journal of Fluid Mechanics}, 48(3):529–545, 1971.

\bibitem{tsai1998leray}
T.-P. Tsai.
\newblock On {L}eray's self-similar solutions of the {N}avier-{S}tokes
  equations satisfying local energy estimates.
\newblock {\em Arch. Rational Mech. Anal.}, 143(1):29--51, 1998.

\bibitem{Taps93}
S.~K. Turitsyn.
\newblock Nonstable solitons and sharp criteria for wave collapse.
\newblock {\em Phys. Rev. E}, 47:R13--R16, Jan 1993.

\bibitem{wang2022self}
Y.~Wang, C.-Y. Lai, J.~G{\'o}mez-Serrano, and T.~Buckmaster.
\newblock Self-similar blow-up profile for the {B}oussinesq equations via a
  physics-informed neural network.
\newblock {\em arXiv:2201.06780}, 2022.

\bibitem{ZAAihn98}
H.~Zaag.
\newblock Blow-up results for vector-valued nonlinear heat equations with no
  gradient structure.
\newblock {\em Ann. Inst. H. Poincar{\'e} Anal. Non Lin{\'e}aire},
  15(5):581--622, 1998.

\end{thebibliography}

\end{document}